\documentclass[reqno]{amsart}
\usepackage{mathrsfs}
\usepackage{color}
\usepackage{amsmath}
\usepackage{amsfonts}
\usepackage{amssymb}
\usepackage{graphicx}
\usepackage{hyperref}

 \newtheorem{Theorem}{Theorem}[section]
 \newtheorem{Corollary}[Theorem]{Corollary}
 \newtheorem{Lemma}[Theorem]{Lemma}
 \newtheorem{Proposition}[Theorem]{Proposition}
 \newtheorem{Question}[Theorem]{Question}

 \newtheorem{Remark}[Theorem]{Remark}

 \numberwithin{equation}{section}

\begin{document}

\title[Minimal $L^2$ integrals]
 {Concavity of minimal $L^2$ integrals related to multipler ideal
sheaves}

\author{Qi'an Guan}
\address{Qi'an Guan: School of
Mathematical Sciences, Peking University, Beijing 100871, China.}
\email{guanqian@math.pku.edu.cn}

\author{Zhitong Mi}
\address{Zhitong Mi: School of
Mathematical Sciences, Peking University, Beijing 100871, China.}
\email{zhitongmi@pku.edu.cn}

\thanks{}

\subjclass[2020]{14F18, 32D15, 32Q15, 32U05}

\keywords{strong openness conjecture, multiplier ideal sheaf, plurisubharmonic function, sublevel set}

\date{\today}

\dedicatory{}

\commby{}


\begin{abstract}
In this note, we present the concavity of the minimal $L^2$
integrals related to multiplier ideals sheaves on Stein manifolds.
As applications, we obtain a necessary condition for the concavity degenerating to linearity, a characterization for 1-dimensional case,
and a characterization for the equality in 1-dimensional optimal $L^{2}$ extension problem to hold.
\end{abstract}

\maketitle

\section{Introduction}
The multiplier ideal sheaf associated to plurisubharmonic functions plays an important role in complex geometry and algebraic geometry (see e.g. \cite{Tian},\cite{Nadel},\cite{Siu96},\cite{DEL},\\
\cite{DK01},\cite{DemaillySoc},
\cite{DP03},\cite{Lazarsfeld},\cite{Siu05},\cite{Siu09},\cite{DemaillyAG},
\cite{Guenancia}). We recall the definition of the multiplier ideal sheaves as follows.
\par
\emph{Let $\varphi$ be a plurisubharmonic function (see \cite{Demaillybook})on a complex manifold. It is known that the multiplier ideal sheaf $\mathcal{I}(\varphi)$ was defined as the sheaf of germs of holomorphic functions $f$ such that $|f|^2e^{-\varphi}$ is locally integrable (see \cite{DemaillyAG})}.
\par
In \cite{DemaillySoc}, Demailly posed the so-called strong openness conjecture on multiplier ideal sheaves (SOC for short) i.e. $\mathcal{I}(\varphi)=\mathcal{I}_+(\varphi):=\mathop{\cup} \limits_{\epsilon>0}\mathcal{I}((1+\epsilon)\varphi)$. When $\mathcal{I}(\varphi)=\mathcal{O}$, SOC degenerates to the openness conjecture posed by Demailly-Koll\'ar \cite{DK01}.
\par
The dimension two case of OC was proved by Favre-Jonsson \cite{FavreJonsson}, and the dimension two case of SOC was proved by Jonsson-Musta\c{t}$\breve{a}$
\cite{JonssonMustata}. OC was proved by Berndtsson \cite{Berndtsson2}. SOC was proved by Guan-Zhou \cite{GZSOC}, see also \cite{lempert} and \cite{Hiep}.\par
In \cite{Berndtsson1}, Berndtsson established an effectiveness result of OC. Simulated by Berndtsson's effectiveness result of OC, continuing the solution of SOC \cite{GZSOC}, Guan-Zhou \cite{GZeff} establish an effectiveness result of SOC.\par
Recently, Guan \cite{G16} established a sharp version of the effectiveness result of SOC by considering the minimal
$L^2$ integrals defined on the sub-level set of plurisubharmonic function,
and established the concavity of the minimal
$L^2$ integrals on pseudoconvex domain in $\mathbb{C}^n$.\par
In the present note, we generalize the above concavity property.

\subsection {A general concavity property}
Let $X$ be an $n$-dimensional Stein manifold, and
let $K_X$ be the canonical line bundle on X. Let $dV_X$ be a continuous volume form with no zero point on $X$. We define $|g|^2=i^{n^{2}}\frac{g\wedge \bar{g}}{dV_X}$ for any holomorphic $(n,0)$ form $g$. Let $\psi<-T$ be a plurisubharmonic function on X, and let $\varphi$ be a Lebesgue measurable function on X, such that
$\varphi+\psi$ is a plurisubharmonic function on X, where $T \in (-\infty,+\infty)$.\par
We call a positive smooth function $c$ on $(T,+\infty)$ in class $\mathcal{G}_T$
if the following three statements hold:\par
(1) $\int^{+\infty}_T c(t)e^{-t}dt< +\infty;$\par
(2) $c(t)e^{-t}$ is decreasing with respect to t;\par
(3) for any compact subset $K \subset X$, $e^{-\varphi}c(-\psi)$ has a positive
lower bound on $K$. \par
Especially, if $\varphi\equiv 0$, then $(3)$ is equivalent to $\liminf_{t \to +\infty}c(t)>0$.
\par
Let $Z_0$ be a subset of $\{\psi=-\infty\}$ such that $Z_0 \cap
Supp(\mathcal{O}/\mathcal{I}(\varphi+\psi))\neq \emptyset$. Let $U \supset Z_0$ be
an open subset of X and let $f$ be a holomorphic $(n,0)$ form on $U$.
Let $\mathcal{F} \supset \mathcal{I}(\varphi+\psi)|_U$ be a coherent subsheaf of
$\mathcal{O}$ on $U$.\par
Denote
\begin{equation}
\begin{split}
\inf\{\int_{ \{ \psi<-t\}}|\tilde{f}|^2e^{-\varphi}c(-\psi)dV_X: &\tilde{f}\in
H^0(\{\psi<-t\},\mathcal{O} (K_X)  ) \\
\& & \exists \;open\; set\; U' \;s.t\; Z_0 \subset U' \subset U \\
&and \; (\tilde{f}-f)\in
H^0(\{\psi<-t\} \cap U' ,\mathcal{O} (K_X) \otimes \mathcal{F}) \}
\end{split}
\end{equation}
by $H(t;c)$ ($H(t)$ for short without misunderstanding), where $c\in
\mathcal{G_T}$.\par
If there is no holomorphic $(n,0)$ form $\tilde{f}$ on $\{\psi< -t\}$ satisfying
$(\tilde{f}-f)\in
H^0(\{\psi<-t\} \cap U' ,\mathcal{O} (K_X) \otimes \mathcal{F})$ for some open subset $U'$ which
satisfies $Z_0 \subset U' \subset U$ , then we set $H(t)=-\infty$.\par
In the present note, we obtain the following concavity of $H(t)$.
\begin{Theorem}
$H(h^{-1}(r)) $ is concave with respect to $r \in (0,\int^{+\infty}_T
c(t)e^{-t}dt]$, where $h(t)=\int^{+\infty}_t c(t_1)e^{-t_1}dt_1, t\in [T,+\infty
)$.
\label{maintheorem}
\end{Theorem}
Especially, when $c(t)\equiv 1$ and $T=0$, Theorem \ref{maintheorem} degenerates to the
concavity of the minimal $L^2$ integrals related to multiplier ideals in \cite{G16} (Proposition 4.1 in \cite{G16}).
\par
Theorem $1.1$ implies the following.
\begin{Corollary}
For any $c \in \mathcal{G}_T$,the following three statements are equivalent \par
(1) $H(h^{-1}(r)) $ is linear with respect to $r \in (0,\int^{+\infty}_T
c(t)e^{-t}dt]$, i.e.,
\begin{equation}
H(t)=\frac{H(T)}{\int^{+\infty}_T c(t)e^{-t}dt}\int^{+\infty}_t c(t_1)e^{-t_1}dt_1
\end{equation}
holds for any $t\in [T,+\infty)$;\par
(2) $\frac{H(h^{-1}(r_0))}{r_0}\le \frac{H(T)}{\int^{+\infty}_T c(t)e^{-t}dt}$
holds for some $r_0 \in (0,\int^{+\infty}_T c(t)e^{-t}dt)$, i.e.,
\begin{equation}
\frac{H(t_0)}{\int^{+\infty}_{t_0} c(t_1)e^{-t_1}dt_1}\le
\frac{H(T)}{\int^{+\infty}_T c(t)e^{-t}dt}
\end{equation}
holds for some $t_0 \in (T,+\infty)$;\par
(3)$\lim_{r \rightarrow 0+0}\frac{H(h^{-1}(r))}{r}\le
\frac{H(T)}{\int^{+\infty}_T c(t)e^{-t}dt}$ holds, i.e.,
\begin{equation}
\lim_{t \rightarrow +\infty}\frac{H(t)}{\int^{+\infty}_{t}
c(t_1)e^{-t_1}dt_1}\le \frac{H(T)}{\int^{+\infty}_T c(t)e^{-t}dt}
\end{equation}
holds.
\label{corollary1.2}
\end{Corollary}
\subsection{Applications}
Following the notations and assumptions in Section $1.1$, we present some applications of Theorem \ref{maintheorem}.
\subsubsection{linear case: necessary condition}
\begin{Theorem}
Assume that $H(T;c)<+\infty$.
If $H(h^{-1}(r);c)$ is linear with respect to $r \in
(0,\int^{+\infty}_T c(t)e^{-t}dt]$, then there exists a holomorphic $(n,0)$ form
$F$ on $X$ such that $(F-f)\in H^0(U',K_M \otimes \mathcal{F})$, where $U'$ is an open subset of $X$ satisfies $Z_0 \subset U' \subset U$, and
\begin{equation}
\int_{\{\psi<-t\}} c(-\psi)|F|^2e^{-\varphi}dV_X=H(t)=H(T)\frac{\int^{+\infty}_t
c(t_1)e^{-t_1}dt_1}{\int^{+\infty}_T c(t_1)e^{-t_1}dt_1}
\end{equation}
holds for any $t\in [T,+\infty)$.
\label{theorem1.3}
\end{Theorem}
When $c(t)\equiv1$, $\varphi$ is a smooth plurisubharmonic function on $X$, and $\{\psi=-\infty\}$ is a closed subset of $X$, Xu \cite{xuwang} also get the Theorem \ref{theorem1.3} independently.
\par
We now consider the linearity of $H(h_c^{-1}(r);c)$ for various $c \in \mathcal{G}_T$ and $c\in C^{\infty}[T,+\infty)$, where $h_c(t)=\int_t^{+\infty}c(t_1)e^{-t_1}dt_1$. We have the following result.

\begin{Corollary}Let $c \in \mathcal{G}_T$ and $c\in C^{\infty}[T,+\infty)$. If $H(T;c)<+\infty$ and $H(h^{-1}(r);c)$ is linear with respect to $r \in
(0,\int^{+\infty}_T c(t)e^{-t}dt]$. Let $F$ be the holomorphic $(n,0)$ form on $X$ such that $\int_{\{\psi<-t\}} c(-\psi)|F|^2e^{-\varphi}dV_X=H(t;c)$ for any $t\ge T$.
Then for any other $\tilde{c}\in \mathcal{G}_T$ and $\tilde{c}\in C^{\infty}[T,+\infty)$, which satisfies $H(T;\tilde{c})<+\infty$ and $(\log\tilde{c}(t))'\ge (\log c(t))'$,  we have
\begin{equation}
\begin{split}
\int_{\{\psi<-t\}} \tilde{c}(-\psi)|F|^2e^{-\varphi}dV_X=H(t;\tilde{c})=&
\frac{H(T;\tilde{c})}{\int^{+\infty}_T \tilde{c}(t_1)e^{-t_1}dt_1}\int^{+\infty}_t
\tilde{c}(t_1)e^{-t_1}dt_1\\
=&k\int^{+\infty}_t
\tilde{c}(t_1)e^{-t_1}dt_1
\end{split}
\end{equation}
holds for any $t\in [T,+\infty)$, where $k=\frac{H(T;c)}{\int^{+\infty}_T c(t_1)e^{-t_1}dt_1}$.
\label{corollary1.4}
\end{Corollary}
We now consider the relation between the linearity of $H(t)$ and the weight function $\varphi$. Let $c(t)\in \mathcal{G}_T$. Denote
\begin{equation}\nonumber
\begin{split}
\inf\{\int_{ \{ \psi<-t\}}c(-\psi)|\tilde{f}|^2e^{-\varphi}dV_X: &\tilde{f}\in
H^0(\{\psi<-t\},\mathcal{O} (K_X)  ) \\
\& & \exists \;open\; set\; U' \;s.t\; Z_0 \subset U' \subset U \\
&and \; (\tilde{f}-f)\in
H^0(\{\psi<-t\} \cap U' ,\mathcal{O} (K_X) \otimes \mathcal{F}) \}
\end{split}
\end{equation}
by $H(t;\varphi)$. We have the following result.
\begin{Corollary}If there exists a Lebesgue measurable function $\tilde{\varphi}$ such that $\psi+\tilde{\varphi}$ is a plurisubharmonic function on $X$ and satisfies
\par
(1) There exists constant $C_1,C_2>T$ such that
$$\tilde{\varphi}|_{\{\psi< -C_1\}\cup\{\psi\ge -C_2\}}=\varphi|_{\{\psi< -C_1\}\cup\{\psi\ge -C_2\}}.$$
\par
(2) $\tilde{\varphi}\ge\varphi$ on $X$ and $\tilde{\varphi}>\varphi$ on a open set $U$ of $X$.
\par
(3) $\tilde{\varphi}-\varphi$ is bounded on $X$.
\par
Then $H(h^{-1}(r);\varphi)$ can not be linear with respect to $r\in (0,\int_T^{+\infty}c(t)e^{-t}dt)]$.
\label{corollary1.5}
\end{Corollary}

If $\varphi+\psi$ is a plurisubharmonic function on X and $\varphi+\psi$ is strictly plurisubharmonic at $z_0 \in X$. Denote
\begin{equation}\nonumber
\begin{split}
\inf\{\int_{ \{ \psi<-t\}}&|\tilde{f}|^2e^{-\varphi}c(-\psi)dV_X: \tilde{f}\in
H^0(\{\psi<-t\},\mathcal{O} (K_X)  ), \\
\& &\  \exists \;open\; set\; U' \;s.t\; Z_0 \subset U' \subset U\ and \\
& \; (\tilde{f}-f)\in
H^0(\{\psi<-t\} \cap U' ,\mathcal{O} (K_X) \otimes \mathcal{F}) \}
\end{split}
\end{equation}
by $H(t;\varphi)$. It follows from Corollary \ref{corollary1.5} that we have
\begin{Corollary}  $H(h^{-1}(r);\varphi)$ can not be linear with respect to $r\in (0,\int_{T}^{+\infty}c(t)e^{-t}dt]$.
\label{corollary1.6}
\end{Corollary}

\subsubsection{Equality in optimal $L^2$ extension problem: necessary condition}
Following Guan-Zhou \cite{GZ14Ann}, for a suitable pair $(X,Y)$, where $Y$ is a closed complex subvariety of a complex manifold $X$, given a holomorphic function $f$ (or a holomorphic section of some vector bundle) on $Y$ satisfying suitable $L^2$
conditions, we can find an $L^2$ holomorphic extension $F$ on $X$ together with an optimal $L^2$ estimate for $F$ on $X$.
\par
For example, let $X$ be a Stein manifold, and let $Y$ be a $n-k$
dimensional complex submanifold of $X$. Let $\psi<0$ be a plurisubharmonic function on $X$,
such that for any point $Y$ of $X$ , $\psi-2klog|\omega''|$ is bounded
near $x$, where $\omega=(\omega',\omega'')$ is the local coordinate near $x$ such
that $\{\omega''=0\}=Y$ near $x$.\\
Following \cite{Ohsawa5} (see also \cite{GZ14Ann}), one can define the measure $dV_X[\psi]$ on $Y$
\begin{equation}
\int_Y fdV_X[\psi]=\limsup_{t \rightarrow \infty}
\frac{2(n-k)}{\sigma_{2n-2k-1}}\int_X|f|^2e^{-\varphi}\mathbb{I}_{\{-1-t<\psi<-t\}}dV_X
\label{measure1.8}
\end{equation}
for any nonnegative continuous function $f$ with $suppf \subset\subset X$, where
$\mathbb{I}_{\{-1-t<\psi<-t\}}$ is the characteristic function of the set
$\{-1-t<\psi<-t\}$. Here denote by $\sigma_m$ the volume of the unit sphere in $R^{m+1}$.Let
$\varphi$ be a locally upperbounded Lebesgue measurable function on X, such that
$\varphi+\psi$ is plurisubharmonic on X.
\par
Let $c(t)\in \mathcal{G}_T$. It was established in \cite{GZ14Ann} (see also \cite{GZsci}) that for any holomorphic $(n,0)$ form $f$ on $Y$, such
that
\begin{equation}
\int_Y |f|^2e^{-\varphi}dV_X[\psi]<+\infty
\end{equation}
there exists a holomorphic $(n,0)$ form $F$ on $X$ such that $F|_Y=f$ and
\begin{equation}
\int_X c(-\psi)|F|^2e^{-\varphi}dV_X\leq (\int_0^{+\infty}c(t)e^{-t}dt) \frac{\pi^k}{k!}\int_Y |f|^2e^{-\varphi}dV_X[\psi]<+\infty
\end{equation}
\par
To simplify our notation, denote that $\|f\|_{L^2}:=(\int_0^{+\infty}c(t)e^{-t}dt) \frac{\pi^k}{k!}\int_Y |f|^2e^{-\varphi}dV_X[\psi]$ and $\|F\|_{L^2}:=\int_X c(-\psi)|F|^2e^{-\varphi}dV_X$.
We will consider the following question
\begin{Question}\textbf{(Equality in optimal $L^2$ extension problem)}
Under which (necessary or sufficient) condition, equality $\|f\|_{L^2}=\inf\{\|F\|_{L^2}: F$ is a holomorphic extension of $f$ from $Y$ to $X\}$ holds?
Moreover, can one obtain the characterization (necessary and sufficient condition)?
\end{Question}
Theorem \ref{theorem1.3} shows that the following necessary condition for the equality $\|f\|_{L^2}=\inf\{\|F\|_{L^2}\}$ to hold.\\
\begin{Theorem}
\label{corollary1.7}
Let $f$ be holomorphic $(n,0)$ form
on $Y$, such that
\begin{equation}
\int_Y |f|^2e^{-\varphi}dV_X[\psi]<+\infty
\end{equation}
If for any holomorphic $(n,0)$ form $\tilde{F}$ on $X$, which is a holomorphic extension of $f$ from $Y$ to $X$ i.e. $\tilde{F}|_Y=f$, then $\tilde{F}$ satisfies
\begin{equation}
\int_X c(-\psi)|\tilde{F}|^2e^{-\varphi}dV_X\ge(\int_0^{+\infty}c(t)e^{-t}dt)\frac{\pi^k}{k!}\int_Y |f|^2e^{-\varphi}dV_X[\psi]
\label{1.22}
\end{equation}
and there exists a holomorphic $(n,0)$ form $F$ on $X$ such that
\begin{equation}
\int_X c(-\psi)|F|^2e^{-\varphi}dV_X=(\int_0^{+\infty}c(t)e^{-t}dt)\frac{\pi^k}{k!}\int_Y |f|^2e^{-\varphi}dV_X[\psi]
\label{1.23}
\end{equation}
Then for any $t \geq 0$, there exists a unique holomorphic $(n,0)$ form $F_t$ on
$\{\psi<-t\}$ such that $F_t|_Y=f$ and
\begin{equation}
\int_{\{\psi<-t\}} c(-\psi)|F_t|^2e^{-\varphi}dV_X=(\int_t^{+\infty}c(t_1)e^{-t_1}dt_1)\frac{\pi^k}{k!}\int_X |f|^2e^{-\varphi}dV_M[\psi]
\end{equation}
In fact, $F_t=F|_{\{\psi<-t\}}$.
\end{Theorem}

\begin{Remark} It follows from Corollary \ref{corollary1.4} and Theorem \ref{corollary1.7} that for any $\tilde{c}\in \mathcal{G}_T$ which satisfies $(\log\tilde{c}(t))'\ge (\log c(t))'$, the holomorphic $(n,0)$ form $F$ satisfies
$$H(t;\tilde{c})=\int_{\{\psi<-t\}} \tilde{c}(-\psi)|F|^2e^{-\varphi}dV_X=(\int_t^{+\infty}\tilde{c}(t_1)e^{-t_1}dt_1)\frac{\pi^k}{k!}\int_Y |f|^2e^{-\varphi}dV_X[\psi].$$
\end{Remark}

Recall that the pluricomplex Green function $G(z,\omega)$ on a pseudoconvex
domain $D\subset \mathbb{C}^n$ satisfies $G_D(z,\omega)<0$ and
$G_D(z,\omega_0)=\log|z-\omega_0|+O(1)$ near $\omega_0 \in D$ (see \cite{Demailly87}).
Let $\psi(z)=2nG_D(z,0)$, $f\equiv 1$ and $\mathcal{F}=(z_1,\cdots,z_n)$, and let
$\varphi \equiv 0$ and $c(t) \equiv 1$. Let $D_t=\{\psi(z)<t\}$. Note that
$K_{D_t}(0,0)=\frac{1}{H(t)}$, then the combination of Corollary 1.2 and Theorem
1.3 implies the following restriction property of Bergman kernels.
\begin{Corollary}
The following two statements are equivalent
\par
(1) $\frac{K_{D_{t_0}}(0,0)}{K_{D}(0,0)}= e^{t_0}$ holds for some $t_0 \in
(0,+\infty)$;
\par
(2) $\frac{K_{D_t}(z,0)}{K_{D}(z,0)}=e^t$ holds for any $t\in(0,+\infty)$ and any $z \in D_t$.
\label{corollary1.8}
\end{Corollary}

\subsubsection{Characterizations for 1-dimensional case}

In this section,
we present a characterization for the concavity degenerating to linearity for 1-dimensional case,
and a characterization for the equality in 1-dimensional optimal $L^{2}$ extension problem to hold.

Let $X$ be an open Riemann Surface which admits a nontrivial Green function $G_{X}(z,w)$.
\par
Let $\psi=kG_{X}(z,z_0)$, where $k\ge 2$ is a real number and $z_0$ is a point of $X$.
\par
Let $U$ be a open neighborhood of $z_0$ in $X$ and $f$ be a holomorphic $(1,0)$ form on U. Let $\varphi$ be a subharmonic function on $X$. Let $c(t)\in C^{\infty}[0,+\infty)$ and $c(t)\in \mathcal{G}_0$. Denote
\begin{equation}
\begin{split}
H(t;c,2\varphi):=\inf\{\int_{ \{ \psi<-t\}}&c(-\psi)|\tilde{F}|^2e^{-2\varphi}dV_X: \tilde{F}\in
H^0(\{\psi<-t\},\mathcal{O} (K_X)  ), \\
\& &\  \exists \;open\; set\; U' \;s.t\; Z_0 \subset U' \subset U\ and \\
& \; (\tilde{F}-f)\in
H^0(\{\psi<-t\} \cap U' ,\mathcal{O} (K_X) \otimes \mathcal{I}(\psi+2\varphi)|_U) \},
\label{def1.15}
\end{split}
\end{equation}
\par
We have the following necessary conditions for the minimal $L^2$ integrals $H(h^{-1}(r);c,2\varphi)$ to be linear with respect to $r\in(0,\int_0^{+\infty}c(t_1)e^{-t_1}dt_1]$.

\begin{Theorem}\label{corollary1.9}Assume that $0<H(0;c,2\varphi)<+\infty$.
If $H(h^{-1}(r);c,2\varphi)$ is linear with respect to $r\in(0,\int_0^{+\infty}c(t_1)e^{-t_1}dt_1]$,
then $\varphi=\log|f_{\varphi}|+v$, where $f_{\varphi}$ is a holomorphic function on $X$ and $u$ is a harmonic function on $X$.
\end{Theorem}
Now, in the definition of $H(t;c,2\varphi)$, we take $\psi=2G_X(z,z_0)$, where $z_0\in X$ is a point.
\par
Let $(V_{z_0},w)$ be a local coordinate neighborhood of $z_0$ satisfying $w(z_0)=o$ and $G_X(z,z_0)=\log|w|+u(w)$ on $V_{z_0}$, where $u(w)$ is a harmonic function on $V_{z_0}$.
Let $U=V_{z_0}$. Let $f$ be a holomorphic $(n,0)$ form on $X$. Let $\varphi$ be a subharmonic function on $X$.
\par
Let $c_{\beta}(z)$ be the logarithmic capacity which is locally defined by
$$c_{\beta}(z_0):=\exp(\lim\limits_{z\to z_0} G_X(z,z_0)-log|w(z)|)$$
\par
To state our result, we introduce the following notations (see \cite{FarkasKra}).
\par
Let $p:\Delta\to X$ be the universal covering from unit disc $\Delta$ to $X$.
We call the holomorphic function $f$ (resp. holomorphic $(1,0)$ form $F$) on $\Delta$ is a multiplicative function (resp. multiplicative differential (Prym differential)) if there is a character $\chi$, where $\chi\in Hom(\pi_1(X),C^*)$ and $|\chi|=1$, such that $g^*f=\chi(g)f$ (resp. $g^*F=\chi(g)F$) for every $g\in \pi_1(X)$ which naturally acts on the universal covering of $X$. Denote the set of such kinds of $f$ (resp. F) by $\mathcal{O}^{\chi}(X)$ (resp. $\Gamma^{\chi}(X)$).
\par
As $p$ is a universal covering, then for any harmonic function $h$ on $X$, there exists a $\chi_h$ and a multiplicative function $f_h \in \mathcal{O}^{\chi_h}(X)$, such that $|f_h|=p^*e^{h}$. And if $g \in \mathcal{O} (X)$ and $g$ has no zero points on $X$. Then $\log|g|$ is harmonic function on $X$ and we know $\chi_h=\chi_{h+\log|g|}$ (for the proof, see Appendix \ref{Appendix4.3}).
\par
For Green function $G_X(\cdot,z_0)$, one can find a $\chi_{z_0}$ and a multiplicative function $f_{z_0} \in \mathcal{O}^{\chi_{z_0}}(X)$, such that $|f_{z_0}|=p^*e^{G_X(\cdot,z_0)}$.
\par

Using Theorem \ref{corollary1.9} and the solution of extend Suita conjecture in \cite{GZ14Ann}(see Theorem \ref{theorem1.10}),
we have the following characterization for $H(h^{-1}(r);c,2\varphi)$ to be linear.
\begin{Theorem}\label{corollary1.11}Assume that $0<H(0;c,2\varphi)<+\infty$. The minimal $L^2$ integral function
$H(h^{-1}(r);c,2\varphi)$ is linear with respect to $r$ if and only if the following statements hold:
\par
(1) $\varphi=\log|f_{\varphi}|+v$, where $f_{\varphi}$ is a holomorphic function on $X$ and $v$ is a harmonic function on $X$.
\par
(2) $\chi_{-v}=\chi_{z_0}$.
\end{Theorem}
The representation $\varphi=\log|f_{\varphi}|+v$ is not unique. If $f_1\in\mathcal{O}(X)$ and $f_1$ has no zero points on $X$.
Then $\varphi=\log|\frac{f_{\varphi}}{f_1}|+(\log|f_1|+v)$ is another representation of $\varphi$.
Since $\chi_{-v}=\chi_{-v-\log|f_1|}$ (see Lemma \ref{Aproposition4.1}),
we know the condition (2) in Theorem \ref{corollary1.11} is free for the choice of the specific representation of $\varphi$.
\par
Let $f\equiv dw$ on $V_{z_0}$  under the local coordinate $w$ on $V_{z_0}$. We also assume that $\varphi(z_0)>-\infty$.
\par
Now we illustrate the relation between  $H(h^{-1}(r);c,2\varphi)$ is linear with respect to $r$ and the equality $\|f\|_{L^2}=\inf\{\|F\|_{L^2}\}$ holds,
where $$\|f\|_{L^2}=(\int_0^{+\infty}c(t)e^{-t}dt) \frac{\pi^k}{k!}\int_Y |f|^2e^{-2\varphi}dV_X[\psi]$$ and $\|F\|_{L^2}=\int_X c(-\psi)|F|^2e^{-2\varphi}dV_X$.
Direct calculation shows that when $\psi=2G_X(z,z_0)$, $f=dw$, the $L^2$ norm $\|f\|_{L^2}$ of  $f$ defined by \eqref{measure1.8} is
$$\|f\|_{L^2}=(\int_0^{+\infty}c(t_1)e^{-t_1}dt_1)
\pi\frac{e^{-2\varphi(z_0)}}{c^2_{\beta}(z_0)}.$$
We will show that (see Proposition \ref{proposition3.2}) that $\lim\limits_{t\to +\infty} \frac{H(t;c,2\varphi)}{\int_t^{+\infty}c(t_1)e^{-t_1}dt_1}=\pi\frac{e^{-2\varphi(z_0)}}{c^2_{\beta}(z_0)}$.
\par
We also want to point out that, when $H(-\log r;c,2\varphi)$ is linear with respect to $r$, there exists (see Lemma \ref{existence of F}) a holomorphic extension $F$ of $f$ on $X$ such that the $L^2$ norm of $F$ is equal to $(\int_0^{+\infty}c(t_1)e^{-t_1}dt_1)\pi\frac{e^{-2\varphi}(z_0)}{c^2_{\beta}(z_0)}$ and the $L^2$ norm of $F$ is minimal among all the holomorphic extension of $f$ from $z_0$ to $X$. This shows that $H(-\log r;c,2\varphi)$ is linear with respect to $r$ implies $\|f\|_{L^2}=\inf\{\|F\|_{L^2}\}$.
\par
When we have $\|f\|_{L^2}=\inf\{\|F\|_{L^2}\}$, it follows from $\lim\limits_{t\to +\infty} \frac{H(t;c,2\varphi)}{\int_t^{+\infty}c(t_1)e^{-t_1}dt_1}
=\pi\frac{e^{-2\varphi(z_0)}}{c^2_{\beta}(z_0)}$,
$\|f\|_{L^2}=(\int_0^{+\infty}c(t_1)e^{-t_1}dt_1)
\pi\frac{e^{-2\varphi(z_0)}}{c^2_{\beta}(z_0)}$ and the concavity of $H(-\log r;2\varphi)$ that $H(-\log r;2\varphi)$ is linear with respect to $r$.

Theorem \ref{corollary1.11} shows the following characterization for the equality in optimal $L^{2}$ extension problem to hold.

\begin{Theorem}\label{theorem1.11}
The equality $\|f\|_{L^2}=$ inf $\{\|F\|_{L^2}: F$ is a holomorphic extension of $f$ from $Y$ to $X\}$ holds if and only if the following statements hold
\par
(1) $\varphi=\log|f_{\varphi}|+v$, where $f_{\varphi}$ is a holomorphic function on $X$ and $v$ is a harmonic function on $X$.
\par
(2) $\chi_{-v}=\chi_{z_0}$.
\end{Theorem}

When $\varphi\equiv0$, Theorem \ref{theorem1.11} is the solution of equality part of Suita conjecture \cite{GZ14Ann}.
When $\varphi$ is harmonic, Theorem \ref{theorem1.11} is the solution of extended Suita conjecture \cite{GZ14Ann}.

\section{Proof of Theorem 1.1}
In this section, we modify some techniques in \cite{G16} and prove the Theorem \ref{maintheorem}.

\subsection{$L^2$ methods related to $L^2$ extension theorem}
Let $c(t)$ be a positive function in $C^{\infty}((T,+\infty))$ satisfying
$\int^{\infty}_{T}c(t)e^{-t}dt<\infty$ and
\begin{equation}
(\int^t_T
c(t_1)e^{-t_1}dt_1)^2>c(t)e^{-t}\int^t_T(\int^{t_2}_Tc(t_1)e^{-t_1}dt_1)dt_2
\end{equation}
for any $t\in(T,+\infty)$, where $T\in(-\infty,+\infty)$. This class of functions
is denoted by $C_T$. Especially, if $c(t)e^{-t}$ is decreasing with respect to
t and $\int^{\infty}_{T}c(t)e^{-t}dt<\infty$, then inequality $(2.1)$ holds.
\par
In this section, we present the following Lemma, whose various forms already
appear in\cite{G16,GZ14Ann,GZ15} etc.
\begin{Lemma}
Let $B \in (0, +\infty)$ and $t_0 \ge T$ be arbitrarily given. Let $X$ be
an n-dimensional Stein manifold. Let $d\lambda_n$ be a continuous volume form on $X$ with no zero point. Let $\psi < -T$ be a
plurisubharmonic function on X. Let $\varphi$ be a plurisubharmonic function on X.
Let F be a holomorphic (n,0) form on $\{\psi< -t_0\}$, such that
\begin{equation}
\int_{K\cap \{\psi<-t_0\}} {|F|}^2d\lambda_n<+\infty
\end{equation}
for any compact subset $K$ of $X$, and
\begin{equation}
\int_X \frac{1}{B} \mathbb{I}_{\{-t_0-B< \psi < -t_0\}}  {|F|}^2
e^{{-}\varphi}d\lambda_n\le C <+\infty
\end{equation}
Then there exists a holomorphic (n,0) form $\widetilde F$ on X, such that
\begin{equation}
\int_X {|\widetilde F-(1-b(\psi))F|}^2
e^{{-}\varphi+v(\psi)}c(-v(\psi))d\lambda_n\le C\int^{t_0+B}_{T}c(t)e^{{-}t}dt
\end{equation}
where $b(t)=\int^{t}_{-\infty}\frac{1}{B} \mathbb{I}_{\{-t_0-B< s < -t_0\}}ds$,
$v(t)=\int^{t}_{0}b(s)ds$ and $c(t)\in C_T$.
\label{lemma2.1}
\end{Lemma}
It is clear that $\mathbb{I}_{\{-t_0,+\infty\}}\le b(t) \le
\mathbb{I}_{\{-t_0-B,+\infty\}}$ and $\max\{t,-t_0-B\} \le v(t) \le \max\{t,-t_0\}$.
\par

\subsection{Some properties of $H(t)$}
Following the notations and assumption in Section 1.1, we present some properties related to $H(t)$ in the present section.\par
Let $Z_0$ be a subset of $\{\psi=-\infty\}$ such that $Z_0 \cap
Supp(\mathcal{O}/\mathcal{I}(\varphi+\psi))\neq \emptyset$. Let $U \supset Z_0$ be
an open subset of X. Let $\mathcal{F} \supset \mathcal{I}(\varphi+\psi)|_U$ be a coherent subsheaf of $\mathcal{O}$ on $U$.
\par
We firstly introduce a property of coherent analytic sheaves which will be used frequently in our discussion of $H(t)$.
\begin{Lemma}(Closedness of Submodules, see \cite{Grauert}) Let $N$ be a submodule of $\mathcal{O}^q_{\mathbb{C}^n,0}$, $1\le q < +\infty$, let $f_j \in \mathcal{O}^q_{\mathbb{C}^n}(U)$ be a sequence of $q$-tuples holomorphic in an open neighborhood $U$ of the origin. Assume that the $f_j$ converge uniformly in $U$ towards a $q$-tuple $f \in \mathcal{O}^q_{\mathbb{C}^n}(U)$, assume furthermore that all germs $f_{j,0}$ belong to $N$. Then $f_0 \in N$.
\label{closedness}
\end{Lemma}
\begin{Lemma}
For any $t_0 \in [T,+\infty)$, assume that $\{\tilde{f}_n\}_{n \in \mathbb{N} ^+}$ is a family of holomorphic $(n,0)$ form on $\{\psi<-t_0\}$, which compactly convergent to $\tilde{f}$ on $\{\psi<-t_0\}$.\par
Assume that for any n, there exists open set $U'_n$ such that $Z_0 \subset U'_n \subset U$ and $\tilde{f}_n \in
H^0(\{\psi<-t_0\} \cap \tilde{U}'_n,\mathcal{O}(K_X)\otimes\mathcal{F} ) $. Then there exists an open set $\tilde{U}'$ which satisfies $Z_0 \subset \tilde{U}' \subset U$ such that $\tilde{f} \in
H^0(\{\psi<-t_0\} \cap \tilde{U}',\mathcal{O}(K_X)\otimes\mathcal{F} ) $.
\begin{proof}
As $K_X$ is a holomorphic line bundle on X, then $\mathcal{O} (K_X) \otimes \mathcal{F}$ is a coherent analytic sheaf.
\par
For any $z\in Z_0$, we know the germ $(\tilde{f}_n,z) \in (\mathcal{O} (K_X) \otimes \mathcal{F})_z$. It follows from Lemma \ref{closedness} and $\tilde{f}_n$ compactly convergent to $\tilde{f}$ (when $n \to +\infty$) on $\{\psi<-t_0\}$ that $(\tilde{f},z) \in (\mathcal{O} (K_X) \otimes \mathcal{F})_z$.
\par
As $\mathcal{O} (K_X) \otimes \mathcal{F}$ is coherent analytic sheaf, there exists a small open neighborhood $U_z$ of $z$ such that $(\mathcal{O} (K_X) \otimes \mathcal{F})|_{U_z}$ is finite generated i.e $\exists f^1,\cdots,f^k \in H^0(U_z,\mathcal{O} (K_X) \otimes \mathcal{F})$ such that $\forall y \in U_z$, $(\mathcal{O} (K_X) \otimes \mathcal{F})_y$ is generated by $f^1_y\cdots,f^k_y$. Hence for $\tilde{f}$, there exists $g^i \in \Gamma(U_z,\mathcal{O})$ such that $\tilde{f}_z=\sum_i g^i_z f^i_z$ i.e $\exists$ small open neighborhood $\tilde{U}'_z$ of $z$ satisfies $\tilde{U}'_z \subset U_z$ and $\tilde{f}|_{\tilde{U}'_z}=(\sum_i g^i_z f^i_z)|_{\tilde{U}'_z}$ which implies $\tilde{f}\in H^0(\tilde{U}'_z,\mathcal{O} (K_X) \otimes \mathcal{F})$. Take $\tilde{U}'=(\mathop{\cup} \limits_{z\in Z_0}\tilde{U}'_z)\cap U$. We now find a open set $\tilde{U}'$ satisfies $Z_0 \subset \tilde{U}' \subset U$ such that $f\in H^0(\{\psi<-t_0\}\cap \tilde{U}',\mathcal{O} (K_X) \otimes \mathcal{F})$ .

\end{proof}
\label{lemma2.2}
\end{Lemma}

The following lemma is a characterization of $H(T)\neq 0$.
\begin{Lemma}
The following two statements are equivalent:
\par
(1) For any open set $U'$ satisfying $Z_0 \subset U' \subset U$, $f \notin
H^0(U',\mathcal{O} (K_X) \otimes \mathcal{F} )$.
\par
(2) $H(T) \neq 0$.
\label{lemma2.3}
\end{Lemma}
\begin{proof}$(2)\Rightarrow (1)$
If there exists open set $U'$ satisfies $Z_0 \subset U' \subset U$ and $f \in
H^0(U',\mathcal{O} (K_X) \otimes \mathcal{F} )$, then $H(T)=0$ (just take $\tilde{f}=0$) .\par
Now we prove $(1)\Rightarrow (2)$ by contradiction.\par
Assume $H(T)=0$, then there exist holomorphic $(n,0)$ forms $\{\tilde{f}_n\}_{n \in \mathbb{N} ^+}$ on $X$ such that $\lim\limits_{n\to +\infty} \int_X
|\tilde{f}_n|^2e^{-\varphi}c(-\psi)dV_X=0$ and for each $n$, $\exists\  U_n'$ satisfies $Z_0 \subset U_n'\subset U$ and $\tilde{f}_n-f\in
H^0(U_n',\mathcal{O} (K_X) \otimes \mathcal{F})$. As $e^{-\varphi}c(-\psi)$ has positive
lower bound on any compact subset of $X$, then (by diagonal method) there exists a subsequence of
$\{\tilde{f}_n\}_{n \in \mathbb{N} ^+}$ denoted by $\{\tilde{f}_{n_k}\}_{k \in
\mathbb{N} ^+}$ compactly convergent to 0 on $X$ when $k \to +\infty$. Hence $\tilde{f}_{n_k}-f$ is
compactly convergent to $0-f=f$ on U.\par
 By Lemma \ref{lemma2.2}, there exists an open set $U'$ satisfies $Z_0 \subset U' \subset U$ such that $f\in H^0(U',\mathcal{O} (K_X) \otimes \mathcal{F})$ which contradicts the condition.
\end{proof}

The following lemma shows the uniqueness of the holomorphic $(n,0)$ form related
to $H(t)$.
\begin{Lemma}
Assume that $H(t)<+\infty$ for some $t\in [T,+\infty)$. Then there exists a unique
holomorphic $(n,0)$ form $F_t$ on $\{\psi<-t\}$ satisfying
$$\ (F_t-f)\in
H^0(\{\psi<-t\}\cap U', \mathcal{O} (K_X) \otimes \mathcal{F}),$$
for some open set $U'$ such that $Z_0 \subset U' \subset U$ and $$\ \int_{\{\psi<-t\}}|F_t|^2e^{-\varphi}c(-\psi)dV_X=H(t).$$
\par
Furthermore, for any holomorphic $(n,0)$ form $\hat{F}$ on $\{\psi<-t\}$ satisfying
$$\int_{\{\psi<-t\}}|\hat{F}|^2e^{-\varphi}c(-\psi)dV_X<+\infty$$ and $$(\hat{F}-f)\in
H^0(\{\psi<-t\}\cap \hat{U}',\mathcal{O} (K_X) \otimes \mathcal{F})$$ for some open set $\hat{U}'$ such that $Z_0 \subset\hat{U}' \subset U$, the following equality holds
\begin{equation}
\begin{split}
&\int_{\{\psi<-t\}}|F_t|^2e^{-\varphi}c(-\psi)dV_X+
\int_{\{\psi<-t\}}|\hat{F}-F_t|^2e^{-\varphi}c(-\psi)dV_X\\
=&\int_{\{\psi<-t\}}|\hat{F}|^2e^{-\varphi}c(-\psi)dV_X
\label{orhnormal F}
\end{split}
\end{equation}
\label{existence of F}
\end{Lemma}
\begin{proof}
As $H(t)<+\infty$, then there exist
holomorphic $(n,0)$-forms $\{f_n\}_{n \in \mathbb{N}^+}$ on $\{\psi<-t\}$ such
that
$\lim\limits_{n \to +\infty}\int_{\{\psi<-t\}}|f_n|^2e^{-\varphi}c(-\psi)dV_X=H(t)$ and for each $n$, there exists $ U_n'$ such that $Z_0 \subset U_n'\subset U$ and $(\tilde{f}_n-f)\in
H^0(\{\psi<-t\}\cap U_n',\mathcal{O} (K_X) \otimes \mathcal{F})$.
As $e^{-\varphi}c(-\psi)$ has
positive lower bound on any compact subset of $\{\psi<-t\}$, then (by diagonal method) there exist a
subsequence of $\{f_j\}$ also denoted by $\{f_j\}$ compact convergence to a
holomorphic $(n,0)$ form $F$ (when $j \to +\infty$) on $\{\psi<-t\}$ satisfying
\begin{equation}
\begin{split}
\int_{K}|F|^2e^{-\varphi}c(-\psi)dV_X
&\leq
\liminf \limits_{j \to +\infty}\int_{K}|f_j|^2e^{-\varphi}c(-\psi)dV_X\\
&\leq
\liminf \limits_{j \to
+\infty}\int_{\{\psi<-t\}}|f_j|^2e^{-\varphi}c(-\psi)dV_X\\
&=H(t)
\end{split}
\end{equation}
\par
Lemma \ref{lemma2.2} shows that there exists an open subset $U'$ such that $Z_0 \subset U' \subset U$ and $(F-f)\in H^0(\{\psi<-t\}\cap U',\mathcal{O} (K_X) \otimes \mathcal{F})$ which implies
$H(t) \leq
\int_{\{\psi<-t\}}|F|^2e^{-\varphi}c(-\psi)dV_X$. Hence we obtain the existence of
$F_t(=F)$.\par
We prove the uniqueness of $F_t$ by contradiction: if not, there exists
two different holomorphic $(n,0)$ forms $f_1$ and $f_2$ on $\{\psi<-t\}$
satisfying $\int_{\{\psi<-t\}}|f_1|^2e^{-\varphi}$  $c(-\psi)dV_X=
\int_{\{\psi<-t\}}|f_2|^2e^{-\varphi}c(-\psi)dV_X=H(t)$, $(f_1-f)\in
H^0(\{\psi<-t\}\cap U'_1, \mathcal{O} (K_X) \otimes \mathcal{F})$ and $(f_2-f)\in
H^0(\{\psi<-t\}\cap U'_2, \mathcal{O} (K_X) \otimes \mathcal{F})$ where both $U'_1,U'_2$ are open set satisfy $Z_0 \subset U'_1 \subset U$ and $Z_0 \subset U'_2 \subset U$ . Note that
\begin{equation}
\begin{split}
\int_{\{\psi<-t\}}|\frac{f_1+f_2}{2}|^2e^{-\varphi}c(-\psi)dV_X+
\int_{\{\psi<-t\}}|\frac{f_1-f_2}{2}|^2e^{-\varphi}c(-\psi)dV_X\\
=\frac{1}{2}(\int_{\{\psi<-t\}}|f_1|^2e^{-\varphi}c(-\psi)dV_X+
\int_{\{\psi<-t\}}|f_1|^2e^{-\varphi}c(-\psi)dV_X)=H(t)
\end{split}
\end{equation}
then we obtain that
\begin{equation}
\begin{split}
\int_{\{\psi<-t\}}|\frac{f_1+f_2}{2}|^2e^{-\varphi}c(-\psi)dV_X
\le H(t)
\end{split}
\end{equation}
and $(\frac{f_1+f_2}{2}-f)\in H^0(\{\psi<-t\}\cap (U'_1\cap U'_2), \mathcal{O} (K_X) \otimes \mathcal{F})$, which contradicts to the definition of $H(t)$.\par
Now we prove the equality \eqref{orhnormal F}. For any holomorphic $(n,0)$ form $h$ on $\{\psi<-t\}$
satisfying $\int_{\{\psi<-t\}}|h|^2e^{-\varphi}c(-\psi)dV_X<+\infty$ and $h \in
H^0(\{\psi<-t\}\cap U'_h, \mathcal{O} (K_X) \otimes \mathcal{F})$ for some open subset $U_h$ which $Z_0 \subset U'_h \subset U$.  It is clear that for any complex
number $\alpha$, $F_t+\alpha h$ satisfying $((F_t+\alpha h)-f) \in
H^0(\{\psi<-t\}\cap (U'_h\cap U'_t), \mathcal{O} (K_X) \otimes \mathcal{F})$ and
$\int_{\{\psi<-t\}}|F_t|^2e^{-\varphi}c(-\psi)dV_X \leq \int_{\{\psi<-t\}}|F_t+\alpha
h|^2e^{-\varphi}c(-\psi)dV_X$. Note that
\begin{equation}
\begin{split}
\int_{\{\psi<-t\}}|F_t+\alpha
h|^2e^{-\varphi}c(-\psi)dV_X-\int_{\{\psi<-t\}}|F_t|^2e^{-\varphi}c(-\psi)dV_X\geq 0
\end{split}
\end{equation}
(By considering $\alpha \to 0$) implies
\begin{equation}
\begin{split}
\mathfrak{R} \int_{\{\psi<-t\}}F_t\bar{h}e^{-\varphi}c(-\psi)dV_X=0
\end{split}
\end{equation}
then we have
\begin{equation}
\begin{split}
\int_{\{\psi<-t\}}|F_t+h|^2e^{-\varphi}c(-\psi)dV_X=
\int_{\{\psi<-t\}}(|F_t|^2+|h|^2)e^{-\varphi}c(-\psi)dV_X
\end{split}
\end{equation}
\par
Letting $h=\hat{F}-F_t$ (and $U'_h=\hat{U}'\cap U'_t $), we obtain equality \eqref{orhnormal F}.
\end{proof}

Now we show the lower semi-continuity property of $H(h^{-1}(r))$.

\begin{Lemma}
Assume that $H(T)<+\infty$. Then $H(t)$ is decreasing with respect to $t\in
[T,+\infty)$, such that $\lim \limits_{t \to t_0+0}H(t)=H(t_0)$ $(t_0\in
[T,+\infty))$, $\lim \limits_{t \to t_0-0}H(t)\geq H(t_0)$ $(t_0\in
(T,+\infty))$, and $\lim \limits_{t \to +\infty}H(t)=0$. Especially, $H(h^{-1}(r))$
 is lower semi-continuous with respect to $r$.
 \label{semicontinuous}
\end{Lemma}
\begin{proof}
By the definition of $H(t)$, it is clear that $H(t)$ is decreasing on
$[T,+\infty)$ and $\lim \limits_{t \to t_0-0}H(t)\geq H(t_0)$. It suffices
to prove $\lim \limits_{t \to t_0+0}H(t)=H(t_0)$ . We prove it by
contradiction: if not, then $\lim \limits_{t \to t_0+0}H(t)<
H(t_0)<+\infty$.
\par
By Lemma \ref{existence of F}, for any $t>t_0$, there exists a unique holomorphic $(n,0)$ form
$F_t$ on $\{\psi<-t\}$ satisfying
$\int_{\{\psi<-t\}}|F_t|^2e^{-\varphi}c(-\psi)dV_X=H(t)$ and $(F_t-f) \in
H^0(\{\psi<-t\}\cap U'_t, \mathcal{O} (K_X) \otimes \mathcal{F})$ where open set $U'_t$ satisfies $Z_0 \subset U'_t\subset U$. Note that $H(t)$ is decreasing
implies that $\int_{\{\psi<-t\}}|F_t|^2e^{-\varphi}c(-\psi)dV_X\leq \lim
\limits_{t \to t_0+0}H(t) <+\infty$ for any $t>t_0$.
\par
For any compact subset $K$ of $\{\psi<-t_0\}$, as $e^{-\varphi}c(-\psi)$
has positive lower bound on $K$, there
exists $F_{t_j}$ $(t_j \to t_0+0,\ as\ j \to +\infty)$ uniformly convergent on
$K$, then (by diagonal method) there exists a subsequence of $F_{t_j}$ (also
denoted by $F_{t_j}$) convergent on any compact subset of $\{\psi<-t_0\}$.
\par
Let $\hat{F}_{t_0}:=\lim \limits_{j \to +\infty}F_{t_j}$, which is a
holomorphic $(n,0)$ form on $\{\psi<-t_0\}$. By Lemma \ref{lemma2.2}, we conclude that there exists an open set $\hat{U}'$ such that $Z_0 \subset \hat{U}' \subset U$ and
$(\hat{F}_{t_0}-f)\in H^0(\{\psi<-t_0\}\cap \hat{U}',
\mathcal{O} (K_X) \otimes \mathcal{F})$.
Then it follows from the decreasing property of $H(t)$ that
\begin{equation}
\begin{split}
\int_{K}|\hat{F}_{t_0}|^2e^{-\varphi}c(-\psi)dV_X
&\leq
\liminf\limits_{j \to +\infty}\int_{K}|F_{t_j}|^2e^{-\varphi}c(-\psi)dV_X\\
&\leq
\liminf \limits_{j \to +\infty}H(t_j)\\
&\leq \lim \limits_{t\to t_0+0} H(t)
\end{split}
\end{equation}
for any compact set $K \subset \{\psi<-t_0\}$. It follows from Levi's theorem that
\begin{equation}
\begin{split}
\int_{\{\psi<-t_0\}}|\hat{F}_{t_0}|^2e^{-\varphi}c(-\psi)dV_X
\leq \lim \limits_{t\to t_0+0} H(t)
\end{split}
\end{equation}
Then we obtain that $H(t_0)\leq
\int_{\{\psi<-t_0\}}|\hat{F}_{t_0}|^2e^{-\varphi}c(-\psi)dV_X
\leq \lim \limits_{t\to t_0+0} H(t)$
which contradicts $\lim \limits_{t\to t_0+0} H(t) <H(t)$.
\end{proof}

We consider the derivatives of $H(t)$ in the following lemma.
\begin{Lemma}
Assume that $H(T)<+\infty$. Then for any $t_0\in(T,+\infty)$, we have
\begin{equation}
\begin{split}
\frac{H(T)-H(t_0)}{\int^{+\infty}_T c(t)e^{-t}dt-\int^{+\infty}_{t_0}
c(t)e^{-t}dt} \leq
\frac{\liminf \limits_{B \to
0+0}(\frac{H(t_0)-H(t_0+B)}{B})}{c(t_0)e^{-t_0}}
\end{split}
\end{equation}
\label{lemma2.6}
\end{Lemma}
\begin{proof}
By Lemma \ref{existence of  F}, there exists a holomorphic $(n,0)$ form $F_{t_0}$ on $\{\psi<-t_0\}$,
such that $\int_{\{\psi<-t_0\}}|F_{t_0}|^2e^{-\varphi}c(-\psi)dV_X=H(t_0)$ and  $(F_{t_0}-f)\in H^0(\{\psi<-t_0\}\cap U'_{t_0}, \mathcal{O} (K_X) \otimes \mathcal{F})$ where open set $ U'_{t_0}$  satisfies $Z_0 \subset U'_{t_0} \subset U$ .\par
Note that $\liminf \limits_{B \to
0+0}\frac{H(t_0)-H(t_0+B)}{B} \in [0,+\infty)$ because of the decreasing
property of $H(t)$. Then there exist $1 \ge B_j \to 0+0$ (as $j \to +\infty$) such that
\begin{equation}
\lim \limits_{j \to +\infty}\frac{H(t_0)-H(t_0+B_j)}{B_j} =\liminf
\limits_{B \to 0+0}\frac{H(t_0)-H(t_0+B)}{B}
\end{equation}
and $\{\frac{H(t_0)-H(t_0+B_j)}{B_j}\}_{j\in \mathbb{N}^+}$ is bounded.\par
As $\int_{\{\psi<-t_0\}}|F_{t_0}|^2e^{-\varphi}c(-\psi)dV_X=H(t_0)<+\infty$ and $e^{-\varphi}c(-\psi)$ has positive lower bounded on any compact set $K$ of $X$. Then $\int_{K\cap \{\psi<-t_0\}} {|F_{t_0}|}^2dV_X<+\infty$ for any compact set $K$.
Note that $c(t)$ is smooth on $(T,+\infty)$, hence bounded on $[t_0,t_0+1]$, so
$\int_X \frac{1}{B_j} \mathbb{I}_{\{-t_0-B_j< \psi < -t_0\}}  {|F_{t_0}|}^2 e^{{-}\varphi}dV_X<+\infty$.
\par
By Lemma 2.1 $(\varphi\backsim\varphi+\psi)$, for any $B_j$, there exists holomorphic
$(n,0)$ form $\tilde{F}_j$ on $X$ such that
 \begin{equation}
 \int_X {|\widetilde F_j-(1-b_{t_0,B_j}(\psi))F_{t_0}|}^2 e^{{-}(\varphi+\psi)+v_j(\psi)}c(-v_j(\psi))dV_X<+\infty
 \label{L2}
  \end{equation}
  where $b_{t_0,B_j}(t)=\int^{t}_{-\infty}\frac{1}{B_j} \mathbb{I}_{\{-t_0-B_j< s < -t_0\}}ds$ and
$v_j(t)=\int^{t}_{0}b_{t_0,B_j}(s)ds$.
\par
It follows from \eqref{L2} that $\int_{\{\psi<-t_0-B_j\}}{|\widetilde F_j-(1-b_{t_0,B_j}(\psi))F_{t_0}|}^2 e^{{-}(\varphi+\psi)+v_j(\psi)}$ $c(-v_j(\psi))dV_X<+\infty$, and note that $e^{-t}c(t)$ is decreasing with respect to $t$ and $v_j(\psi)\ge \max\{\psi,-t_0-B_j\} \ge -t_0-1 $. Hence $e^{v_j(\psi)}c(-v_j(\psi))$ has positive lower bound, which implies
 \begin{equation}
\int_{\{\psi<-t_0-B_j\}}{|\widetilde F_j-(1-b_{t_0,B_j}(\psi))F_{t_0}|}^2 e^{-(\varphi+\psi)}dV_X<+\infty
\end{equation}
\par
As $\{\psi<-t_0-B_j\}$ is open, there exists an open subset $U'_j \subset (\{\psi<-t_0-B_j\}\cap U)$ such that $(\widetilde F_j- F_{t_0}) \in H^0(\{\psi<-t_0\}\cap U'_j,\mathcal{O} (K_X) \otimes \mathcal{I}(\varphi+\psi)) \subset H^0(\{\psi<-t_0\}\cap U'_j,\mathcal{O} (K_X) \otimes \mathcal{F})$, which implies
$(\widetilde F_j- f) \in H^0(\{\psi<-t_0\}\cap (U'_j\cap U'_{t_0}),\mathcal{O} (K_X) \otimes \mathcal{F})$.
\par
As $t \leq v(t)$, the decreasing property of $c(t)e^{-t}$ shows that
\begin{equation}
c(t)e^{-t}\leq c(-v(-t))e^{v(-t)}
\end{equation}
for any $t \geq 0$,which implies
\begin{equation}
e^{-\psi+v(\psi)}c(-v(\psi))\geq c(-\psi)
\end{equation}
\par
So we have
\begin{flalign}
&\int_X|\tilde{F}_j-(1-b_{t_0,B_j}(\psi))F_{t_0}|^2e^{-\varphi}c(-\psi)dV_X\nonumber\\
\leq &
\int_X|\tilde{F}_j-(1-b_{t_0,B_j}(\psi))F_{t_0}|^2e^{-\varphi}e^{-\psi+v(\psi)}c(-v(\psi))dV_X\nonumber\\
\leq &
\int^{t_0+B_j}_Tc(t)e^{-t}dt\int_X\frac{1}{B_j}
\mathbb{I}_{\{-t_0-B_j<\psi<-t_0\}}|F_{t_0}|^2e^{-\varphi-\psi}dV_X\nonumber\\
\leq &
\frac{e^{t_0+B_j}\int^{t_0+B_j}_Tc(t)e^{-t}dt}{\inf
\limits_{t\in(t_0,t_0+B_j)}c(t)}\int_X\frac{1}{B_j}
\mathbb{I}_{\{-t_0-B_j<\psi<-t_0\}}|F_{t_0}|^2e^{-\varphi}c(-\psi)dV_X\nonumber\\
= &
\frac{e^{t_0+B_j}\int^{t_0+B}_Tc(t)e^{-t}dt}{\inf
\limits_{t\in(t_0,t_0+B_j)}c(t)}\times
(\int_X\frac{1}{B_j}\mathbb{I}_{\{\psi<-t_0\}}|F_{t_0}|^2e^{-\varphi}c(-\psi)dV_X\nonumber\\
&-\int_X\frac{1}{B_j}\mathbb{I}_{\{\psi<-t_0-B_j\}}|F_{t_0}|^2e^{-\varphi}c(-\psi)dV_X)\nonumber\\
\leq &
\frac{e^{t_0+B_j}\int^{t_0+B}_Tc(t)e^{-t}dt}{\inf
\limits_{t\in(t_0,t_0+B_j)}c(t)} \times
\frac{H(t_0)-H(t_0+B_j)}{B_j}
\label{219}
\end{flalign}

\par
After the estimate for $\int_X|\tilde{F}_j-(1-b_{t_0,B_j}(\psi))F_{t_0}|^2e^{-\varphi}c(-\psi)dV_X$, we can prove the main result.
\par
Firstly, we will prove that $\int_X |\tilde{F}_j|^2e^{-\varphi}c(\psi)dV_X$ is uniformly bounded
with respect to j.
\par
Note that
\begin{equation}
\begin{split}
&(\int_X|\tilde{F}_j-(1-b_{t_0,B_j}(\psi))F_{t_0}|^2e^{-\varphi}c(-\psi)dV_X)^{1/2}\\
\geq &
(\int_X|\tilde{F}_j|^2e^{-\varphi}c(-\psi)dV_X)^{1/2}-
(\int_X|(1-b_{t_0,B_j}(\psi))F_{t_0}|^2e^{-\varphi}c(-\psi)dV_X)^{1/2}
\end{split}
\end{equation}
then it follows from inequality (\ref{219}) that
\begin{equation}
\begin{split}
&(\int_X|\tilde{F}_j|^2e^{-\varphi}c(-\psi)dV_X)^{1/2} \\
\leq
&(\frac{e^{t_0+B_j}\int^{t_0+B}_Tc(t)e^{-t}dt}{\inf
\limits_{t\in(t_0,t_0+B_j)}c(t)})^{1/2} \times
(\frac{H(t_0)-H(t_0+B_j)}{B_j})^{1/2}\\
&+
(\int_X|(1-b_ {t_0,B_j}(\psi))F_{t_0}|^2e^{-\varphi}c(-\psi)dV_X)^{1/2}
\end{split}
\end{equation}
\par
Since $\{\frac{H(t_0)-H(t_0+B_j)}{B_j}\}_{j\in \mathbb{N}^+}$  is bounded and
$0\leq b_{t_0,B_j}(\psi) \leq 1$, then $\int_X|\tilde{F}_j|^2e^{-\varphi}$ $c(-\psi)dV_X$
is uniformly bounded with respect to j.\par
Now we will prove the main result.\par
It follows from $b_{t_0,B_j}(\psi)=1$ on $\{\psi \geq -t_0\}$ that
\begin{flalign}
&\int_X|\tilde{F}_j-(1-b_{t_0,B_j}(\psi))F_{t_0}|^2e^{-\varphi}c(-\psi)dV_X \nonumber\\
=&
\int_{\{\psi \geq -t_0\}}|\tilde{F}_j|^2e^{-\varphi}c(-\psi)dV_X\nonumber\\
+&
\int_{\{\psi <
-t_0\}}|\tilde{F}_j-(1-b_{t_0,B_j}(\psi))F_{t_0}|^2e^{-\varphi}c(-\psi)dV_X
\end{flalign}
It is clear that
\begin{flalign}
&\int_{\{\psi <
-t_0\}}|\tilde{F}_j-(1-b_{t_0,B_j}(\psi))F_{t_0}|^2e^{-\varphi}c(-\psi)dV_X\nonumber\\
\geq &
((\int_{\{\psi < -t_0\}}|\tilde{F}_j-F_{t_0}|^2e^{-\varphi}c(-\psi)dV_X)^{1/2}-
(\int_{\{\psi <
-t_0\}}|b_{t_0,B_j}(\psi)F_{t_0}|^2e^{-\varphi}c(-\psi)dV_X)^{1/2})^2\nonumber\\
\geq &
\int_{\{\psi < -t_0\}}|\tilde{F}_j-F_{t_0}|^2e^{-\varphi}c(-\psi)dV_X\nonumber\\
&-
2(\int_{\{\psi < -t_0\}}|\tilde{F}_j-F_{t_0}|^2e^{-\varphi}c(-\psi)dV_X)^{1/2}
(\int_{\{\psi < -t_0\}}|b_{t_0,B_j}(\psi)F_{t_0}|^2e^{-\varphi}c(-\psi)dV_X)^{1/2}\nonumber\\
\geq &
\int_{\{\psi < -t_0\}}|\tilde{F}_j-F_{t_0}|^2e^{-\varphi}c(-\psi)dV_X\nonumber\\
&-
2(\int_{\{\psi < -t_0\}}|\tilde{F}_j-F_{t_0}|^2e^{-\varphi}c(-\psi)dV_X)^{1/2}
(\int_{\{-t_0-B_j<\psi < -t_0\}}|F_{t_0}|^2e^{-\varphi}c(-\psi)dV_X)^{1/2}
\end{flalign}
where the last inequality follow from $0\leq b_{t_0,B_j}(\psi) \leq 1$ and
$b_{t_0,B_j}(\psi)=0$ on $\{\psi \leq -t_0-B_j\}$.
Combining equality (2.23), inequality (2.24) and equality (\ref{orhnormal F}), we obtain that
\begin{flalign}
&\int_X|\tilde{F}_j-(1-b_{t_0,B_j}(\psi))F_{t_0}|^2e^{-\varphi}c(-\psi)dV_X \nonumber\\
=&
\int_{\{\psi \geq -t_0\}}|\tilde{F}_j|^2e^{-\varphi}c(-\psi)dV_X+
\int_{\{\psi <
-t_0\}}|\tilde{F}_j-(1-b_{t_0,B_j}(\psi))F_{t_0}|^2e^{-\varphi}c(-\psi)dV_X\nonumber\\
\geq &
\int_{\{\psi \geq -t_0\}}|\tilde{F}_j|^2e^{-\varphi}c(-\psi)dV_X+
\int_{\{\psi < -t_0\}}|\tilde{F}_j-F_{t_0}|^2e^{-\varphi}c(-\psi)dV_X\nonumber\\
&-
2(\int_{\{\psi < -t_0\}}|\tilde{F}_j-F_{t_0}|^2e^{-\varphi}c(-\psi)dV_X)^{1/2}
(\int_{\{-t_0-B_j<\psi < -t_0\}}|F_{t_0}|^2e^{-\varphi}c(-\psi)dV_X)^{1/2}\nonumber\\
= &
\int_{\{\psi \geq -t_0\}}|\tilde{F}_j|^2e^{-\varphi}c(-\psi)dV_X+
\int_{\{\psi < -t_0\}}|\tilde{F}_j|^2e^{-\varphi}c(-\psi)dV_X-
\int_{\{\psi < -t_0\}}|F_{t_0}|^2e^{-\varphi}c(-\psi)dV_X\nonumber\\
&-
2(\int_{\{\psi < -t_0\}}|\tilde{F}_j-F_{t_0}|^2e^{-\varphi}c(-\psi)dV_X)^{1/2}
(\int_{\{-t_0-B_j<\psi < -t_0\}}|F_{t_0}|^2e^{-\varphi}c(-\psi)dV_X)^{1/2}\nonumber\\
= &
\int_{X}|\tilde{F}_j|^2e^{-\varphi}c(-\psi)dV_X-
\int_{\{\psi < -t_0\}}|F_{t_0}|^2e^{-\varphi}c(-\psi)dV_X \nonumber\\
&-
2(\int_{\{\psi < -t_0\}}|\tilde{F}_j-F_{t_0}|^2e^{-\varphi}c(-\psi)dV_X)^{1/2}
(\int_{\{-t_0-B_j<\psi < -t_0\}}|F_{t_0}|^2e^{-\varphi}c(-\psi)dV_X)^{1/2}\nonumber\\
\label{2.25}
\end{flalign}

\par
It follows from equality (\ref{orhnormal F}) that
\begin{equation}
\begin{split}
&(\int_{\{\psi < -t_0\}}|\tilde{F}_j-F_{t_0}|^2e^{-\varphi}c(-\psi)dV_X)^{1/2}\\
=&
(\int_{\{\psi < -t_0\}}(|\tilde{F}_j|^2-|F_{t_0}|^2)e^{-\varphi}c(-\psi)dV_X)^{1/2}\\
\leq &
(\int_{\{\psi < -t_0\}}|\tilde{F}_j|^2e^{-\varphi}c(-\psi)dV_X)^{1/2}\\
\leq &
(\int_{X}|\tilde{F}_j|^2e^{-\varphi}c(-\psi)dV_X)^{1/2}
\end{split}
\label{2.26}
\end{equation}
\par
Since $\int_{X}|\tilde{F}_j|^2e^{-\varphi}c(-\psi)dV_X$ is uniformly bounded with respect to
$j$, inequality \eqref{2.26} implies that $(\int_{\{\psi <
-t_0\}}|\tilde{F}_j-F_{t_0}|^2e^{-\varphi}c(-\psi)dV_X)^{1/2}$ is uniformly bounded with respect
to $j$. \par
It follows from $\int_{\{\psi < -t_0\}}|F_{t_0}|^2e^{-\varphi}c(-\psi)dV_X=H(t_0)\leq H(T)<+\infty$ and the dominated convergence theorem that $\lim\limits_{j \to +\infty} \int_{\{-t_0-B_j<\psi <
-t_0\}}|F_{t_0}|^2e^{-\varphi}c(-\psi)$ $dV_X=0$. Hence
\begin{flalign}\nonumber
\lim \limits_{j \to +\infty}
(\int_{\{\psi < -t_0\}}|\tilde{F}_j-F_{t_0}|^2e^{-\varphi}c(-\psi)dV_X)^{1/2}
(\int_{\{-t_0-B_j<\psi < -t_0\}}|F_{t_0}|^2e^{-\varphi}c(-\psi)dV_X)^{1/2}=0\\
\end{flalign}
\par
Combining with inequality \eqref{2.25}, we obtain
\begin{equation}
\begin{split}
&\liminf\limits_{j \to +\infty}
\int_X|\tilde{F}_j-(1-b_{t_0,B_j}(\psi))F_{t_0}|^2e^{-\varphi}c(-\psi)dV_X \\
\geq &
\liminf \limits_{j \to +\infty}
\int_X|\tilde{F}_j|^2e^{-\varphi}c(-\psi)dV_X-
\int_{\{\psi<-t_0\}}|F_{t_0}|^2e^{-\varphi}c(-\psi)dV_X
\end{split}
\label{2.27}
\end{equation}
\par
Using inequality \eqref{219} and inequality \eqref{2.27}, we obtain that

\begin{flalign}
&\frac{\int^{t_0}_Tc(t)e^{-t}dt}{c(t_0)e^{-t_0}} \lim\limits_{j \to
+\infty}
\frac{H(t_0)-H(t_0+B_j)}{B_j}\nonumber\\
=&
\lim \limits_{j \to +\infty}
\frac{e^{t_0+B_j}\int^{t_0+B_j}_Tc(t)e^{-t}dt}{\inf
\limits_{t\in(t_0,t_0+B_j)}c(t)} \times
\frac{H(t_0)-H(t_0+B_j)}{B_j} \nonumber\\
\geq &
\liminf \limits_{j \to +\infty}
\frac{e^{t_0+B_j}\int^{t_0+B_j}_Tc(t)e^{-t}dt}{\inf
\limits_{t\in(t_0,t_0+B_j)}c(t)}
\int_X
\frac{1}{B_j}\mathbb{I}_{\{-t_0-B_j<\psi<-t_0\}}|F_{t_0}|^2e^{-\varphi}c(-\psi)dV_X\nonumber\\
\geq&
\liminf \limits_{j \to +\infty}
\int_X |\tilde{F}_j-(1-b_{t_0,B_j}(\psi))F_{t_0}|^2e^{-\varphi}c(-\psi)dV_X\nonumber\\
\geq&
\liminf \limits_{j \to +\infty}
\int_X |\tilde{F}_j|^2e^{-\varphi}c(-\psi)dV_X-
\int_{\{\psi<-t_0\}} |F_{t_0}|^2e^{-\varphi}c(-\psi)dV_X\nonumber\\
\geq&
H(T)-H(t_0)
\end{flalign}
\par
This proves Lemma \ref{lemma2.6}.

\end{proof}

Lemma \ref{lemma2.6} implies the following lemma.
\begin{Lemma}
Assume that $H(T)<+\infty$. Then for any $t_0,t_1\in[T,+\infty)$, where $t_0<t_1$, we have
\begin{equation}
\begin{split}
\frac{H(t_0)-H(t_1)}{\int^{t_1}_{t_0} c(t)e^{-t}dt} \leq
\frac{\liminf \limits_{B \to
0+0}(\frac{H(t_1)-H(t_1+B)}{B})}{c(t_1)e^{-(t_1)}}
\end{split}
\end{equation}
i.e.
\begin{equation}
\begin{split}
\frac{H(t_0)-H(t_1)}{\int^{+\infty}_{t_0}
c(t)e^{-t}dt-\int^{+\infty}_{t_1} c(t)e^{-t}dt} \leq
\liminf \limits_{B \to 0+0}
\frac{H(t_1)-H(t_1+B)}{\int^{+\infty}_{t_1}
c(t)e^{-t}dt-\int^{+\infty}_{t_1+B} c(t)e^{-t}dt}
\end{split}
\end{equation}
\label{lemma2.7}
\end{Lemma}
\subsection{Proof of Theorem 1.1}
As $H(h^{-1}(r);c(t))$ is lower semicontinuous
(Lemma \ref{semicontinuous}), then it follows from the following well-known property of concave
functions that Lemma \ref{lemma2.7} implies Theorem \ref{maintheorem}.
\begin{Lemma}(see \cite{G16})
Let $H(r)$ be a lower semicontinuous function on $(0,R]$. Then $H(r)$ is concave
if and only if
\begin{equation}
\begin{split}
\frac{H(r_1)-H(r_2)}{r_1-r_2} \leq
\liminf\limits_{r_3 \to r_2-0}
\frac{H(r_3)-H(r_2)}{r_3-r_2}
\end{split}
\end{equation}
holds for any $0<r_2<r_1 \leq R$.
\label{lemma2.8}
\end{Lemma}

\section{Proof of Theorem 1.3}
In this section, we will prove Theorem \ref{theorem1.3}.
\begin{proof}[proof of Theorem \ref{theorem1.3}]
We firstly recall some basic construction in the proof of Lemma \ref{lemma2.6}.\par
Given $t_0 \in (T,+\infty)$.
By Lemma \ref{existence of  F}, there exists a holomorphic $(n,0)$ form $F_{t_0}$ on $\{\psi<-t_0\}$,
such that $\int_{\{\psi<-t_0\}}|F_{t_0}|^2e^{-\varphi}c(-\psi)dV_X=H(t_0)$ and  $(F_{t_0}-f)\in H^0(\{\psi<-t_0\}\cap U'_{t_0}, \mathcal{O} (K_X) \otimes \mathcal{F})$, where open subset $ U'_{t_0}$  satisfies $Z_0 \subset U'_{t_0} \subset U$ .\par
Note that $\liminf \limits_{B \to
0+0}\frac{H(t_0)-H(t_0+B)}{B} \in [0,+\infty)$ because of the decreasing
property of $H(t)$. Then there exist $1 \ge B_j \to 0+0$ (as $j \to +\infty$) such that
\begin{equation}
\lim \limits_{j \to +\infty}\frac{H(t_0)-H(t_0+B_j)}{B_j} =\liminf
\limits_{B \to 0+0}\frac{H(t_0)-H(t_0+B)}{B}
\end{equation}
and $\{\frac{H(t_0)-H(t_0+B_j)}{B_j}\}_{j\in \mathbb{N}^+}$ is bounded.
\par
As $\int_{\{\psi<-t_0\}}|F_{t_0}|^2e^{-\varphi}c(-\psi)dV_X=H(t_0)<+\infty$ and $e^{-\varphi}c(-\psi)$ has positive lower bounded on any compact set $K$ of $X$. Then $\int_{K\cap \{\psi<-t_0\}} {|F_{t_0}|}^2dV_X<+\infty$ for any compact set $K$.
Note that $c(t)$ is smooth on $(T,+\infty)$, hence bounded on $[t_0,t_0+1]$, so
$\int_X \frac{1}{B_j} \mathbb{I}_{\{-t_0-B_j< \psi < -t_0\}}  {|F_{t_0}|}^2 e^{{-}\varphi}dV_X<+\infty$.
\par
By Lemma 2.1 $(\varphi\backsim\varphi+\psi)$, for any $B_j$, there exists holomorphic
$(n,0)$ form $\tilde{F}_j$ on $X$ such that
 \begin{equation}
 \begin{split}
 &\int_X {|\widetilde F_j-(1-b_{t_0,B_j}(\psi))F_{t_0}|}^2 e^{{-}(\varphi+\psi)+v_j(\psi)}c(-v_j(\psi))dV_X \\
 \le &\int_T^{t_0+B_j}c(t)e^{-t}dt\int_X \frac{1}{B_j}\mathbb{I}_{\{-t_0-B_j<\psi<-t_0\}}|F_{t_0}|^2e^{-\varphi-\psi}dV_X<+\infty
 \label{3.3}
 \end{split}
 \end{equation}
where $b_{t_0,B_j}(t)=\int^{t}_{-\infty}\frac{1}{B_j} \mathbb{I}_{\{-t_0-B_j< s < -t_0\}}ds$ and
$v_j(t)=\int^{t}_{0}b_{t_0,B_j}(s)ds$.
\par
It follows from \eqref{3.3} that $\int_{\{\psi<-t_0-B_j\}}{|\widetilde F_j-(1-b_{t_0,B_j}(\psi))F_{t_0}|}^2 e^{{-}(\varphi+\psi)+v_j(\psi)}$ $c(-v_j(\psi))dV_X<+\infty$, and note that $e^{-t}c(t)$ is decreasing with respect to $t$ and $v_j(\psi)\ge \max\{\psi,-t_0-B_j\} \ge -t_0-1 $. Hence $e^{v_j(\psi)}c(-v_j(\psi))$ has positive lower bound, which implies
 \begin{equation}
\int_{\{\psi<-t_0-B_j\}}{|\widetilde F_j-(1-b_{t_0,B_j}(\psi))F_{t_0}|}^2 e^{-(\varphi+\psi)}dV_X<+\infty
\end{equation}
\par
As $\{\psi<-t_0-B_j\}$ is open, there exists an open subset $U'_j \subset (\{\psi<-t_0-B_j\}\cap U)$ such that $(\widetilde F_j- F_{t_0}) \in H^0(\{\psi<-t_0\}\cap U'_j,\mathcal{O} (K_X) \otimes \mathcal{I}(\varphi+\psi)) \subset H^0(\{\psi<-t_0\}\cap U'_j,\mathcal{O} (K_X) \otimes \mathcal{F})$, which implies
$(\widetilde F_j- f) \in H^0(\{\psi<-t_0\}\cap (U'_j\cap U'_{t_0}),\mathcal{O} (K_X) \otimes \mathcal{F})$.
\par
We have already prove in Lemma \ref{lemma2.6} that $\int_X |\tilde{F}_j|^2e^{-\varphi}c(-\psi)dV_X$ is uniformly bounded with respect to $j$.\par
As $e^{-\varphi}c(-\psi)$ has positive lower bound on any compact subset $K$ of $X$, (by diagonal method) there exist a subsequence of $\{\tilde{F}_j\}$ (also denoted by $\{\tilde{F}_j\}$) compact convergence to a holomorphic $(n,0)$ form $\tilde{F}_0$ (when $j \to +\infty$) on $X$. Since $(\widetilde F_j- f) \in H^0(\{\psi<-t_0\}\cap (U'_j\cap U'_{t_0}),\mathcal{O} (K_X) \otimes \mathcal{F})$, it follows from Lemma \ref{lemma2.2} that there exists an open set $\tilde{U}'$ which satisfies $Z_0\subset \tilde{U}' \subset U$ such that $(\widetilde F_0- f) \in H^0(\{\psi<-t_0\}\cap \tilde{U}' ,\mathcal{O} (K_X) \otimes \mathcal{F})$.
\par
It follows from \eqref{3.3} that
\begin{flalign}
&\int_X {|\tilde{F}_0-(1-b_{t_0}(\psi))F_{t_0}|}^2 e^{-\varphi}e^{-\psi+v_{t_0}(\psi)}c(-v_{t_0}(\psi))dV_X\\ \nonumber
\le &\liminf_{j \to +\infty} \int_X {|\widetilde F_j-(1-b_{t_0,B_j}(\psi))F_{t_0}|}^2 e^{{-}(\varphi+\psi)+v_j(\psi)}c(-v_j(\psi))dV_X \\ \nonumber
\le & \liminf_{j \to +\infty}\int_T^{t_0+B_j}c(t)e^{-t}dt\int_X \frac{1}{B_j}\mathbb{I}_{\{-t_0-B_j<\psi<-t_0\}}|F_{t_0}|^2e^{-\varphi-\psi}dV_X\\\nonumber
\le &\liminf_{j \to +\infty}\frac{e^{t_0+B_j\int_T^{t_0+B_j}c(t)e^{-t}dt}}{\inf_{t\in (t_0,t_0+B_j)}c(t)}\int_X \frac{1}{B_j}\mathbb{I}_{\{-t_0-B_j<\psi<-t_0\}}|F_{t_0}|^2e^{-\varphi}c(-\psi)dV_X\\\nonumber
\le &\liminf_{j \to +\infty}\frac{e^{t_0+B_j}\int_T^{t_0+B_j}c(t)e^{-t}dt}{\inf_{t\in (t_0,t_0+B_j)}c(t)}\cdot \frac{H(t_0)-H(t_0+B_j)}{B_j}\\\nonumber
=&\frac{\int_T^{t_0+B_j}c(t)e^{-t}dt}{c(t_0)e^{-t_0}}\liminf_{B\to 0+0}\frac{H(t_0)-H(t_0+B)}{B}
\label{3.5}
\end{flalign}
 the first $``\le"$ holds because of Fatou Lemma, where $b_{t_0}(t)=\mathbb{I}_{\{t \ge -t_0\}}$ and $v_{t_0}(t)=\int_0^t b_{t_0}(s)ds$. Note that $1-b_{t_0}(\psi)=\mathbb{I}_{\{\psi < -t_0\}}$. \par
 Note that $v_{t_0}(\psi)\ge \psi$ and $c(t)e^{-t}$ is decreasing with respect to $t$, then $e^{-\psi+v_{t_0}(\psi)}c(-v_{t_0}(\psi))\ge c(-\psi)$ holds on $X$. Hence we have
\begin{equation}
\begin{split}
&\int_X {|\tilde{F}_0-(1-b_{t_0}(\psi))F_{t_0}|}^2e^{-\varphi} c(-\psi)dV_X\\
\le &\int_X {|\tilde{F}_0-(1-b_{t_0}(\psi))F_{t_0}|}^2 e^{-\varphi}e^{-\psi+v_{t_0}(\psi)}c(-v_{t_0}(\psi))dV_X\\
\le &\frac{\int_T^{t_0+B_j}c(t)e^{-t}dt}{c(t_0)e^{-t_0}}\liminf_{B\to 0+0}\frac{H(t_0)-H(t_0+B)}{B}
\label{5.6}
\end{split}
\end{equation}
However,
\begin{flalign}\nonumber
 &\int_X {|\tilde{F}_0-\mathbb{I}_{\{\psi < -t_0\}}F_{t_0}|}^2e^{-\varphi} c(-\psi)dV_X\\ \nonumber
 =&\int_{\{\psi\ge-t_0\}} {|\tilde{F}_0|}^2e^{-\varphi} c(-\psi)dV_X+\int_{\{\psi<- t_0\}} {|\tilde{F}_0-F_{t_0}|}^2e^{-\varphi} c(-\psi)dV_X\\ \nonumber
 =&\int_{\{\psi\ge -t_0\}} {|\tilde{F}_0|}^2e^{-\varphi} c(-\psi)dV_X+\int_{\{\psi<- t_0\}} {|\tilde{F}|}^2e^{-\varphi} c(-\psi)dV_X-\int_{\{\psi<- t_0\}} {|F_{t_0}|}^2e^{-\varphi} c(-\psi)dV_X \\ \nonumber
 =&\int_{X} {|\tilde{F}_0|}^2e^{-\varphi} c(-\psi)dV_X-\int_{\{\psi<-t_0\}} {|F_{t_0}|}^2e^{-\varphi} c(-\psi)dV_X \\
 \ge &H(T)-H(t_0)
 \label{3.7}
\end{flalign}
Combining with \eqref{5.6} and \eqref{3.7}, we have
\begin{flalign}\nonumber
 &H(T)-H(t_0)\\\nonumber
 \le &\int_{X} {|\tilde{F}_0|}^2e^{-\varphi} c(-\psi)dV_X-\int_{\{\psi<-t_0\}} {|F_{t_0}|}^2e^{-\varphi} c(-\psi)dV_X \\\nonumber
 = &\int_X {|\tilde{F}_0-\mathbb{I}_{\{\psi < -t_0\}}F_{t_0}|}^2 e^{-\varphi}c(-\psi)dV_X\\\nonumber
\le &\int_X {|\tilde{F}_0-\mathbb{I}_{\{\psi < -t_0\}}F_{t_0}|}^2 e^{-\varphi}e^{-\psi+v_{t_0}(\psi)}c(-v_{t_0}(\psi))dV_X\\
\le &\frac{\int_T^{t_0+B_j}c(t)e^{-t}dt}{c(t_0)e^{-t_0}}\liminf_{B\to 0+0}\frac{H(t_0)-H(t_0+B)}{B}
 \label{3.8}
\end{flalign}
As $H(h^{-1}(r))$ is linear with respect to $r$, hence $$\frac{H(T)-H(t_0)}{\int_T^{t_0+B_j}c(t)e^{-t}dt}=\frac{\liminf_{B\to 0+0}\frac{H(t_0)-H(t_0+B)}{B}}{c(t_0)e^{-t_0}}$$, then all $``\le"$ in \eqref{3.8}
should be $``="$, i.e.
\begin{equation}
\begin{split}
 &H(T)-H(t_0)\\
= &\int_{X} {|\tilde{F}_0|}^2e^{-\varphi} c(-\psi)dV_X-\int_{\{\psi<-t_0\}} {|F_{t_0}|}^2e^{-\varphi} c(-\psi)dV_X \\
 = &\int_X {|\tilde{F}_0-\mathbb{I}_{\{\psi < -t_0\}}F_{t_0}|}^2e^{-\varphi} c(-\psi)dV_X\\
= &\int_X {|\tilde{F}_0-\mathbb{I}_{\{\psi < -t_0\}}F_{t_0}|}^2 e^{-\varphi}e^{-\psi+v_{t_0}(\psi)}c(-v_{t_0}(\psi))dV_X\\
= &\frac{\int_T^{t_0+B_j}c(t)e^{-t}dt}{c(t_0)e^{-t_0}}\liminf_{B\to 0+0}\frac{H(t_0)-H(t_0+B)}{B}
\label{5.9}
\end{split}
\end{equation}
It follows from the first $``="$ in \eqref{5.9} and $H(t_0)=\int_{\{\psi< t_0\}} {|F_{t_0}|}^2e^{-\varphi} c(-\psi)dV_X$ that
$$H(T)=\int_{X} {|\tilde{F}_0|}^2e^{-\varphi} c(-\psi)dV_X$$
It follows from $c(-\psi)=e^{-\psi+v_{t_0}(\psi)}c(-v_{t_0}(\psi))$ on $\{\psi\ge t_0\}$ and
\begin{equation}\nonumber
\begin{split}
&\int_X {|\tilde{F}_0-\mathbb{I}_{\{\psi < -t_0\}}F_{t_0}|}^2e^{-\varphi} c(-\psi)dV_X\\
= &\int_X {|\tilde{F}_0-\mathbb{I}_{\{\psi < -t_0\}}F_{t_0}|}^2 e^{-\varphi}e^{-\psi+v_{t_0}(\psi)}c(-v_{t_0}(\psi))dV_X
\end{split}
\end{equation}
that
\begin{equation}
\begin{split}
\label{th1.33.10}
&\int_{\{\psi<-t_0\}} {|\tilde{F}_0-F_{t_0}|}^2e^{-\varphi} c(-\psi)dV_X\\
= &\int_{\{\psi<-t_0\}}{|\tilde{F}_0-F_{t_0}|}^2 e^{-\varphi}e^{-\psi+v_{t_0}(\psi)}c(-v_{t_0}(\psi))dV_X
\end{split}
\end{equation}
Note that, on  $\{\psi< t_0\}$,$$c(-\psi)<e^{-\psi+v_{t_0}(\psi)}c(-v_{t_0}(\psi))$$ and the integrand in \eqref{th1.33.10} is nonnegative, we must have $\tilde{F}_0|_{\{\psi<-t_0\}}=F_{t_0}$.
\par
Theorem \ref{theorem1.3} is proved.
\end{proof}
\subsection{Proof of Corollary \ref{corollary1.4}}
\par
To prove Corollary 1.4, we need the following Propositions.
\begin{Proposition}
If $H(h^{-1}(r);c)$ is linear with respect to $r \in
(0,\int^{+\infty}_T c(t)e^{-t}dt]$. Let $t_0 \ge T$ be given. Let $\tilde{F}$ be a holomorphic $(n,0)$ form on $\{\psi<-t_0\}$ which satisfies $\tilde{F}\neq F|_{\{\psi<-t_0\}}$, $(\tilde{F}-f)\in H^0(\tilde{U}'\cap\{\psi<-t_0\},K_M \otimes \mathcal{F})$, where $\tilde{U}'$ is an open subset of $X$ satisfies $Z_0 \subset \tilde{U}' \subset U$ and $\int_{\{\psi<-t_0\}} c(-\psi)|\tilde{F}|^2e^{-\varphi}dV_X<+\infty$. Then for any $t_0\le t_1 < t_2 \le +\infty$, we have
$$\int_{\{-t_2\le\psi<-t_1\}} c(-\psi)|\tilde{F}|^2e^{-\varphi}dV_X
>\int_{\{-t_2\le\psi<-t_1\}} c(-\psi)|F|^2e^{-\varphi}dV_X$$
\label{observation}
\end{Proposition}
\begin{proof}
when $t_2=+\infty$, it follows form Lemma \ref{existence of F} that
\begin{equation}\nonumber
\begin{split}
&\int_{\{\psi<-t_1\}} c(-\psi)|\tilde{F}|^2e^{-\varphi}dV_X-
\int_{\{\psi<-t_1\}}c(-\psi)|F|^2e^{-\varphi}dV_X\\
=&\int_{\{\psi<-t_1\}} c(-\psi)|\tilde{F}-F|^2e^{-\varphi}dV_X
\end{split}
\end{equation}
As $\tilde{F}-F\neq 0$ on $\{\psi<-t_1\}$, the zero set of $\tilde{F}-F$ (denoted by $Z(\tilde{F}-F)$) is an analytic set of $\{\psi<-t_1\}$ and the measure of $Z(\tilde{F}-F)$ is zero. Then
$$\int_{\{\psi<-t_1\}} c(-\psi)|\tilde{F}-F|^2e^{-\varphi}dV_X>0,$$
hence
$$\int_{\{\psi<-t_1\}} c(-\psi)|\tilde{F}|^2e^{-\varphi}dV_X>
\int_{\{\psi<-t_1\}}c(-\psi)|F|^2e^{-\varphi}dV_X$$
\par
When $t_0\le t_1 <t_2<+\infty$, we have
\begin{flalign}\nonumber
&\int_{\{t_2\le\psi<-t_1\}} c(-\psi)|\tilde{F}|^2e^{-\varphi}dV_X-
\int_{\{t_2\le\psi<-t_1\}}c(-\psi)|F|^2e^{-\varphi}dV_X\\\nonumber
=&\int_{\{\psi<-t_1\}} c(-\psi)|\tilde{F}|^2e^{-\varphi}dV_X-\int_{\{\psi<-t_2\}} c(-\psi)|\tilde{F}|^2e^{-\varphi}dV_X\\\nonumber
-&(\int_{\{\psi<-t_1\}} c(-\psi)|F|^2e^{-\varphi}dV_X-\int_{\{\psi<-t_2\}} c(-\psi)|F|^2e^{-\varphi}dV_X)\\\nonumber
=&\int_{\{\psi<-t_1\}} c(-\psi)|\tilde{F}-F|^2e^{-\varphi}dV_X-\int_{\{\psi<-t_2\}} c(-\psi)|\tilde{F}-F|^2e^{-\varphi}dV_X\\\nonumber
=&\int_{\{-t_2\le\psi<-t_1\}} c(-\psi)|\tilde{F}-F|^2e^{-\varphi}dV_X
\end{flalign}
As $\tilde{F}-F\neq 0$ on $\{\psi<-t_1\}$, the zero set of $\tilde{F}-F$ is an analytic set of $\{\psi<-t_1\}$. Note that the measure of the set $\{-t_2\le t<t_1\}$ is positive and the measure of $Z(\tilde{F}-F)$ is zero, we know
$$\int_{\{-t_2\le\psi<-t_1\}} c(-\psi)|\tilde{F}-F|^2e^{-\varphi}dV_X>0,$$
hence
$$\int_{\{t_2\le\psi<-t_1\}} c(-\psi)|\tilde{F}|^2e^{-\varphi}dV_X>
\int_{\{t_2\le\psi<-t_1\}}c(-\psi)|F|^2e^{-\varphi}dV_X$$
\par
\end{proof}

\par

Now we begin to prove Corollary \ref{corollary1.4}.
\begin{proof}
~\\
\textbf{Step 1:}
\par
Given $t_2 \ge T$. It follows from Lemma \ref{existence of F} that there exists a holomorphic $(n,0)$ form $\tilde{F}$ on $\{\psi<-t_2\}$ such that $(\tilde{F}-f)\in H^0(\tilde{U}'\cap\{\psi<-t_2\},K_M \otimes \mathcal{I}(\psi+\varphi)|_U)$, where $\tilde{U}'$ is an open subset of $X$ satisfies $Z_0 \subset \tilde{U}' \subset U$ and
$$H(t_2;\tilde{c})=\int_{\{\psi<-t_2\}} \tilde{c}(-\psi)|\tilde{F}|^2e^{-\varphi}dV_X<+\infty$$
As $(\log\tilde{c}(t))'\ge (\log c(t))'$, we have $\tilde{c}(t)\ge M c(t)$ for some constant $M>0$.
It follows from $\int_{\{\psi<-t_2\}} \tilde{c}(-\psi)|\tilde{F}|^2e^{-\varphi}dV_X<+\infty$ that we have $$\int_{\{\psi<-t_2\}} c(-\psi)|\tilde{F}|^2e^{-\varphi}dV_X<+\infty.$$
\textbf{Step 2:}
\par
Denote $I(t)=\int_{\{\psi<-t\}} c(-\psi)|\tilde{F}|^2e^{-\varphi}dV_X$, where $t\ge t_2$. For any  $t_0>t_1\ge t_2$, Proposition \ref{observation} shows that
$$\int_{\{-t_0\le\psi<-t_1\}} c(-\psi)|\tilde{F}|^2e^{-\varphi}dV_X
\ge\int_{\{-t_0\le\psi<-t_1\}} c(-\psi)|F|^2e^{-\varphi}dV_X,$$
the equality holds if and only if $\tilde{F}=F|_{\{\psi<-t_2\}}$.
Hence we know
\begin{equation}
\frac{I(t_1)-I(t_0)}{\int_{t_1}^{t_0}c(t)e^{-t}dt}\ge \frac{H(t_1;c)-H(t_0;c)}{\int_{t_1}^{t_0}c(t)e^{-t}dt}=k,
\label{3.10}
\end{equation}
the equality holds if and only if $\tilde{F}=F|_{\{\psi<-t_2\}}$.
\par
Note that we also have
\begin{equation}
\begin{split}
H(t_2;\tilde{c})-H(t_1;\tilde{c})
\ge& \int_{\{-t_1\le\psi<-t_2\}} \tilde{c}(-\psi)|\tilde{F}|^2e^{-\varphi}dV_X\\
=&\sum_{i=1}^n \int_{\{-t_1+(i-1)\frac{t_1-t_2}{n}\le\psi<-t_1+i\frac{t_1-t_2}{n}\}} \frac{\tilde{c}(-\psi)}{c(-\psi)}c(-\psi)|\tilde{F}|^2e^{-\varphi}dV_X
\end{split}
\end{equation}
As $c(t)\in \mathcal{G}_T$, it follows from condition (2) and (3) of $\mathcal{G}_T$ that $c(t)\neq 0$ for any $t\ge T$. Then $\frac{\tilde{c}(-\psi)}{c(-\psi)}$ is uniformly continuous on $[t_2,t_1]$. When $n$ big enough, we have
\begin{equation}\nonumber
\begin{split}
H(t_2;\tilde{c})-H(t_1;\tilde{c})
\ge&\sum_{i=1}^n (\int_{\{-t_1+(i-1)\frac{t_1-t_2}{n}\le\psi<-t_1+i\frac{t_1-t_2}{n}\}} c(-\psi)|\tilde{F}|^2e^{-\varphi}dV_X)\times\\
&(\frac{\tilde{c}(t_1-i\frac{t_1-t_2}{n})}{c(t_1-i\frac{t_1-t_2}{n})}-\epsilon)\\
=&S_{1,n}+S_{2,n}
\end{split}
\end{equation}
where
$$S_{1,n}=\sum_{i=1}^n (\int_{\{-t_1+(i-1)\frac{t_1-t_2}{n}\le\psi<-t_1+i\frac{t_1-t_2}{n}\}} c(-\psi)|\tilde{F}|^2e^{-\varphi}dV_X)
\frac{\tilde{c}(t_1-i\frac{t_1-t_2}{n})}{c(t_1-i\frac{t_1-t_2}{n})},$$
and
$$S_{2,n}=-\epsilon\sum_{i=1}^n \int_{\{-t_1+(i-1)\frac{t_1-t_2}{n}\le\psi<-t_1+i\frac{t_1-t_2}{n}\}} c(-\psi)|\tilde{F}|^2e^{-\varphi}dV_X.
$$
It is easy to see that $\lim\limits_{n \to +\infty}S_{2,n}=0$. For $S_{1,n}$, we have
\begin{equation}
\begin{split}
S_{1,n}=&\sum_{i=1}^n \frac{I(t_1-i\frac{t_1-t_2}{n})-I(t_1-(i-1)\frac{t_1-t_2}{n})}
{\int_{t_1-i\frac{t_1-t_2}{n}}^{t_1-(i-1)\frac{t_1-t_2}{n}}c(t)e^{-t}dt}\times\\
&[\frac{\int_{t_1-i\frac{t_1-t_2}{n}}^{t_1-(i-1)\frac{t_1-t_2}{n}}c(t)e^{-t}dt}{c(t_1-i\frac{t_1-t_2}{n})e^{-t_1+i\frac{t_1-t_2}{n}}\frac{t_1-t_2}{n}}
\tilde{c}(t_1-i\frac{t_1-t_2}{n})e^{-t_1+i\frac{t_1-t_2}{n}}\frac{t_1-t_2}{n}]\\
\ge&\sum_{i=1}^nk[\frac{\int_{t_1-i\frac{t_1-t_2}{n}}^{t_1-(i-1)\frac{t_1-t_2}{n}}c(t)e^{-t}dt}{c(t_1-i\frac{t_1-t_2}{n})e^{-t_1+i\frac{t_1-t_2}{n}}\frac{t_1-t_2}{n}}
\tilde{c}(t_1-i\frac{t_1-t_2}{n})e^{-t_1+i\frac{t_1-t_2}{n}}\frac{t_1-t_2}{n}]\\
\label{5.11}
\end{split}
\end{equation}
The $``\ge"$ holds because of \eqref{3.10}.
Let $n \to +\infty$ in \eqref{5.11} we have
$\lim\limits_{n \to +\infty}S_{1,n}\ge k\int_{t_2}^{t_1}\tilde{c}(t)e^{-t}dt$. Hence we have
$$H(t_2;\tilde{c})-H(t_1;\tilde{c})\ge k \int_{t_2}^{t_1}\tilde{c}(t)e^{-t}dt$$
i.e.
$$\frac{H(t_2;\tilde{c})-H(t_1;\tilde{c})}{\int_{t_2}^{t_1}\tilde{c}(t)e^{-t}dt}\ge k.$$
Let $t_1 \to +\infty$, then
\begin{equation}
\frac{H(t_2;\tilde{c})}{\int_{t_2}^{+\infty}\tilde{c}(t)e^{-t}dt}\ge k
\label{5.12}
\end{equation}
\par
Recall that
 $F$ is the holomorphic $(n,0)$ form on $X$ such that $H(t;c)=\int_{\{\psi<-t\}} c(-\psi)|F|^2e^{-\varphi}dV_X$, for any $t\ge T$.
Let $T\le t_2<t_1<+\infty$, we have
$$\frac{\int_{\{-t_1\le \psi<-t_2\}}c(-\psi)|F|^2e^{-\varphi}dV_X}{\int_{t_2}^{t_1}c(t)e^{-t}dt}=k$$
\par
Note that
\begin{equation}
\begin{split}
&\int_{\{-t_1\le\psi<-t_2\}}\tilde{c}(-\psi)|F|^2e^{-\varphi}dV_X\\
=&\sum_{i=1}^n \int_{\{-t_1+(i-1)\frac{t_1-t_2}{n}\le\psi<-t_1+i\frac{t_1-t_2}{n}\}} \tilde{c}(-\psi)|F|^2e^{-\varphi}dV_X
\label{5.13}
\end{split}
\end{equation}
Let $n$ be big enough, the right hand side of \eqref{5.13} is bounded by
\begin{equation}
\begin{split}
&\sum_{i=1}^n (\int_{\{-t_1+(i-1)\frac{t_1-t_2}{n}\le\psi<-t_1+i\frac{t_1-t_2}{n}\}} c(-\psi)|F|^2e^{-\varphi}dV_X)(\frac{\tilde{c}(t_1-i\frac{t_1-t_2}{n})}{c(t_1-i\frac{t_1-t_2}{n})}\pm \epsilon)\\
=&\sum_{i=1}^n k\int_{t_1-i\frac{t_1-t_2}{n}}^{t_1-(i-1)\frac{t_1-t_2}{n}}c(t)e^{-t}dt(\frac{\tilde{c}(t_1-i\frac{t_1-t_2}{n})}{c(t_1-i\frac{t_1-t_2}{n})}\pm \epsilon)\\
=&\sum_{i=1}^n [k \frac{\int_{t_1-i\frac{t_1-t_2}{n}}^{t_1-(i-1)\frac{t_1-t_2}{n}}c(t)e^{-t}dt}{c(t_1-i\frac{t_1-t_2}{n})e^{-t_1+i\frac{t_1-t_2}{n}}\frac{t_1-t_2}{n}}
(\tilde{c}(t_1-i\frac{t_1-t_2}{n})e^{-t_1+i\frac{t_1-t_2}{n}}\frac{t_1-t_2}{n})\pm \\
& k\epsilon\int_{t_1-i\frac{t_1-t_2}{n}}^{t_1-(i-1)\frac{t_1-t_2}{n}}c(t)e^{-t}dt]
\label{5.14}
\end{split}
\end{equation}
When $n \to +\infty$, combining \eqref{5.13} and \eqref{5.14}, we have
\begin{equation}
\begin{split}
\int_{\{-t_1\le\psi<-t_2\}}\tilde{c}(-\psi)|F|^2e^{-\varphi}dV_X=k\int_{t_2}^{t_1}\tilde{c}(t)e^{-t}dt
\label{5.15}
\end{split}
\end{equation}
Let $t_1$ goes to $+\infty$ in \eqref{5.15}, we know
\begin{equation}\nonumber
\begin{split}
\frac{\int_{\{\psi<-t_2\}}\tilde{c}(-\psi)|F|^2e^{-\varphi}dV_X}
{\int_{t_2}^{+\infty}\tilde{c}(t)e^{-t}dt}=k
\end{split}
\end{equation}
Hence
\begin{equation}
\begin{split}
\frac{H(t_2;\tilde{c})}{\int_{t_2}^{+\infty}\tilde{c}(t)e^{-t}dt}\le
\frac{\int_{\{\psi<-t_2\}}\tilde{c}(-\psi)|F|^2e^{-\varphi}dV_X}
{\int_{t_2}^{+\infty}\tilde{c}(t)e^{-t}dt}=k
\label{5.16}
\end{split}
\end{equation}
It follows from \eqref{5.12} and \eqref{5.16} that for any $t_2\ge T$,
$$\frac{H(t_2;\tilde{c})}{\int_{t_2}^{+\infty}\tilde{c}(t)e^{-t}dt}=k$$
holds, i.e. $H(h^{-1}_{\tilde{c}}(r);\tilde{c})$ is linear with respect to $r$. Hence there exists a holomorphic $(n,0)$ form $F_{\tilde{c}}$ on X such that
$$H(t_2;\tilde{c})=\int_{\{\psi<-t_2\}}\tilde{c}(-\psi)|F_{\tilde{c}}|^2e^{-\varphi}dV_X$$
and we also have
\begin{equation}
\frac{H(t_2;\tilde{c})-H(t_1;\tilde{c})}{\int_{t_2}^{t_1}\tilde{c}(t)e^{-t}dt}=
\frac{\int_{\{-t_1\le\psi<-t_2\}}\tilde{c}(-\psi)|F_{\tilde{c}}|^2e^{-\varphi}dV_X}{\int_{t_2}^{t_1}\tilde{c}(t)e^{-t}dt}
\label{5.17}
\end{equation}
If $F_{\tilde{c}}\neq F$ on $X$, it follows from Proposition \ref{observation}, \eqref{5.15} and \eqref{5.17} that
$$k=\frac{\int_{\{-t_1\le\psi<-t_2\}}\tilde{c}(-\psi)|F|e^{-\varphi}dV_X}{\int_{t_2}^{t_1}\tilde{c}(t)e^{-t}dt}
>\frac{\int_{\{-t_1\le\psi<-t_2\}}\tilde{c}(-\psi)|F_{\tilde{c}}|^2e^{-\varphi}dV_X}{\int_{t_2}^{t_1}\tilde{c}(t)e^{-t}dt}=k,$$
which is a contradiction. Hence we must have $F_{\tilde{c}}=F$ on $X$.
\par
Corollary  \ref{corollary1.4} is proved.
\end{proof}
\subsection{Proof of Corollary \ref{corollary1.5} and Corollary \ref{corollary1.6}}
In this section, we prove Corollary \ref{corollary1.5} and Corollary \ref{corollary1.6}.
\begin{proof}We prove the Corollary \ref{corollary1.5} by contradiction.
\par
Assume $H(h^{-1}(r);\varphi)$ is linear with respect to $r\in (0,\int_T^{+\infty}c(t)e^{-t}dt)]$. Then it follows from Theorem \ref{theorem1.3} that there exists a holomorphic $(n,0)$ form $F$ on $X$ such that
$$H(t;\varphi)=\int_{\{\psi<-t\}}c(-\psi)|F|^2e^{-\varphi}dV_X$$
holds for any $t \ge T$.
\par
Denote
\begin{equation}\nonumber
\begin{split}
\inf\{\int_{ \{ \psi<-t\}}&c(-\psi)|\tilde{f}|^2e^{-\tilde{\varphi}}dV_X: \tilde{f}\in
H^0(\{\psi<-t\},\mathcal{O} (K_X)  ), \\
\& &\  \exists \;open\; set\; U' \;s.t\; Z_0 \subset U' \subset U\ and \\
& \; (\tilde{f}-f)\in
H^0(\{\psi<-t\} \cap U' ,\mathcal{O} (K_X) \otimes \mathcal{F}) \}
\end{split}
\end{equation}
by $H(t;\tilde{\varphi})$. As $e^{-\tilde{\varphi}}\le e^{-\varphi}$, we know $H(T;\tilde{\varphi})<+\infty$.
\par
Let $C_2>t_1>t_2\ge T$. It follows from Lemma \ref{existence of F} that there exists a holomorphic $(n,0)$ form $\tilde{F}_{t_2}$ on $\{\psi<-t_2\}$ such that
$$H(t_2;\tilde{\varphi})=\int_{\{\psi<-t_2\}}c(-\psi)|\tilde{F}_{t_2}|^2e^{-\tilde{\varphi}}dV_X<+\infty.$$
As $\tilde{\varphi}-\varphi$ is bounded on $X$, we have
$$H(t_2;\tilde{\varphi})=\int_{\{\psi<-t_2\}}c(-\psi)|\tilde{F}_{t_2}|^2e^{-\varphi}dV_X<+\infty.$$
 Note that on $\{-t_1\le \psi <-t_2\}\subset \{\psi\ge -C_2\}$, we have $\tilde{\varphi}=\varphi$, hence
\begin{equation}
\begin{split}
H(t_2;\tilde{\varphi})-H(t_1;\tilde{\varphi})
\ge&\int_{\{-t_1\le\psi<-t_2\}}c(-\psi)|\tilde{F}_{t_2}|^2e^{-\tilde{\varphi}}dV_X\\
\ge&\int_{\{-t_1\le\psi<-t_2\}}c(-\psi)|F|^2e^{-\tilde{\varphi}}dV_X\\
=&H(t_2;\varphi)-H(t_1;\varphi)
\label{formula1}
\end{split}
\end{equation}
The second inequality holds because of Proposition \ref{observation}. It follows from \eqref{formula1} that
\begin{equation}
\begin{split}
\frac{H(t_2;\tilde{\varphi})-H(t_1;\tilde{\varphi})}{\int_{t_2}^{t_1}c(t)e^{-t}dt}
\ge\frac{H(t_2;\varphi)-H(t_1;\varphi)}{\int_{t_2}^{t_1}c(t)e^{-t}dt}=k
\label{formula2}
\end{split}
\end{equation}
Let $t_2=T$, it follows from Theorem 1.1 and note that $\tilde{\varphi}\ge \varphi$ on $X$, we have
\begin{equation}
\begin{split}
\frac{H(T;\tilde{\varphi})-H(t_1;\tilde{\varphi})}{\int_{T}^{t_1}c(t)e^{-t}dt}
\le\frac{H(T;\tilde{\varphi})}{\int_T^{+\infty}c(t)e^{-t}dt}\le\frac{H(T;\varphi)}{\int_T^{+\infty}c(t)e^{-t}dt}=k
\label{formula3}
\end{split}
\end{equation}
It follows from \eqref{formula2} and \eqref{formula3} that
$$\frac{H(T;\tilde{\varphi})}{\int_T^{+\infty}c(t)e^{-t}dt}=k.$$
\par
Let $t_3$ be big enough such that $\{\psi<-t_3\} \subset \{\psi<-C_1\}$. Then, on $\{\psi<-t_3\}$, we have $\tilde{\varphi}=\varphi$. When $t \ge t_3$, we have $H(t;\tilde{\varphi})=H(t;\varphi)$ and
$$\frac{H(t;\tilde{\varphi})}{\int_t^{+\infty}c(t)e^{-t}dt}=
\frac{H(t;\varphi)}{\int_t^{+\infty}c(t)e^{-t}dt}=k.$$
\par
Recall that $\frac{H(T;\tilde{\varphi})}{\int_T^{+\infty}c(t)e^{-t}dt}=k$, we know $\lim\limits_{t \to +\infty}\frac{H(t;\tilde{\varphi})}{\int_t^{+\infty}c(t)e^{-t}dt}=\frac{H(T;\tilde{\varphi})}{\int_T^{+\infty}c(t)e^{-t}dt}=k$, then we know $H(-\log r;\tilde{\varphi})$ is linear with respect to $r$. Then there exist a holomorphic $(n,0)$ form $\tilde{F}$ on $X$ such that for any $t\ge T$, we have
$$H(t;\tilde{\varphi})=\int_{\{\psi<-t\}}c(-\psi)|\tilde{F}|^2e^{-\tilde{\varphi}}dV_X.$$
When $t_0$ big enough such that $\tilde{\varphi}=\varphi$ on $\{\psi<-t_0\}$, then
$H(t_0,\tilde{\varphi})=H(t_0,\varphi)$, hence we have (note that $\tilde{\varphi}=\varphi$)
$$\int_{\{\psi<-t_0\}}c(-\psi)|\tilde{F}|^2e^{-\tilde{\varphi}}dV_X
=H(t_0,\tilde{\varphi})=H(t_0,\varphi)=\int_{\{\psi<-t_0\}}c(-\psi)|F|^2e^{-\varphi}dV_X$$
which (by Proposition \ref{observation}) implies $\tilde{F}=F$ on $\{\psi<-t_0\}$. Note that $\{\psi<-t_0\}$ is an open subset of $X$, $\tilde{F}$ and $F$ are holomorphic $(n,0)$ form on $X$, it follows from $\tilde{F}=F$ on $\{\psi<-t_0\}$ that $\tilde{F}=F$ on $X$.
\par
However $e^{-\varphi}>e^{-\tilde{\varphi}}$ on $U \subset X$, we must have
$$k=\frac{H(T;\varphi)}{\int_T^{+\infty}c(t)e^{-t}dt}=\frac{\int_X |F|^2e^{-\varphi}dV_X}{\int_T^{+\infty}c(t)e^{-t}dt}
>\frac{\int_X |F|^2e^{-\tilde{\varphi}}dV_X}{\int_T^{+\infty}c(t)e^{-t}dt}=\frac{H(T;\tilde{\varphi})}{\int_T^{+\infty}c(t)e^{-t}dt}=k$$
This is a contradiction. Hence $H(h^{-1}(r);\varphi)$ can not be linear with respect to $r$. Corollary \ref{corollary1.5} is proved.
\end{proof}
To prove Corollary \ref{corollary1.6}, we only need to construction a function $\tilde{\varphi}$ on $X$ which satisfies the condition of Corollary \ref{corollary1.5}.
\begin{proof}
As $\varphi+\psi$ is strictly plurisubharmonic at $z_0$, we can find a small open neighborhood $(U,z)$ of $z_0$ and $z=(z_1,\ldots,z_n)$ is the local coordinate on $U$ such that $i\partial \bar{\partial} (\varphi+\psi)>\epsilon \omega$ for some $\epsilon>0$, where $\omega=i\sum\limits_{i=1}^n dz_i\wedge d \bar{z}_i$ under the local coordinate on $U$. By shrinking $U$, we also assume that $U\subset\subset X$. Take $z_1\in U$, $z_1 \notin \{\psi=-\infty\}$, then we can choose an open subset $V$ such that $z_1\in V$ and $V$ satisfies
\par
(1) $V\subset\subset U$,
\par
(2) $V \cap \{\psi=-\infty\}=\emptyset$.
\\
Let $\rho$ be a smooth nonnegative function on $X$ which satisfies $\rho \equiv 1$ on $W\subset V$ and  $supp \rho \subset\subset V$. Let $\delta$ be a small positive constant such that
$$i\partial \bar{\partial} (\varphi+\psi)+i\partial \bar{\partial}(\delta \rho)>\frac{\epsilon}{2} \omega$$ on $V$. Let $\tilde{\varphi}=\varphi+\delta \rho$, note that $0\le\delta \rho\le\delta$ is a smooth function, then $\tilde{\varphi}$ satisfies
\par
(1) $\tilde{\varphi}+\psi$ is plurisubharmonic function on $X$.
\par
(2) $\tilde{\varphi} > \varphi$ on $W$ and $\tilde{\varphi}=\varphi$ on $X\backslash U$.
\par
(3) $\tilde{\varphi}-\varphi$ is bounded on $X$.
\par
It is  easy to see that the function $\tilde{\varphi}$ satisfies the conditions $(1),(2),(3)$ in Corollary \ref{corollary1.5}. Then it follows Corollary \ref{corollary1.5} that $H(h^{-1}(r);\varphi)$ can not be linear with respect to $r$.
\par
Corollary \ref{corollary1.6} is proved.
\end{proof}

\subsection{Proof of Theorem \ref{corollary1.7}}
Let $c(t)\in \mathcal{G}_0$. Let $Z_0=Y$. Let $\hat{f}$ be a holomorphic extension of $f$ from $Y$ to $U$, where  $U \supset Y$ is an open subset of X.
Let $\mathcal{F}= \mathcal{I}(\psi)|_U$ on $U$.\par
Define
\begin{equation}
\begin{split}
H(t):=\inf\{\int_{ \{ \psi<-t\}}|\tilde{f}|^2e^{-\varphi}c(-\psi)dV_X: &\tilde{f}\in
H^0(\{\psi<-t\},\mathcal{O} (K_X)  ) \\
\& & \exists \;open\; set\; U' \;s.t\; Z_0 \subset U' \subset U \\
&pand \; (\tilde{f}-\hat{f})\in
H^0(\{\psi<-t\} \cap U' ,\mathcal{O} (K_X) \otimes \mathcal{I}(\psi)) \}
\end{split}
\end{equation}
 It follows from condition \eqref{1.22} and \eqref{1.23} that
$$\int_{X} c(-\psi)|F|^2e^{-\varphi}dV_X=H(0).$$
The optimal $L^2$ extension theorem in \cite{GZsci} shows that
$$
\int_{\{\psi<-t\}} c(-\psi)|F_t|^2e^{-\varphi}dV_X\le(\int_t^{+\infty}c(t)e^{-t}dt)\frac{\pi^k}{k!}\int_Y |f|^2e^{-\varphi}dV_X[\psi]
$$
holds for any $t \in [0,+\infty)$, where $F_t$ is a holomorphic extension of $f$ from $Y$ to $\{\psi<-t\}$. Note that by the definition of $H(t)$, we have
$$
H(t)\le\int_{\{\psi<-t\}}c(-\psi) |F_t|^2e^{-\varphi}dV_X
$$
Theorem \ref{maintheorem} implies that
$$
\int_{X} |F|^2e^{-\varphi}dV_X=H(0)\le \frac{\int_0^{+\infty}c(t)e^{-t}dt}{\int_t^{+\infty}c(t)e^{-t}dt}H(t)
$$

Now we have
\begin{equation}
\begin{split}
H(0)=&\int_{X} c(-\psi)|F|^2e^{-\varphi}dV_X\\
\le &\frac{\int_0^{+\infty}c(t)e^{-t}dt}{\int_t^{+\infty}c(t)e^{-t}dt}H(t)\\
\le & \frac{\int_0^{+\infty}c(t)e^{-t}dt}{\int_t^{+\infty}c(t)e^{-t}dt}\int_{\{\psi<-t\}} |F_t|^2e^{-\varphi}dV_X\\
\le& (\int_0^{+\infty}c(t)e^{-t}dt)\frac{\pi^k}{k!}\int_Y |f|^2e^{-\varphi}dV_X[\psi]
\label{311}
\end{split}
\end{equation}
holds for any $t \in [0,+\infty)$.
Recall that $F$ satisfies
$$\int_{X} c(-\psi)|F|^2e^{-\varphi}dV_X=(\int_0^{+\infty}c(t)e^{-t}dt)\frac{\pi^k}{k!}\int_Y |f|^2e^{-\varphi}dV_X[\psi]$$
Hence all $``\le"$ in \eqref{311} should be $``="$, i.e.
\begin{equation}
\begin{split}
H(0)=&\int_{X} c(-\psi)|F|^2e^{-\varphi}dV_X\\
= &\frac{\int_0^{+\infty}c(t)e^{-t}dt}{\int_t^{+\infty}c(t)e^{-t}dt}H(t)\\
=& \frac{\int_0^{+\infty}c(t)e^{-t}dt}{\int_t^{+\infty}c(t)e^{-t}dt}\int_{\{\psi<-t\}} |F_t|^2e^{-\varphi}dV_X\\
=& (\int_0^{+\infty}c(t)e^{-t}dt)\frac{\pi^k}{k!}\int_Y |f|^2e^{-\varphi}dV_X[\psi]
\label{312}
\end{split}
\end{equation}
holds for any $t \in [0,+\infty)$.
Especially,
$$\frac{\int_{\{\psi<-t\}} |F_t|^2e^{-\varphi}dV_X}{\int_t^{+\infty}c(t)e^{-t}dt}=\frac{\pi^k}{k!}\int_Y |f|^2e^{-\varphi}dV_X[\psi]$$ holds. It follows from Theorem \ref{theorem1.3} that $F|_{\{\psi<-t\}}=F_t$.
\par
Theorem \ref{corollary1.7} is proved

\subsection{Proof of Corollary \ref{corollary1.8}}
\
\par
It is easy to see that (2) implies (1).
\par
Now we assume that the statement (1) holds. It follows from Corollary \ref{corollary1.2} that $H(-\log r)$ is linear with respect to $r$, i.e. $\frac{K_{D_t}(0,0)}{K_{D}(0,0)}=e^t$ holds for any $t \in [0,+\infty)$. Now we only need to show that the linearity of $H(-\log r)$ implies (2).
\par
 It is known that $\frac{K_{D_t}(z,0)}{K_{D_t}(0,0)}$ satisfies
$\int_{D_t} |\frac{K_{D_t}(z,0)}{K_{D_t}(0,0)}|^2 d\lambda_n=H(t)$, where $d\lambda_n$ is the Lebesgue measure on $\mathbb{C}^n$.
It follows from Theorem \ref{theorem1.3} that the linearity of $H(-\log r)$ implies $\frac{K_{D}(z,0)}{K_{D}(0,0)}|_{D_t}=\frac{K_{D_t}(z,0)}{K_{D_t}(0,0)}$. As $\frac{K_{D_t}(0,0)}{K_{D}(0,0)}=e^t$ holds for any $t \in [0,+\infty)$, we have $\frac{K_{D_t}(z,0)}{K_{D}(z,0)}=e^t$ holds for any $t \in [0,+\infty)$ and any $z \in D_t$.
\par
Corollary \ref{corollary1.8} is proved.

\subsection{Proof of Theorem \ref{corollary1.9} and Theorem \ref{corollary1.11}}

We firstly discuss some property of $H(t;c,\varphi)$.
\par
Recall that $X$ is an open Riemann Surface which admits a nontrivial Green function $G_{X}(z,w)$.
\par
Let $\psi=kG_{X}(z,z_0)$, where $z_0$ is a point of $X$ and $k\ge 2$ is a real number.
\par
Let $U$ be a open neighborhood of $z_0$.
Let $f$ be a holomorphic $(n,0)$ form on  $V_{z_0}$.

 Let $\varphi$ be a subharmonic function on $X$. Let $c(t)\in \mathcal{G}_0$. Denote
\begin{equation}
\begin{split}
H(t;c,\varphi):=\inf\{\int_{ \{ \psi<-t\}}&c(-\psi)|\tilde{F}|^2e^{-\varphi}dV_X: \tilde{F}\in
H^0(\{\psi<-t\},\mathcal{O} (K_X)  ), \\
\& &\  \exists \;open\; set\; U' \;s.t\; Z_0 \subset U' \subset U\ and \\
& \; (\tilde{F}-f)\in
H^0(\{\psi<-t\} \cap U' ,\mathcal{O} (K_X) \otimes \mathcal{I}(\psi+\varphi)|_U) \}.
\end{split}
\end{equation}

We now consider the linearity of $H(h^{-1}(r);c,\varphi)$ with respect to $r$ for various $c \in \mathcal{G}_T$ and $c\in C^{\infty}[T,+\infty)$, where $h(t)=\int_t^{+\infty}c(t_1)e^{-t_1}dt_1$. We have the following result.
\begin{Proposition}Let $c \in C^{\infty}[T,+\infty)$ and $c \in \mathcal{G}_T$. If $H(T;c,\varphi)<+\infty$ and $H(h^{-1}(r);c,\varphi)$ is linear with respect to $r \in
(0,\int^{+\infty}_T c(t)e^{-t}dt]$. Let $F$ be the holomorphic $(n,0)$ form on $X$ such that $\int_{\{\psi<-t\}} c(-\psi)|F|^2e^{-\varphi}dV_X=H(t;c,\varphi)$ for any $t\ge T$.
Then for any other $\tilde{c} \in C^{\infty}[T,+\infty)$ and $\tilde{c}\in \mathcal{G}_T$, which satisfies $H(T;\tilde{c},\varphi)<+\infty$ we have
\begin{equation}
\begin{split}
\int_{\{\psi<-t\}} \tilde{c}(-\psi)|F|^2e^{-\varphi}dV_X=H(t;\tilde{c},\varphi)=&
\frac{H(T;\tilde{c},\varphi)}{\int^{+\infty}_T \tilde{c}(t_1)e^{-t_1}dt_1}\int^{+\infty}_t
\tilde{c}(t_1)e^{-t_1}dt_1\\
=&k\int^{+\infty}_t
\tilde{c}(t_1)e^{-t_1}dt_1
\end{split}
\end{equation}
holds for any $t\in [T,+\infty)$, where $k=\frac{H(T;c,\varphi)}{\int^{+\infty}_T c(t_1)e^{-t_1}dt_1}$.
\label{variousC}
\end{Proposition}
\begin{proof}
~\\
\textbf{Step 1:}
\par  Fix any $t_2\ge 0$,
 we firstly show that for any holomorphic $(n,0)$ form $F$ defined on $\{\psi<-t_2\}$ which satisfied $(\tilde{F}-f)\in
H^0(\{\psi<-t_2\} \cap U' ,\mathcal{O} (K_X) \otimes \mathcal{I}(\psi+\varphi)|_{U})$ for some open set $z_0\subset U' \subset U$ and
\begin{equation}
\int_{\{\psi<-t_2\}}c(-\psi)|F|^2e^{-\varphi}dV_X<+\infty.
\label{formula3.32}
\end{equation}
The follows inequality holds,
$$\int_{\{\psi<-t_2\}}\tilde{c}(-\psi)|F|^2e^{-\varphi}dV_X<+\infty.$$
\par
As $H(T;\tilde{c},\varphi)<+\infty$, it follows from Lemma \ref{existence of F} that there exists a holomorphic $(n,0)$ form $\tilde{F}$ on $\{\psi<-t_2\}$ which satisfies
$$(\tilde{F}-f)\in
H^0(\{\psi<-t_2\} \cap \bar{U} ,\mathcal{O} (K_X) \otimes \mathcal{I}(\psi+\varphi)|_U)$$
for some open set $\bar{U}$ such that $z_0\subset \bar{U} \subset U$ and
$$H(t_2;\tilde{c},\varphi)=
\int_{\{\psi<-t_2\}}\tilde{c}(-\psi)|\tilde{F}|^2e^{-\varphi}dV_X<+\infty.$$
Let $t_1$ be big enough such that $\{\psi=kG_X(z,z_0)<-t_1\}\subset U'\cap \bar{U}$ and $\{\psi<-t_1\}$ is an relative compact open subset in $X$ containing $z_0$. Then
\begin{equation}
\begin{split}
&\int_{\{\psi<-t_2\}}\tilde{c}(-\psi)|F|^2e^{-\varphi}dV_X\\
=&\int_{\{-t_1\le\psi<-t_2\}}\tilde{c}(-\psi)|F|^2e^{-\varphi}dV_X
+\int_{\{\psi<-t_1\}}\tilde{c}(-\psi)|F|^2e^{-\varphi}dV_X\\
=&I_1+I_2
\label{formula3.33}
\end{split}
\end{equation}
where $I_1=\int_{\{-t_1\le\psi<-t_2\}}\tilde{c}(-\psi)|F|^2e^{-\varphi}dV_X$
and $I_2=\int_{\{\psi<-t_1\}}\tilde{c}(-\psi)|F|^2e^{-\varphi}dV_X$. Formula \eqref{formula3.32} implies that
\begin{equation}
\int_{\{-t_1\le\psi<-t_2\}}c(-\psi)|F|^2e^{-\varphi}dV_X<+\infty.
\label{formula3.34}
\end{equation}
As $c(t)\in \mathcal{G}_0$, it follows from condition $(2)$ and $(3)$ of $\mathcal{G}_0$ that $c(t)\neq 0$ for any $t>0$. $c(t)$ is also smooth on $[t_2,t_1]$, hence $\inf\limits_{t\in[t_2,t_1]}c(t)>0$. Then by inequality \eqref{formula3.34}, we have
$$\int_{\{-t_1\le\psi<-t_2\}}|F|^2e^{-\varphi}dV_X<+\infty.$$
Since $\tilde{c}(t)$ is smooth on $[t_2,t_1]$, we know
\begin{equation}\label{formula3.35}
  I_1\le (\sup_{t\in[t_1,t_2]}\tilde{c}(t))\int_{\{-t_1\le\psi<-t_2\}}|F|^2e^{-\varphi}dV_X<+\infty.
\end{equation}
For $I_2$, we have
\begin{equation}\label{formula3.36}
\begin{split}
  I_2&\le \int_{\{\psi<-t_1\}}\tilde{c}(-\psi)|F-f|^2e^{-\varphi}dV_X
  +\int_{\{\psi<-t_1\}}\tilde{c}(-\psi)|f|^2e^{-\varphi}dV_X\\
  =&S_1+S_2,
  \end{split}
\end{equation}
where $S_1=\int_{\{\psi<-t_1\}}\tilde{c}(-\psi)|F-f|^2e^{-\varphi}dV_X$
and
$S_2=\int_{\{\psi<-t_1\}}\tilde{c}(-\psi)|f|^2e^{-\varphi}dV_X$.
\par
Note that $\tilde{c(t)}\in\mathcal{G}_0$, we know $\tilde{c}(t)< Ce^t$ for some constant $C>0$. It follows from $(\tilde{F}-f)\in
H^0(\{\psi<-t_2\} \cap U' ,\mathcal{O} (K_X) \otimes \mathcal{I}(\psi+\varphi)|_U)$ and $\{\psi<-t_1\}$ is relatively compact in $X$ that
$$S_1
=\int_{\{\psi<-t_1\}}\tilde{c}(-\psi)|F-f|^2e^{-\varphi}dV_X
\le
C\int_{\{\psi<-t_1\}}e^{-\psi}|F-f|^2e^{-\varphi}dV_X<+\infty.$$
For $S_2$, we have
\begin{equation}\label{forumula3.37}
\begin{split}
S_2&\le \int_{\{\psi<-t_1\}}\tilde{c}(-\psi)|f-\tilde{F}|^2e^{-\varphi}dV_X
+\int_{\{\psi<-t_1\}}\tilde{c}(-\psi)|\tilde{F}|^2e^{-\varphi}dV_X\\
&\le
C\int_{\{\psi<-t_1\}}e^{-\psi}|f-\tilde{F}|^2e^{-\varphi}dV_X
+\int_{\{\psi<-t_1\}}\tilde{c}(-\psi)|\tilde{F}|^2e^{-\varphi}dV_X
\end{split}
\end{equation}
It follows from the set $\{\psi<-t_1\}$ is relatively compact in $X$ and
$$(\tilde{F}-f)\in
H^0(\{\psi<-t_2\} \cap \bar{U} ,\mathcal{O} (K_X) \otimes \mathcal{I}(\psi+\varphi)|_U)$$
for some open set $\bar{U}$ such that $z_0\subset \bar{U} \subset U$ and
$$H(t_2;\tilde{c},\varphi)=
\int_{\{\psi<-t_2\}}\tilde{c}(-\psi)|\tilde{F}|^2e^{-\varphi}dV_X<+\infty$$
that we know $S_2<+\infty$. Hence we have
$$\int_{\{\psi<-t_2\}}\tilde{c}(-\psi)|F|^2e^{-\varphi}dV_X<+\infty.$$

\textbf{Step 2:} The following proof is almost the same as the Step 2 in the proof of Corollary \ref{corollary1.4}.
\par
Given $t_2\ge 0$. It follows from Lemma \ref{existence of F} that there exists a holomorphic $(n,0)$ form $\tilde{F}$ on $\{\psi<-t_2\}$ which satisfies
$$(\tilde{F}-f)\in
H^0(\{\psi<-t_2\} \cap \bar{U} ,\mathcal{O} (K_X) \otimes \mathcal{I}(\psi+\varphi)|_U)$$
for some open set $\bar{U}$ such that $z_0\subset \bar{U} \subset U$ and
$$H(t_2;\tilde{c},\varphi)=
\int_{\{\psi<-t_2\}}\tilde{c}(-\psi)|\tilde{F}|^2e^{-\varphi}dV_X<+\infty.$$
It follows the result in Step 1 that $I(t):=\int_{\{\psi<-t\}} c(-\psi)|\tilde{F}|^2e^{-\varphi}dV_X<+\infty$, for any $t\ge t_2$. Fix  $t_0>t_1\ge t_2$, Proposition \ref{observation} shows that
$$\int_{\{-t_0\le\psi<-t_1\}} c(-\psi)|\tilde{F}|^2e^{-\varphi}dV_X
\ge\int_{\{-t_0\le\psi<-t_1\}} c(-\psi)|F|^2e^{-\varphi}dV_X,$$
the equality holds if and only if $\tilde{F}=F|_{\{\psi<-t_2\}}$.
Hence we know
\begin{equation}
\frac{I(t_1)-I(t_0)}{\int_{t_1}^{t_0}c(t)e^{-t}dt}\ge \frac{H(t_1;c)-H(t_0;c)}{\int_{t_1}^{t_0}c(t)e^{-t}dt}=k,
\label{formula3.38}
\end{equation}
the equality holds if and only if $\tilde{F}=F|_{\{\psi<-t_2\}}$.
\par
Note that we also have
\begin{equation}
\begin{split}
H(t_2;\tilde{c})-H(t_1;\tilde{c})
\ge& \int_{\{-t_1\le\psi<-t_2\}} \tilde{c}(-\psi)|\tilde{F}|^2e^{-\varphi}dV_X\\
=&\sum_{i=1}^n \int_{\{-t_1+(i-1)\frac{t_1-t_2}{n}\le\psi<-t_1+i\frac{t_1-t_2}{n}\}} \frac{\tilde{c}(-\psi)}{c(-\psi)}c(-\psi)|\tilde{F}|^2e^{-\varphi}dV_X
\end{split}
\end{equation}
As $c(t)\in \mathcal{G}_T$, it follows from condition (2) and (3) of $\mathcal{G}_T$ that $c(t)\neq 0$ for any $t\ge T$. Then $\frac{\tilde{c}(-\psi)}{c(-\psi)}$ is uniformly continuous on $[t_2,t_1]$. When $n$ big enough, we have
\begin{equation}\nonumber
\begin{split}
H(t_2;\tilde{c})-H(t_1;\tilde{c})
\ge&\sum_{i=1}^n (\int_{\{-t_1+(i-1)\frac{t_1-t_2}{n}\le\psi<-t_1+i\frac{t_1-t_2}{n}\}} c(-\psi)|\tilde{F}|^2e^{-\varphi}dV_X)\times\\
&(\frac{\tilde{c}(t_1-i\frac{t_1-t_2}{n})}{c(t_1-i\frac{t_1-t_2}{n})}-\epsilon)\\
=&S_{1,n}+S_{2,n}
\end{split}
\end{equation}
where
$$S_{1,n}=\sum_{i=1}^n (\int_{\{-t_1+(i-1)\frac{t_1-t_2}{n}\le\psi<-t_1+i\frac{t_1-t_2}{n}\}} c(-\psi)|\tilde{F}|^2e^{-\varphi}dV_X)
\frac{\tilde{c}(t_1-i\frac{t_1-t_2}{n})}{c(t_1-i\frac{t_1-t_2}{n})},$$
and
$$S_{2,n}=-\epsilon\sum_{i=1}^n \int_{\{-t_1+(i-1)\frac{t_1-t_2}{n}\le\psi<-t_1+i\frac{t_1-t_2}{n}\}} c(-\psi)|\tilde{F}|^2e^{-\varphi}dV_X.
$$
It is easy to see that $\lim\limits_{n \to +\infty}S_{2,n}=0$. For $S_{1,n}$, we have
\begin{equation}
\begin{split}
S_{1,n}=&\sum_{i=1}^n \frac{I(t_1-i\frac{t_1-t_2}{n})-I(t_1-(i-1)\frac{t_1-t_2}{n})}
{\int_{t_1-i\frac{t_1-t_2}{n}}^{t_1-(i-1)\frac{t_1-t_2}{n}}c(t)e^{-t}dt}\times\\
&[\frac{\int_{t_1-i\frac{t_1-t_2}{n}}^{t_1-(i-1)\frac{t_1-t_2}{n}}c(t)e^{-t}dt}{c(t_1-i\frac{t_1-t_2}{n})e^{-t_1+i\frac{t_1-t_2}{n}}\frac{t_1-t_2}{n}}
\tilde{c}(t_1-i\frac{t_1-t_2}{n})e^{-t_1+i\frac{t_1-t_2}{n}}\frac{t_1-t_2}{n}]\\
\ge&\sum_{i=1}^nk[\frac{\int_{t_1-i\frac{t_1-t_2}{n}}^{t_1-(i-1)\frac{t_1-t_2}{n}}c(t)e^{-t}dt}{c(t_1-i\frac{t_1-t_2}{n})e^{-t_1+i\frac{t_1-t_2}{n}}\frac{t_1-t_2}{n}}
\tilde{c}(t_1-i\frac{t_1-t_2}{n})e^{-t_1+i\frac{t_1-t_2}{n}}\frac{t_1-t_2}{n}]\\
\label{formula3.40}
\end{split}
\end{equation}
The $``\ge"$ holds because of \eqref{formula3.38}.
Let $n \to +\infty$ in \eqref{formula3.40} we have
$\lim\limits_{n \to +\infty}S_{1,n}\ge k\int_{t_2}^{t_1}\tilde{c}(t)e^{-t}dt$. Hence we have
$$H(t_2;\tilde{c})-H(t_1;\tilde{c})\ge k \int_{t_2}^{t_1}\tilde{c}(t)e^{-t}dt$$
i.e.
$$\frac{H(t_2;\tilde{c})-H(t_1;\tilde{c})}{\int_{t_2}^{t_1}\tilde{c}(t)e^{-t}dt}\ge k.$$
Let $t_1 \to +\infty$, then
\begin{equation}
\frac{H(t_2;\tilde{c})}{\int_{t_2}^{+\infty}\tilde{c}(t)e^{-t}dt}\ge k
\label{formula3.41}
\end{equation}
\par
Recall that
 $F$ is the holomorphic $(n,0)$ form on $X$ such that $H(t;c)=\int_{\{\psi<-t\}} c(-\psi)|F|^2e^{-\varphi}dV_X$, for any $t\ge T$.
Let $T\le t_2<t_1<+\infty$, we have
$$\frac{\int_{\{-t_1\le \psi<-t_2\}}c(-\psi)|F|^2e^{-\varphi}dV_X}{\int_{t_2}^{t_1}c(t)e^{-t}dt}=k$$
\par
Note that
\begin{equation}
\begin{split}
&\int_{\{-t_1\le\psi<-t_2\}}\tilde{c}(-\psi)|F|^2e^{-\varphi}dV_X\\
=&\sum_{i=1}^n \int_{\{-t_1+(i-1)\frac{t_1-t_2}{n}\le\psi<-t_1+i\frac{t_1-t_2}{n}\}} \tilde{c}(-\psi)|F|^2e^{-\varphi}dV_X
\label{formula3.42}
\end{split}
\end{equation}
Let $n$ be big enough, the right hand side of \eqref{formula3.42} is bounded by
\begin{equation}
\begin{split}
&\sum_{i=1}^n (\int_{\{-t_1+(i-1)\frac{t_1-t_2}{n}\le\psi<-t_1+i\frac{t_1-t_2}{n}\}} c(-\psi)|F|^2e^{-\varphi}dV_X)(\frac{\tilde{c}(t_1-i\frac{t_1-t_2}{n})}{c(t_1-i\frac{t_1-t_2}{n})}\pm \epsilon)\\
=&\sum_{i=1}^n k\int_{t_1-i\frac{t_1-t_2}{n}}^{t_1-(i-1)\frac{t_1-t_2}{n}}c(t)e^{-t}dt(\frac{\tilde{c}(t_1-i\frac{t_1-t_2}{n})}{c(t_1-i\frac{t_1-t_2}{n})}\pm \epsilon)\\
=&\sum_{i=1}^n [k \frac{\int_{t_1-i\frac{t_1-t_2}{n}}^{t_1-(i-1)\frac{t_1-t_2}{n}}c(t)e^{-t}dt}{c(t_1-i\frac{t_1-t_2}{n})e^{-t_1+i\frac{t_1-t_2}{n}}\frac{t_1-t_2}{n}}
(\tilde{c}(t_1-i\frac{t_1-t_2}{n})e^{-t_1+i\frac{t_1-t_2}{n}}\frac{t_1-t_2}{n})\pm \\
& k\epsilon\int_{t_1-i\frac{t_1-t_2}{n}}^{t_1-(i-1)\frac{t_1-t_2}{n}}c(t)e^{-t}dt]
\label{formula3.43}
\end{split}
\end{equation}
When $n \to +\infty$, combining \eqref{formula3.42} and \eqref{formula3.43}, we have
\begin{equation}
\begin{split}
\int_{\{-t_1\le\psi<-t_2\}}\tilde{c}(-\psi)|F|^2e^{-\varphi}dV_X=k\int_{t_2}^{t_1}\tilde{c}(t)e^{-t}dt
\label{formula3.44}
\end{split}
\end{equation}
Let $t_1$ goes to $+\infty$ in \eqref{formula3.44}, we know
\begin{equation}\nonumber
\begin{split}
\frac{\int_{\{\psi<-t_2\}}\tilde{c}(-\psi)|F|^2e^{-\varphi}dV_X}
{\int_{t_2}^{+\infty}\tilde{c}(t)e^{-t}dt}=k
\end{split}
\end{equation}
Hence
\begin{equation}
\begin{split}
\frac{H(t_2;\tilde{c})}{\int_{t_2}^{+\infty}\tilde{c}(t)e^{-t}dt}\le
\frac{\int_{\{\psi<-t_2\}}\tilde{c}(-\psi)|F|^2e^{-\varphi}dV_X}
{\int_{t_2}^{+\infty}\tilde{c}(t)e^{-t}dt}=k
\label{formula3.
45}
\end{split}
\end{equation}
It follows from \eqref{formula3.41} and \eqref{formula3.
45} that for any $t_2\ge T$,
$$\frac{H(t_2;\tilde{c})}{\int_{t_2}^{+\infty}\tilde{c}(t)e^{-t}dt}=k$$
holds, i.e. $H(h^{-1}_{\tilde{c}}(r);\tilde{c})$ is linear with respect to $r$. Hence there exists a holomorphic $(n,0)$ form $F_{\tilde{c}}$ on X such that
$$H(t_2;\tilde{c})=\int_{\{\psi<-t_2\}}\tilde{c}(-\psi)|F_{\tilde{c}}|^2e^{-\varphi}dV_X$$
and we also have
\begin{equation}
\frac{H(t_2;\tilde{c})-H(t_1;\tilde{c})}{\int_{t_2}^{t_1}\tilde{c}(t)e^{-t}dt}=
\frac{\int_{\{-t_1\le\psi<-t_2\}}\tilde{c}(-\psi)|F_{\tilde{c}}|^2e^{-\varphi}dV_X}{\int_{t_2}^{t_1}\tilde{c}(t)e^{-t}dt}
\label{formula3.46}
\end{equation}
If $F_{\tilde{c}}\neq F$ on $X$, it follows from Proposition \ref{observation}, \eqref{formula3.44} and \eqref{formula3.46} that
$$k=\frac{\int_{\{-t_1\le\psi<-t_2\}}\tilde{c}(-\psi)|F|e^{-\varphi}dV_X}{\int_{t_2}^{t_1}\tilde{c}(t)e^{-t}dt}
>\frac{\int_{\{-t_1\le\psi<-t_2\}}\tilde{c}(-\psi)|F_{\tilde{c}}|^2e^{-\varphi}dV_X}{\int_{t_2}^{t_1}\tilde{c}(t)e^{-t}dt}=k,$$
which is a contradiction. Hence we must have $F_{\tilde{c}}=F$ on $X$.

\end{proof}

\begin{Remark}
Prposition \ref{variousC} shows that if there exists  $c_1(t)\in C^{\infty}[0,+\infty)$ and $c_1(t)\in \mathcal{G}_0$ such that $H(h_1^{-1}(r);c_1,\varphi)$ is linear with respect to $r$, where $h_1(t)=\int_t^{+\infty}c_1(t_1)e^{-t_1}dt_1$. Then for any other $c(t)\in C^{\infty}[0,+\infty)$ and $c(t)\in \mathcal{G}_0$, we know $H(h^{-1}(r);c,\varphi)$ is also linear with respect to $r$, where $h(t)=\int_t^{+\infty}c(t_1)e^{-t_1}dt_1$.
\par
Let $c(t)\equiv 1$, then $h^{-1}(r)=\log r$. It follows form Proposition \ref{variousC} that to prove Theorem \ref{corollary1.9} and Theorem \ref{corollary1.11}, we only need to consider the necessary and sufficient condition
for the function $H(-\log r;1,\varphi)$ being linear with respect to $r$. We denote $H(t;1,\varphi)$ by $H(t;\varphi)$ for simplicity.
\label{remark3.1}
\end{Remark}
\par
Now we begin to prove Theorem \ref{corollary1.9}.
\par
As $\varphi$ is a plurisubharmonic function on $X$ and $i\partial\bar{\partial}\varphi\neq 0$ on $X$.
By Siu's decomposition theorem, we have
$$\frac{i}{\pi}\partial\bar{\partial}\varphi=\sum\limits_{j\ge 1} \lambda_j[x_j]+R,\ \lambda_j>0$$
where $x_j\in X$ is a point, $\lambda_j=v(i\partial\bar{\partial}\varphi,x_j)$ is the Lelong number of $i\partial\bar{\partial}\varphi$ at $x_j$, R is a closed positive $(1,1)$ current with $v(R,x)=0$ for $x \in X$. Note that $E_1(T)=\{x\in X| v(i\partial\bar{\partial}\varphi,x)\ge 1\}=\{x_j | \lambda_j\ge 1\}$ is a analytic subset of $X$, hence $E_1(T)$ a set of isolated points.
Denote $E:=\{x\in X|\  v(T,x)\  is\ a\  positive\  integer\}$, $E\subset E_1(T)$ is also a set of isolated points.
\par
 We need the following Lemma to prove
Theorem \ref{corollary1.9}.
\begin{Lemma} If $(i\partial\bar{\partial}\varphi) |_{X\backslash E}\neq 0$. Then there exists a function $\tilde{\varphi}\in PSH(X)$ $\tilde{\varphi}>\varphi$ and $\mathcal{I}(2\tilde{\varphi})_x=\mathcal{I}(2\varphi)_x$, for any $x\in X$. Let $z\in \{-t<kG_X(z,z_0)<0\}$, when $t\to 0$, we have $\tilde{\varphi}(z)\to \varphi(z)$. Moreover, there exists a relatively compact open subset $U\subset \subset X$ such that $\varphi-\tilde{\varphi}$ has lower bound $-A$ ($A>0$ is a constant) for any $z \in X\backslash U$.
\label{lemma for cor1.9}
\end{Lemma}

 Lemma \ref{lemma for cor1.9} will be proved in the Appendix (see Section \ref{4.2}).
 Now we prove Theorem \ref{corollary1.9}.
\begin{proof}[Proof of Theorem \ref{corollary1.9}]
We only need to show that if $H(-\log r; 2\varphi)$ is linear with respect to $r$,
then we have $\varphi=\log|f_{\varphi}|+v$, where $f_{\varphi}$ is a holomorphic function on $X$ and $v$ is a harmonic function on $X$.
\par
Our proof will be divided into two steps.
\\
\textbf{Step 1:}
\par
In step 1, we will show that if $(i\partial\bar{\partial}\varphi) |_{X\backslash E}\neq 0$, then $H(-\log r;2\varphi)$ can not be linear with respect to $r\in (0,1]$.
\par
 Assume that $H(-\log r;2\varphi)$ is linear with respect to $r\in(0,1]$.

As $H(-\log r;2\varphi)$ is linear with respect to $r$, it follows from Theorem \ref{theorem1.3} that there exists a holomorphic $(1,0) $ form $F$ on $X$ such that $\forall \ t \ge 0$,
$$H(t;\varphi)=\int_{\{\psi<-t\}}|F|^2e^{-2\varphi}dV_X$$
holds.
As $e^{-2\tilde{\varphi}}<e^{-2\varphi}$, we have
\begin{equation}
k=\frac{H(t;2\varphi)}{e^{-t}}>\frac{\int_{\{\psi<-t\}}|F|^2e^{-2\tilde{\varphi}}dV_X}{e^{-t}}
\ge \frac{H(t;\tilde{2\varphi})}{e^{-t}}
\label{3.29}
\end{equation}
When $t=0$, there exists a holomorphic $(1,0) $ form $\tilde{F}$ on $X$ such that
$$H(t;2\tilde{\varphi})=\int_{X}|\tilde{F}|^2e^{-2\tilde{\varphi}}dV_X<+\infty$$
\par
By Lemma \ref{lemma for cor1.9}, there exist $U\subset \subset X$ such that $\varphi-\tilde{\varphi}$ has lower bound $-A$ ($A>0$ is a constant) for any $z \in X\backslash U$.
\par
Denote
$$I_1=\int_{U}|\tilde{F}|^2e^{-2\varphi}dV_X$$
and $$I_2=\int_{X\backslash U}|\tilde{F}|^2e^{-2\varphi}dV_X.$$
\par
As $U$ is relatively compact in $X$,  $\int_{U}|\tilde{F}|^2e^{-2\tilde{\varphi}}dV_X<+\infty$ and $\mathcal{I}(2\tilde{\varphi})_x=\mathcal{I}(2\varphi)_x$, for any $x\in X$, then we know
$$I_1=\int_{U}|\tilde{F}|^2e^{-2\varphi}dV_X<+\infty.$$
On $X\backslash U$, we have
$$I_2=\int_{X\backslash U}|\tilde{F}|^2e^{-2\varphi}dV_X\le e^{2A}\int_{X\backslash U}|\tilde{F}|^2e^{-2\tilde{\varphi}}dV_X<+\infty.$$
Hence \begin{equation}\nonumber
\begin{split}
\int_{X}|\tilde{F}|^2e^{-2\varphi}dV_X
=&\int_{U}|\tilde{F}|^2e^{-2\varphi}dV_X+
\int_{X\backslash U}|\tilde{F}|^2e^{-2\varphi}dV_X\\
=&I_1+I_2<+\infty
\end{split}
\end{equation}

Let $t_1>0$ be small enough such that $|\tilde{\varphi}-\varphi(z)|<\epsilon$, then we have
\begin{equation}
\begin{split}
H(0;2\tilde{\varphi})-H(t_1;2\tilde{\varphi})
&\ge\int_{\{-t_1\le\psi<0\}} |\tilde{F}|^2e^{-2\tilde{\varphi}}dV_X\\
&\ge\int_{\{-t_1\le\psi<0\}} |\tilde{F}|^2e^{-2\varphi-2\epsilon}dV_X\\
&\ge e^{-2\epsilon}\int_{\{-t_1\le\psi<0\}} |F|^2e^{-2\varphi}dV_X\\
&=e^{-2\epsilon}(H(0;2\varphi)-H(t_1;2\varphi)).
\end{split}
\end{equation}
The third $``\ge"$ holds because of Proposition \ref{observation}. Hence
\begin{equation}
\lim\limits_{t_1\to 0}\frac{H(0;2\tilde{\varphi})-H(t_1;2\tilde{\varphi})}{1-e^{-t_1}}\ge
\lim\limits_{t_1\to 0}\frac{H(0;2\varphi)-H(t_1;2\varphi)}{1-e^{-t_1}}=k
\label{3.31}
\end{equation}
It follows from \eqref{3.29}, \eqref{3.31} and Theorem \ref{maintheorem} that
$$k>\frac{H(t;2\tilde{\varphi})}{e^{-t}}\ge\lim\limits_{t_1\to 0}\frac{H(0;2\tilde{\varphi})-H(t_1;2\tilde{\varphi})}{1-e^{-t_1}}\ge k$$
which is a contradiction.
\par
Hence $H(-\log r;2\varphi)$ can not be linear with respect to $r$.
\\
\textbf{Step 2:}
\par
It follows from the result in Step 1 and $H(-\log r;\varphi)$ is linear with respect to $r$ that we know
$$\frac{i}{\pi}\partial\bar{\partial}\varphi=\sum\limits_{j\ge 1} \lambda_j[x_j],$$
where $\lambda_j$ is positive integer for any $j\ge 1$.
\par
It follows from the Weierstrass Theorem on noncompact Riemann surface (see \cite{GTM81} chapter 3, $\S$26), for divisor $D=\sum_{j\ge 1} \lambda_j x_j$, there exist a meromorphic function $f_{\varphi}$ on $X$ such that $(f_{\varphi})=D$. As $\lambda_j>0$, $f$ is actually a holomorphic function on $X$. It follows from Lelong-Poincar\'e equation that
$$\frac{i}{\pi}\partial\bar{\partial}\log |f_{\varphi}|=\sum\limits_{j\ge 1} \lambda_j[x_j].$$
Then $i\partial\bar{\partial}\varphi-i\partial\bar{\partial}\log |f_{\varphi}|=0$, i.e., $u=\varphi-\log |f_{\varphi}|$ is a harmonic function.
\par
Hence $\varphi=\log |f_{\varphi}|+u$, where $f_{\varphi}$ is a holomorphic function on $X$ and $u$ is a harmonic function on $X$. Theorem \ref{corollary1.9} is proved.
\end{proof}
Now we begin to prove Theorem \ref{corollary1.11}.
\par
Recall that $X$ is an open Riemann Surface which admits a nontrivial Green function $G_{X}(z,w)$ and $\psi=2G_{X}(z,z_0)$, where $z_0$ is a point of $X$.

Let $w$ be a local coordinate on a neighborhood $V_{z_0}$ of $z_0$ satisfying $w(z_0)=o$. Let $U=V_{z_0}$.
Let $f$ be a holomorphic $(n,0)$ form on  $V_{z_0}$.

 Let $\varphi=\log |f_{\varphi}|+v$ on $X$, where $f_{\varphi}$ is a holomorphic function on $X$ and $v$ is a harmonic function on $X$. Let $c(t)\in \mathcal{G}_0$. Denote
\begin{equation}
\begin{split}
H(t;c,2\varphi):=\inf\{\int_{ \{ \psi<-t\}}&c(-\psi)|\tilde{F}|^2e^{-2\varphi}dV_X: \tilde{F}\in
H^0(\{\psi<-t\},\mathcal{O} (K_X)  ), \\
\& &\  \exists \;open\; set\; U' \;s.t\; Z_0 \subset U' \subset U\ and \\
& \; (\tilde{F}-f)\in
H^0(\{\psi<-t\} \cap U' ,\mathcal{O} (K_X) \otimes \mathcal{I}(\psi+2\varphi)|_U) \}.
\end{split}
\end{equation}
We assume that $0<H(0;c,2\varphi)<+\infty$.
\par
It follows from Lemma \ref{existence of F} that there exists a unique
holomorphic $(n,0)$ form $F_0$ on $X$ satisfying
$$\ (F_0-f)\in
H^0(\{\psi<-t\}\cap U_0', \mathcal{O} (K_X) \otimes \mathcal{I}(\psi+2\varphi)|_U),$$
for some open set $U_1'$ such that $z_0 \subset U_0' \subset U$ and
$$\int_{X}|F_0|^2e^{-2\varphi}c(-\psi)dV_X=H(0)<+\infty.$$
\par
As $\int_{X}|F_0|^2e^{-2\varphi}c(-\psi)dV_X=
\int_{X}\frac{|F_0|^2}{|f_{\varphi}|^2}e^{-2v}c(-\psi)dV_X<+\infty$, we know $\frac{F_0}{f_{\varphi}}$ is a holomorphic $(n,0)$ form on $X$. It follows from $$\ (F_0-f)\in
H^0(\{\psi<-t\}\cap U_1', \mathcal{O} (K_X) \otimes \mathcal{I}(\psi+2\varphi)|_U)$$
that there exist a small open neighborhood $V$ such that $\frac{f}{f_{\varphi}}$ is a holomorphic $(n,0)$ form on $V$.
Denote $h:=\frac{f}{f_{\varphi}}$, we know $h$ is a holomorphic $(n,0)$ form on $V$. We also note that $h(z_0)\neq 0$, otherwise $f=h\cdot f_{\varphi}$ will belong to $\mathcal{I}(\psi+2\varphi)|_V$ which contradict to the fact that $H(0;c,2\varphi)>0$.

 We have the following limiting property of $H(t;c,2\varphi)$.
\begin{Proposition} Assume that $0<H(0;c,2\varphi)<+\infty$.
When $t \to +\infty$, we have
$$\lim\limits_{t \to +\infty} \frac{H(t;c,2\varphi)}{\int_t^{+\infty}c(t_1)e^{-t_1}dt_1}=\pi\frac{e^{-2v(z_0)}}{c^2_{\beta}(z_0)}|h(z_0)|^2$$
\label{proposition3.2}
\end{Proposition}
\begin{proof}

Let $t$ be big enough, we can assume that $\{2G_X(z,z_0)<-t\}\subset V_{z_0}$. Under the local coordinate $(V_{z_0},w)$, we have $2G_X(z,z_0)=2\log|w|+u(w)$ where $u(w)$ is a harmonic function on $V_{z_0}$. Note that $c^2_{\beta}(z_0)=e^{u(z_0)}$.
\par
For ant $t\ge0$, denote $$I_t=\int_{\{\log|w|^2+u(w)<-t\}}c(-2G_X(z,z_0))|F_t|^2e^{-2\varphi}dV_X,$$ where $F_t$ is holomorphic $(1,0)$ form on $\{2G_X(z,z_0)<-t\}$ such that
\begin{equation}
H(t)=\int_{\{2G_X(z,z_0)<-t\}}|F_t|^2e^{-2\varphi}dV_X<+\infty
\label{Prop3.3.1}
\end{equation}
and
\begin{equation}
\ (F_t-f)\in
H^0(\{\psi<-t\}\cap U_t', \mathcal{O} (K_X) \otimes \mathcal{I}(\psi+2\varphi)|_U)
\label{Prop3.3.2}
\end{equation}
for some open set $U_t'$ such that $z_0 \subset U_t' \subset U$.
\par
Denote $h_t=\frac{F_t}{g}$, it follows from \eqref{Prop3.3.1} and \eqref{Prop3.3.2} that we know $h_t$ is a holomorphic $(n,0)$ form on $\{\psi<-t\}$ and $h_t(z_0)=h(z_0)$.
 \par
When $t$ is big enough, we know $|w|$ is small. By the continuity of $u$ and $v$ at $z_0$ and note that $|h_t|^2$ is subharmonic function, we have
\begin{flalign}\nonumber
I_t &=\int_{\{\log|w|^2+u(w)<-t\}}c(-2G_X(z,z_0))\frac{|F_t|^2}{|f_{\varphi}|^2}e^{-2v}dV_X\\\nonumber
&\ge \int_{\{\log|w|^2+u(z_0)+\epsilon<-t\}}c(-2G_X(z,z_0))|h_t|^2e^{-2v}dV_X\\\nonumber
&\ge \int_{\{\log|w|^2+u(z_0)+\epsilon<-t\}}c(-\log|w|^2-u(w))|h_t|^2e^{-2v(z_0)-\epsilon}dV_X\\\nonumber
&\ge
\int_{\{\log|w|^2+u(z_0)+\epsilon<-t\}}c(-\log|w|^2-u(z_0)+\epsilon)e^{-2\epsilon}
|h_t|^2e^{-2v(z_0)-\epsilon}dV_X\\\nonumber
&=
\int_0^{2\pi}\int_{\{\log|r|^2+u(z_0)+\epsilon<-t\}}
c(-\log|r|^2-u(z_0)+\epsilon)e^{-2\epsilon}
|h_t(r,\theta)|^2e^{-2v(z_0)-\epsilon}rdrd\theta\\\nonumber
&\ge
2\pi e^{-2v(z_0)}e^{-3\epsilon}|h(z_0)|^2
\int_{\{\log|r|^2+u(z_0)+\epsilon<-t\}}
c(-\log|r|^2-u(z_0)+\epsilon)rdr\\\nonumber
&=\pi e^{-2v(z_0)}e^{-3\epsilon}|h(z_0)|^2\int_{t-\epsilon}^{+\infty}c(t_1)e^{-t_1}e^{-u(z_0)}e^{\epsilon}dt_1
\end{flalign}
The third inequality holds because of $c(t)e^{-t}$ is decreasing with respect to $t$.
The fourth inequality holds because of mean value inequality of subharmonic function.
Hence we have
\begin{equation}
\begin{split}
\liminf\limits_{t \to +\infty} \frac{I_t}{\int_t^{+\infty}c(t_1)e^{-t_1}dt_1}
&\ge \liminf\limits_{t \to +\infty} \frac{\pi e^{-2v(z_0)}e^{-3\epsilon}|h(z_0)|^2\int_{t-\epsilon}^{+\infty}c(t_1)e^{-t_1}e^{-u(z_0)}e^{\epsilon}dt_1}
{\int_t^{+\infty}c(t_1)e^{-t_1}dt_1}\\
&=\pi e^{-u(z_0)-2v(z_0)}|h(z_0)|^2\\
&=\frac{\pi e^{-2\varphi(z_0)}}{c^2_{\beta}(z_0)}|h(z_0)|^2
\label{3.34}
\end{split}
\end{equation}
\par
When $t=0$, denote $S_t=\int_{\{\psi<-t\}}c(-\psi)|F_0|^2e^{-2\varphi}dV_X$.
When $t$ is big enough, we know $|w|$ is small. By the continuity of $u$, $v$ and $h_0=\frac{F_0}{f_{\varphi}}$ at $z_0$,
then we have
\begin{flalign}\nonumber
S_t &=\int_{\{\psi<-t\}}c(-\psi)\frac{|F_0|^2}{|f_{\varphi}|^2}e^{-2v}dV_X\\\nonumber
&=\int_{\{\psi<-t\}}c(-\psi)|h_0|^2e^{-2v}dV_X\\\nonumber
&\le \int_{\{\log|w|^2+u(z_0)-\epsilon<-t\}}c(-2G_X(z,z_0))|h_0|^2e^{-2v+\epsilon}dV_X\\\nonumber
&\le
\int_{\{\log|w|^2+u(z_0)-\epsilon<-t\}}c(-\log|w|^2-u(z_0)-\epsilon)e^{2\epsilon}
|h_0|^2e^{-2v(z_0)+\epsilon}dV_X\\\nonumber
&=
\int_0^{2\pi}\int_{\{\log|r|^2+u(z_0)-\epsilon<-t\}}
c(-\log|r|^2-u(z_0)-\epsilon)e^{+2\epsilon}
|h_0(r,\theta)|^2e^{-2v(z_0)+\epsilon}rdrd\theta\\\nonumber
&\le
2\pi e^{-2v(z_0)}e^{+3\epsilon}(|h(z_0)|^2+\epsilon)
\int_{\{\log|r|^2+u(z_0)-\epsilon<-t\}}
c(-\log|r|^2-u(z_0)-\epsilon)rdr\\\nonumber
&=\pi e^{-2v(z_0)}e^{3\epsilon}(|h(z_0)|^2+\epsilon)\int_{t+\epsilon}^{+\infty}c(t_1)e^{-t_1}e^{-u(z_0)}e^{-\epsilon}dt_1
\end{flalign}
The second inequality holds because of $c(t)e^{-t}$ is decreasing with respect to $t$.
\par

Hence
\begin{equation}
\limsup\limits_{t \to +\infty} \frac{H(t;c,2\varphi)}{\int_t^{+\infty}c(t_1)e^{-t_1}dt_1}\le
\limsup\limits_{t \to +\infty} \frac{S_t}{\int_t^{+\infty}c(t_1)e^{-t_1}dt_1}\le
\frac{\pi e^{-2v(z_0)}}{c^2_{\beta}(z_0)}|h(z_0)|^2
\label{3.35}
\end{equation}
It follows from inequality \eqref{3.34} and \eqref{3.35} that
$$\lim\limits_{t \to +\infty} \frac{H(t;c,2\varphi)}{\int_t^{+\infty}c(t_1)e^{-t_1}dt_1}
=\frac{\pi e^{-2v(z_0)}}{c^2_{\beta}(z_0)}|h(z_0)|^2$$
\par
Proposition \ref{proposition3.2} is proved.
\end{proof}

 Recall that $\varphi$ is a subharmonic function on $X$. such that $\varphi=\log|f_{\varphi}|+v$, where $f_{\varphi}$ is a holomorphic function on $X$ and $v$ is a harmonic function on $X$.

We also note that $h:=\frac{f}{f_{\varphi}}$ is a holomorphic $(n,0)$ form on $V\subset V_{z_0}$ and $h(z_0)\neq 0$.

 Denote
\begin{equation}
\begin{split}
H(t;2v):=\inf\{\int_{ \{ \psi<-t\}}&|\tilde{F}|^2e^{-2u}dV_X: \tilde{F}\in
H^0(\{\psi<-t\},\mathcal{O} (K_X)  ), \\
\& &\  \exists \;open\; set\; U' \;s.t\; Z_0 \subset U' \subset U\ and \\
& \; (\tilde{F}-h)\in
H^0(\{\psi<-t\} \cap U' ,\mathcal{O} (K_X) \otimes \mathcal{I}(\psi)|_U) \}.
\end{split}
\end{equation}
\begin{Proposition}
We have $H(0;2v)=H(0;2\varphi)$ holds.
\label{proposition3.3}
\end{Proposition}

\begin{proof}
Denote
$$H_1:=\{F \in H^0(X,K_X)| \int_X |F|^2e^{-2\varphi}dV_X<+\infty \ \&\  (F-f,z_0)\in\mathcal{I}(\psi+2\varphi)_{z_0}\}$$
and
$$H_2:=\{\tilde{F} \in H^0(X,K_X)|
\int_X |\tilde{F}|^2e^{-2v}dV_X<+\infty\  \&\  (\tilde{F}-h,z_0)\in\mathcal{I}(\psi)_{z_0}\}$$
\par
As $\int_X |F|^2e^{-2\varphi}dV_X=\int_X \frac{|F|^2}{|f_{\varphi}|^2}e^{-2v}dV_X<+\infty$, we know for any $F\in H_1$, $\frac{F}{f_{\varphi}}$ is a holomorphic $(1,0)$ form on $X$. It follows from $\varphi=\log|f_{\varphi}|+v$ and $(F-f,z_0)\in\mathcal{I}(\psi+2\varphi)_{z_0}$ that we know $\frac{F}{f_{\varphi}}$ belongs to $H_2$. For any $\tilde{F}\in H_2$, $\tilde{F}\cdot f_{\varphi}$ belongs to $H_1$ for the similar reason.
\par
Hence there exists a bijection $\Phi$ between $H_1$ and $H_2$:
\begin{equation}\nonumber
\begin{split}
 \Phi:&H_2\to H_1\\
 &F\to F\cdot f_{\varphi}
\end{split}
\end{equation}
It follows from Lemma \ref{existence of F} that there exist unique holomorphic $(1,0)$ form $F_{\varphi}\in H_1$ such that
$$H(0,\varphi)=\int_X |F_{\varphi}|^2e^{-2\varphi}dV_X,$$
and unique holomorphic $(1,0)$ form $F_v \in H_2$ such that
$$H(0,v)=\int_X |F_{v}|^2e^{-2v}dV_X.$$
We claim that $F_{\varphi}=F_v\cdot f_{\varphi}$ i.e the weighted $L^2$ norm of $F_v\cdot f_{\varphi}$ is minimal along $H_1$.
If not, we have
$$\int_X|F_{\varphi}|^2e^{-2\varphi}dV_X
<\int_X|F_v|^2|f_{\varphi}|^2e^{-2\varphi}dV_X=\int_X |F_{v}|^2e^{-2v}dV_X.$$
Note that $\frac{F_{\varphi}}{f_{\varphi}}\in H_2$ and then we have
$$\int_X \frac{|F_{\varphi}|^2}{|f_{\varphi}|^2}e^{-2v}dV_X=\int_X |F_{\varphi}|^2e^{-2\varphi}dV_X<\int_X |F_{v}|^2e^{-2v}dV_X.$$
which contradicts to the fact that the weighted $L^2$ norm of $F_v$ is minimal along $H_2$. Hence we must have $F_{\varphi}=F_v\cdot f_{\varphi}$. Then we know
$$H(0;2\varphi)=\int_X |F_{\varphi}|^2e^{-2\varphi}dV_X
=\int_X |F_{v}|^2e^{-2v}dV_X=H(0;2v).$$

\par
Proposition \ref{proposition3.3} is proved.
\end{proof}
\begin{Remark}
It follows from Proposition \ref{proposition3.3} that we know when $\varphi=\log|f_{\varphi}|+v$, where $f_{\varphi}$ is a holomorphic function on $X$ and $v$ is a harmonic function on $X$, we have $H(0;2\varphi)=\pi\frac{e^{-2\varphi(z_0)}}{c_{\beta}^2(z_0)}|h(z_0)|^2$ if and only if $H(0;2v)=\pi\frac{e^{-2v(z_0)}}{c_{\beta}^2(z_0)}|h(z_0)|^2$.
\par
Then it follows from Theorem \ref{maintheorem}, Proposition \ref{proposition3.2} and $``H(0;2\varphi)=\pi\frac{e^{-2\varphi(z_0)}}{c_{\beta}^2(z_0)}|h(z_0)|^2$ if and only if $H(0;2v)=\pi\frac{e^{-2v(z_0)}}{c_{\beta}^2(z_0)}|h(z_0)|^2"$ that we know
$H(-\log r;2\varphi)$ is linear with respect to $r\in(0,1]$ if and only if $H(-\log r;2v)$ is linear with respect to $r\in(0,1]$.
\label{remark3.2}
\end{Remark}

Let $v$ be a harmonic function on $X$, we have the following result which was proved by Guan-Zhou \cite{GZsci},\cite{GZ14Ann}.
\begin{Theorem}(Guan-Zhou\cite{GZsci},\cite{GZ14Ann})  The equality
$$H(0;2v)=\pi\frac{e^{-2v(z_0)}}{c_{\beta}^2(z_0)}|h(z_0)|^2$$
holds if and only if $\chi_{-v}=\chi_{z_0}$.
\label{theorem1.10}
\end{Theorem}

Now we prove Theorem \ref{corollary1.11}

\begin{proof}
It follows from Proposition \ref{variousC} that we only need to prove the following statement:
\par
$H(-\log r;2\varphi)$ is linear with respect to $r\in(0,1]$ if and only if the following hold,
\par
(1) $\varphi=\log|f_{\varphi}|+v$, where $f_{\varphi}$ is a holomorphic function on $X$ and $v$ is a harmonic function on $X$,
\par
(2) $\chi_{-u}=\chi_{z_0}$.
~\\
\par

If $H(-\log r;2\varphi)$ is linear with respect to $r\in(0,1]$.
It follows from Theorem \ref{corollary1.9} that $\varphi=\log|f_{\varphi}|+v$, where $f_{\varphi}$ is a holomorphic function on $X$ and $v$ is a harmonic function on $X$. As
$H(-\log r;2\varphi)$ is linear with respect to $r$, then by the Remark \ref{remark3.2} that we know
$H(-\log r;2v)$ is linear with respect to $r$, hence

$$H(0;2v)=\pi\frac{e^{-2v(z_0)}}{c_{\beta}^2(z_0)}|h(z_0)|^2.$$
\par
Then
it follows form Theorem \ref{theorem1.10} that we know
$\chi_{-v}=\chi_{z_0}$.
\par
If $\varphi=\log|f_{\varphi}|+v$, where $f_{\varphi}$ is a holomorphic function on $X$ and $v$ is a harmonic function on $X$ and $\chi_{-v}=\chi_{z_0}$. It follows form Theorem \ref{theorem1.10} that we know $$H(0;2v)=\pi\frac{e^{-2v(z_0)}}{c_{\beta}^2(z_0)}|h(z_0)|^2.$$
Hence $H(-\log r;2v)$ is linear with respect to $r$. By Remark \ref{remark3.2}, we know $H(-\log r;2\varphi)$ is linear with respect to $r$.
\par

Theorem \ref{corollary1.11} is proved.
\end{proof}

\section {Appendix}
\subsection {Appendix: Proof of Lemma 2.1}

\subsubsection {Some results used in the proof of Lemma 2.1}
\begin{Lemma}(see \cite{Demailly00})
Let Q be a Hermitian vector bundle on a K\"ahler manifold X of dimension n with a
k\"ahler metric $\omega$. Assume that $\eta , g >0$ are smooth functions on X. Then
for every form $v\in D(X,\wedge^{n,q}T^*X \otimes Q)$ with compact support we have
\begin{equation}
\begin{split}
&\int_X (\eta+g^{-1})|D^{''*}v|^2_QdV_X+\int_X \eta|D^{''}v|^2_QdV_X \\
\ge  \int_X \langle[\eta &\sqrt{-1}\Theta_Q-\sqrt{-1}\partial \bar{\partial}
\eta-\sqrt{-1}g
\partial\eta \wedge\bar{\partial}, \Lambda_{\omega}]v,v\rangle_QdV_X
\end{split}
\end{equation}
\end{Lemma}

\begin{Lemma}(see \cite{GZ14Ann})
Let X and Q be as in the above lemma and $\theta$ be a continuous (1,0) form on X.
Then we have
\begin{equation}
[\sqrt{-1}\theta \wedge
\bar{\theta},\Lambda_\omega]\alpha=\bar{\theta}\wedge(\alpha\llcorner(\bar{\theta})^\sharp),
\end{equation}
for any (n,1) form $\alpha$ with value in Q. Moreover, for any positive (1,1) form
$\beta$,we have $[\beta,\Lambda_\omega]$ is semipositive.
\label{semipositive}
\end{Lemma}

\begin{Lemma}
(see \cite{DemaillyAG},,\cite{Demailly15})
Let (X,$\omega$) be a complete k\"ahler manifold equipped with a (non-necessarily
complete) k\"ahler metric $\omega$, and let Q be a Hermitian vector bundle over $X$.
Assume that $\eta$ and $g$ are smooth bounded positive functions on X and let
$\textbf{B}:=[\eta \sqrt{-1}\Theta_Q-\sqrt{-1}\partial \bar{\partial} \eta-\sqrt{-1}g
\partial\eta \wedge\bar{\partial}\eta, \Lambda_{\omega}]$. Assume that $\textbf{B}$ is semi-positive definite everywhere on
$\wedge^{n,q}T^*X \otimes Q$ for some $q \ge 1$. Then given a form $g \in
L^2(X,\wedge^{n,q}T^*X \otimes Q)$ such that $D^{''}g=0$ and $\int_X \langle
\textbf{B}^{-1}g,g\rangle_Q dV_X < +\infty$, there exists $u \in L^2(X,\wedge^{n,q-1}T^*X \otimes Q)$ such that $D^{''}u=g$ and
\begin{equation}
\int_X(\eta+g^{-1})^{-1}|u|^2_QdV_X \le \int_X \langle \textbf{B}^{-1}g,g\rangle_Q dV_X
\end{equation}
\label{solveequation}
\end{Lemma}
\par
In the last part of this section, we recall a theorem of Forn$\ae$ss and Narasimhan on approximation property of plurisubharmonic function of Stein manifolds.
\begin{Lemma}(see \cite{FornaessNarasimhan}) Let X be a Stein manifold and $\varphi\in PSH(X)$. Then there exists a sequence $\{\varphi\}_{n=1,2,\ldots}$ of smooth strongly plurisubharmonic functions  such that $\varphi_n\downarrow \varphi$.
\label{regularization}
\end{Lemma}

\subsubsection {Proof of Lemma 2.1}
Since $X$ is Stein manifold, there exists a smooth plurisubharmonic
exhaustion function P on X. Let $X_k:=\{P<k\}$ ($k=1,2,...$, we choose P such that
$X_1\ne \emptyset$).\par
Then $X_k$ satisfies $X_1 \subset \subset X_2\subset \subset ...\subset \subset
X_k\subset \subset X_{k+1}\subset \subset ...$ such that $\bigcup_{k=1}^{+\infty}X_k=X$ and each $X_k$ is Stein manifold with exhaustion plurisubharmonic function
$P_k=1/(k-P)$. We will discuss for fixed $k$ until step 8.\\

\emph{Step 1: Regularization of $\psi$ and $\varphi$}
\par
It follows from Lemma \ref{regularization} that there exist smooth strongly plurisubharmonic functions $\psi_m$ and $\varphi_{m'}$ on $X$ decreasing convergent to $\psi$ and $\varphi$ respectively, satisfying $\sup_m \sup_{X_k}\psi_m<-T$ and $\sup_{m'}\sup_{X_k}\varphi_{m'}<+\infty$.
\par

\emph{Step \  2: Recall some notations}\\
\par
Let $\epsilon \in (0,\frac{1}{8}B)$. Let $\{v_\epsilon\}_{\epsilon \in
(0,\frac{1}{8}B)}$ be a family of smooth increasing convex functions on $\mathbb{R}$, such
that:
\par
(1)$v_{\epsilon}(t)=t$ for $t\ge-t_0-\epsilon$, $v_{\epsilon}(t)=constant$ for
$t<-t_0-B+\epsilon$;\par
(2)$v_{\epsilon}^{''}(t)$ are pointwise convergent
to $\frac{1}{B}\mathbb{I}_{(-t_0-B,-t_0)}$,when $\epsilon \to 0$, and $0\le
v_{\epsilon}^{''}(t) \le \frac{2}{B}\mathbb{I}_{(-t_0-B+\epsilon,-t_0-\epsilon)}$
for ant $t \in \mathbb{R}$;\par
(3)$v_{\epsilon}^{'}(t)$ are pointwise convergent to $b(t)$ which is a continuous
function on R when $\epsilon \to 0$ and $0 \le v_{\epsilon}^{'}(t) \le 1$ for any
$t\in \mathbb{R}$.\par
One can construct the family $\{v_\epsilon\}_{\epsilon \in (0,\frac{1}{8}B)}$  by
the setting
\begin{equation}
\begin{split}
v_\epsilon(t):=&\int_{-\infty}^{t}(\int_{-\infty}^{t_1}(\frac{1}{B-4\epsilon}
\mathbb{I}_{(-t_0-B+2\epsilon,-t_0-2\epsilon)}*\rho_{\frac{1}{4}\epsilon})(s)ds)dt_1\\
&-\int_{-\infty}^{0}(\int_{-\infty}^{t_1}(\frac{1}{B-4\epsilon}
\mathbb{I}_{(-t_0-B+2\epsilon,-t_0-2\epsilon)}*\rho_{\frac{1}{4}\epsilon})(s)ds)dt_1
\end{split}
\end{equation}
where $\rho_{\frac{1}{4}\epsilon}$ is the kernel of convolution satisfying
$supp(\rho_{\frac{1}{4}\epsilon})\subset
(-\frac{1}{4}\epsilon,{\frac{1}{4}\epsilon})$.
Then it follows that
\begin{equation}
v_\epsilon^{''}(t)=\frac{1}{B-4\epsilon}
\mathbb{I}_{(-t_0-B+2\epsilon,-t_0-2\epsilon)}*\rho_{\frac{1}{4}\epsilon}(t)
\end{equation}
and
\begin{equation}
v_\epsilon^{'}(t)=\int_{-\infty}^{t}(\frac{1}{B-4\epsilon}
\mathbb{I}_{(-t_0-B+2\epsilon,-t_0-2\epsilon)}*\rho_{\frac{1}{4}\epsilon})(s)ds
\end{equation}
\par
Let $\eta=s(-v_\epsilon(\psi_m))$ and $\phi=u(-v_\epsilon(\psi_m))$, where $s \in
C^{\infty}((T,+\infty))$ satisfies $s\ge 0$ and $u\in C^{\infty}((T,+\infty))$
satisfies $\lim\limits_{t \to +\infty}u(t)$ exists, such that $u''s-s''>0$
and $s'-u's=1$. It follows form $\sup\limits_m \sup\limits_{X_k}\psi_m<-T$  that
$\phi=u(-v_\epsilon(\psi_m))$ are uniformly bounded on $X_k$ with respect to $m$
and $\epsilon$, and $u(-v_\epsilon(\psi))$ are uniformly bounded on $X_k$ with
respect to $\epsilon$.
Let $\Phi=\phi+\varphi_{m'}$ and let $\widetilde{h}=e^{-\Phi}$.
\\

\emph{Step 3: Solving $\bar{\partial}$-equation with smooth polar function and smooth weight}
\\
\par
Set $\textbf{B}=[\eta \sqrt{-1}\Theta_{\widetilde{h}}-\sqrt{-1}\partial \bar{\partial}
\eta-\sqrt{-1}g\partial\eta \wedge\bar{\partial}\eta, \Lambda_{\omega}]$, where
$g$
is a positive continuous function on $X_k$. We will determine $g$ by calculations. As
\begin{equation}
\begin{split}
\partial\bar{\partial}\eta=&
-s'(-v_{\epsilon}(\psi_m))\partial\bar{\partial}(v_{\epsilon}(\psi_m))
+s''(-v_{\epsilon}(\psi_m))\partial(v_{\epsilon}(\psi_m))\wedge
\bar{\partial}(v_{\epsilon}(\psi_m))\\
\eta\Theta_{\widetilde{h}}=&\eta\partial\bar{\partial}\phi+\eta\partial\bar{\partial}\varphi_{m'}\\
=&su''(-v_{\epsilon}(\psi_m))\partial(v_{\epsilon}(\psi_m))\wedge
\bar{\partial}(v_{\epsilon}(\psi_m))
-su'(-v_{\epsilon}(\psi_m))\partial\bar{\partial}(v_{\epsilon}(\psi_m))\\
+&s\partial\bar{\partial}\varphi_{m'}
\end{split}
\end{equation}
Hence
\begin{flalign}
&\eta \sqrt{-1}\Theta_{\widetilde{h}}-\sqrt{-1}\partial \bar{\partial}
\eta-\sqrt{-1}g\partial\eta \wedge\bar{\partial}\eta\\\nonumber
=&s\sqrt{-1}\partial\bar{\partial}\varphi_{m'}
+(s'-su')(v'_{\epsilon}(\psi_m)\sqrt{-1}\partial\bar{\partial}(\psi_m)+
v''_\epsilon(\psi_m)\sqrt{-1}\partial(\psi_m)\wedge\bar{\partial}(\psi_m))\\\nonumber
&+[(u''s-s'')-gs'^2]\partial(v_\epsilon(\psi_m))\wedge\bar{\partial}(v_\epsilon(\psi_m))\\\nonumber
\end{flalign}
\par
Let $g=\frac{u''s-s''}{s'^2}(-v_\epsilon(\psi_m))$ and note that $s'-su'=1$,
$v'_\epsilon(\psi_m)\ge 0$. Then
\begin{equation}
\begin{split}
&\eta \sqrt{-1}\Theta_{\widetilde{h}}-\sqrt{-1}\partial \bar{\partial}
\eta-\sqrt{-1}g\partial\eta \wedge\bar{\partial}\eta\\
=&s\sqrt{-1}\partial\bar{\partial}\varphi_{m'}+v'_{\epsilon}(\psi_m)\sqrt{-1}\partial\bar{\partial}(\psi_m)+
v''_\epsilon(\psi_m)\sqrt{-1}\partial(\psi_m)\wedge\bar{\partial}(\psi_m)\\
\ge&
v''_\epsilon(\psi_m)\sqrt{-1}\partial(\psi_m)\wedge\bar{\partial}(\psi_m)
\end{split}
\end{equation}
\par
Hence
\begin{equation}
\begin{split}
&\langle(\textbf{B}\alpha,\alpha\rangle_{\widetilde h}\\
\ge&\langle[v''_\epsilon(\psi_m)\partial(\psi_m)\wedge\bar{\partial}(\psi_m),
\Lambda_{\omega}]\alpha,\alpha\rangle_{\widetilde h}\\
=&\langle(v''_\epsilon(\psi_m)\bar{\partial}(\psi_m)
\wedge(\alpha\llcorner(\bar{\partial}\psi_m)^{\sharp})),\alpha\rangle_{\widetilde
h}
\end{split}
\label{320}
\end{equation}
\par
By Lemma \ref{semipositive}, \textbf{B} is semipositive.
\par
Using the definition of contraction, Cauchy-Schwarz inequality
and the inequality (\ref{320}), we have
\begin{equation}
\begin{split}
|\langle
v''_\epsilon(\psi_m)\bar{\partial}\psi_m\wedge\gamma,\widetilde{\alpha}\rangle_
{\widetilde h}|^2=
&|\langle
v''_\epsilon(\psi_m)\gamma,\widetilde{\alpha}\llcorner(\bar{\partial}\psi_m)^{\sharp}
\rangle_{\widetilde h}|^2\\
\le&\langle
(v''_\epsilon(\psi_m)\gamma,\gamma)
\rangle_{\widetilde h}
(v''_\epsilon(\psi_m))|\widetilde{\alpha}\llcorner(\bar{\partial}\psi_m)^{\sharp}|^2_{\widetilde
h}\\
=&\langle
(v''_\epsilon(\psi_m)\gamma,\gamma)
\rangle_{\widetilde h}
\langle
(v''_\epsilon(\psi_m))\bar{\partial}\psi_m\wedge
(\widetilde{\alpha}\llcorner(\bar{\partial}\psi_m)^{\sharp}),\widetilde{\alpha}
\rangle_{\widetilde h}\\
\le&\langle
(v''_\epsilon(\psi_m)\gamma,\gamma)
\rangle_{\widetilde h}
\langle
\textbf{B}\widetilde{\alpha},\widetilde{\alpha})
\rangle_{\widetilde h}
\label{411}
\end{split}
\end{equation}
for any $(n,0)$-form $\gamma$ and (n,1)-form $\widetilde{\alpha}$.
\par
As F is holomorphic on $\{\psi < -t_0\}\supset \supset Supp(1-v'_\epsilon(\psi_m))$
, then $\lambda:= D''[(1-v'_{\epsilon}(\psi_m))F] $ is well-defined and smooth on
$X_k$.
\par
Taking  $\gamma=F$, $\widetilde{\alpha}=
\textbf{B}^{-1}(\bar{\partial}v'_{\epsilon}(\psi_m))\wedge
F$. Note that $\widetilde h=e^{-\Phi}$, using inequality (\ref{411}), we have
\begin{equation}
\langle
\textbf{B}^{-1}\lambda,\lambda
\rangle_{\widetilde h}
\le
v''_{\epsilon}(\psi_m)|F|^2e^{-\Phi}
\end{equation}
\par
Then it follows that
\begin{equation}
\int_{X_k \backslash E_{\delta_m}(T)}\langle
\textbf{B}^{-1}\lambda,\lambda
\rangle_{\widetilde h}
\le
\int_{X_k \backslash E_{\delta_m}(T)}
v''_{\epsilon}(\psi_m)|F|^2e^{-\Phi}
< +\infty
\label{323}
\end{equation}
\par
Using Lemma \ref{solveequation}, there exists $u_{k,m,m',\epsilon} \in L^2(X_k,K_X)$ such that
 \begin{equation}
 D''u_{k,m,m',\epsilon}=\lambda
 \label{324}
 \end{equation}
 and
\begin{equation}
\begin{split}
\int_{X_k}
\frac{1}{\eta+g^{-1}}|u_{k,m,m',\epsilon}|^2 e^{-\Phi}
\le
\int_{X_k}
\langle
\textbf{B}^{-1}\lambda,\lambda
\rangle_{\widetilde h}
\le
\int_{X_k}
v''_{\epsilon}(\psi_m)|F|^2 e^{-\Phi}
\label{325}
\end{split}
\end{equation}
\par
Note that $g=\frac{u''s-s''}{s'^2}(-v_\epsilon(\psi_m))$. Assume that we can choose
$\eta$ and $\phi$ such that
$e^{v_\epsilon(\psi_m)}e^{\phi}c(-v_\epsilon(\psi_m))=(\eta+g^{-1})^{-1}$. Then
inequality (\ref{325}) becomes
\begin{equation}
\begin{split}
\int_{X_k}
|u_{k,m,m',\epsilon}|^2
e^{v_\epsilon(\psi_m)-\varphi_{l,m'}}c(-v_\epsilon(\psi_m))
\le
\int_{X_k}
v''_{\epsilon}(\psi_m)|F|^2 e^{-\phi-\varphi_{m'}}
<+\infty
\end{split}
\label{326}
\end{equation}
\par
Let $F_{k,m,m',\epsilon}:=-u_{k,m,m',\epsilon}+(1-v'_{\epsilon}(\psi_m))F$. Then inequality \eqref{326} becomes
\begin{equation}
\begin{split}
&\int_{X_k}
|F_{k,m,m',\epsilon}-(1-v'_{\epsilon}(\psi_m))F|^2
e^{v_\epsilon(\psi_m)-\varphi_{l,m'}}c(-v_\epsilon(\psi_m))\\
\le&
\int_{X_k}
v''_{\epsilon}(\psi_m)|F|^2 e^{-\phi-\varphi_{m'}}
\end{split}
\label{415}
\end{equation}

\emph{Step 4: Singular polar function and smooth weight}
\par
As $\sup_{m,\epsilon}|\phi|=\sup_{m,\epsilon}|u(-v_{\epsilon}(\psi_m))|<+\infty$ and $\varphi_{m'}$ is continuous on $\overline{X}_k$, then $\sup_{m,\epsilon}e^{-\phi-\varphi_{m'}}<+\infty$. Note that
$$
v''_{\epsilon}(\psi_m)|F|^2e^{-\phi-\varphi_{m'}}\le\frac{2}{B}
\mathbb{I}_{\{\psi<-t_0\}}|F|^2\sup_{m,\epsilon}e^{-\phi-\varphi_{m'}}
$$
on $X_k$, then it follows from $\int_{\{\psi<-t_0\}}|F|^2<+\infty$ and the dominated convergence theorem that
\begin{equation}
\begin{split}
\lim_{m\to+\infty}\int_{X_k}
v''_{\epsilon}(\psi_m)|F|^2 e^{-\phi-\varphi_{m'}}=\int_{X_k}v''_{\epsilon}(\psi)|F|^2 e^{-u(-v_{\epsilon}(\psi))-\varphi_{m'}}
\end{split}
\label{416}
\end{equation}
\par
Note that $\inf_m\inf_{X_k} e^{v_\epsilon(\psi_m)-\varphi_{m'}}c(-v_\epsilon(\psi_m))>0$, then it follows from inequality \eqref{415} and \eqref{416} that $\sup_m\int_{X_k}|F_{k,m,m',\epsilon}-(1-v'_{\epsilon}(\psi_m))F|^2<+\infty$. Note that
\begin{equation}
\begin{split}
|(1-v'_{\epsilon}(\psi_m))F|\le|\mathbb{I}_{\{\psi<-t_0\}}F|,
\end{split}
\label{417}
\end{equation}
then it follows from $\int_{\{\psi<-t_0\}}|F|^2<+\infty$ that  $\sup_m\int_{X_k}|F_{k,m,m',\epsilon}|^2<+\infty$, which implies that there exists a subsequence of $\{F_{k,m,m',\epsilon}\}_m$ (also denoted by $F_{k,m,m',\epsilon}$) compactly convergent to a holomorphic $F_{k,m',\epsilon}$ on $X_k$.
\par
Note that $v_{\epsilon}(\psi_m)-\varphi_{m'}$ are uniformly bounded on $X_k$ with respect to $m$, then it follows from $|F_{k,m,m',\epsilon}-(1-v'_{\epsilon}(\psi_m))F|^2\le
2(|F_{k,m,m',\epsilon}|^2+|(1-v'_{\epsilon}(\psi_m))F|^2)\le
2(|F_{k,m,m',\epsilon}|^2+|\mathbb{I}_{\{\psi<-t_0\}}F|^2)$ and the dominated convergence theorem that
\begin{equation}
\begin{split}
\lim_{m\to+\infty}&\int_K|F_{k,m,m',\epsilon}-(1-v'_{\epsilon}(\psi_m))F|^2
e^{v_\epsilon(\psi_m)-\varphi_{m'}}c(-v_\epsilon(\psi_m))\\
=&\int_K|F_{k,m',\epsilon}-(1-v'_{\epsilon}(\psi))F|^2
e^{v_\epsilon(\psi)-\varphi_{m'}}c(-v_\epsilon(\psi))
\end{split}
\label{418}
\end{equation}
holds for any compact subset $K$ of $X_k$. Combing with inequality \eqref{415} and
\eqref{416}, one can obtain that
\begin{equation}
\begin{split}
&\int_K|F_{k,m',\epsilon}-(1-v'_{\epsilon}(\psi))F|^2
e^{v_\epsilon(\psi)-\varphi_{m'}}c(-v_\epsilon(\psi))\\
\le &\int_{X_k}v''_{\epsilon}(\psi)|F|^2 e^{-u(-v_{\epsilon}(\psi))-\varphi_{m'}}
\end{split}
\label{419}
\end{equation}
which implies
\begin{equation}
\begin{split}
&\int_{X_k}|F_{k,m',\epsilon}-(1-v'_{\epsilon}(\psi))F|^2
e^{v_\epsilon(\psi)-\varphi_{m'}}c(-v_\epsilon(\psi))\\
\le &\int_{X_k}v''_{\epsilon}(\psi)|F|^2 e^{-u(-v_{\epsilon}(\psi))-\varphi_{m'}}
\end{split}
\label{420}
\end{equation}

 \emph{Step 5: Nonsmooth cut-off function}
 \\
 \par
Note that $\sup_{\epsilon}sup_{X_k}e^{-u(-v_{\epsilon}(\psi))-\varphi_{m'}}<+\infty$ and $$
v''_{\epsilon}(\psi)|F|^2 e^{-u(-v_{\epsilon}(\psi))-\varphi_{m'}}\le
\frac{2}{B}\mathbb{I}_{\{-t_0-B<\psi<-t_0\}}|F|^2\sup_{\epsilon}sup_{X_k}e^{-u(-v_{\epsilon}(\psi))-\varphi_{m'}}
$$
then it follows from $\int_{\{\psi<-t_0\}}|F|^2<+\infty$ and the dominated convergence theorem that
\begin{equation}
\begin{split}
&\lim_{\epsilon\to 0}\int_{X_k}v''_{\epsilon}(\psi)|F|^2 e^{-u(-v_{\epsilon}(\psi))-\varphi_{m'}}\\
=&\int_{X_k}\frac{1}{B}\mathbb{I}_{\{-t_0-B<\psi<-t_0\}}|F|^2 e^{-u(-v(\psi))-\varphi_{m'}}\\
\le&(\sup_{X_k}e^{-u(-v(\psi))})
\int_{X_k}\frac{1}{B}\mathbb{I}_{\{-t_0-B<\psi<-t_0\}}|F|^2 e^{-\varphi_{m'}}<+\infty
\end{split}
\label{421}
\end{equation}
Note that $\inf_{\epsilon}\inf_{X_k} e^{v_\epsilon(\psi)-\varphi_{m'}}c(-v_\epsilon(\psi))>0$, then it follows from inequality \eqref{420} and \eqref{421} that $\sup_{\epsilon}\int_{X_k}|F_{k,m',\epsilon}-(1-v'_{\epsilon}(\psi))F|^2<+\infty$. Combing with
\begin{equation}
\begin{split}
\sup_{\epsilon}\int_{X_k}|1-v'_{\epsilon}(\psi)F|^2
\le\int_{X_k}\mathbb{I}_{\{\psi<-t_0\}}|F|^2<+\infty
\end{split}
\label{422}
\end{equation}
one can obtain that $\sup_{\epsilon}\int_{X_k}|F_{k,m',\epsilon}|^2<+\infty$, which implies that there exists a subsequence of $\{F_{k,m',\epsilon}\}_{\epsilon \to 0}$ (also denoted by $F_{k,m',\epsilon}$) compactly convergent to a holomorphic $F_{k,m'}$ on $X_k$.
\par
Note that $\sup_{\epsilon}\sup_{X_k}e^{v_\epsilon(\psi)-\varphi_{m'}}c(-v_{\epsilon}(\psi))<+\infty$ and $|F_{k,m',\epsilon}-(1-v'_{\epsilon}(\psi_m))F|^2$ $\le
2(|F_{k,m',\epsilon}|^2+|\mathbb{I}_{\{\psi<-t_0\}}F|^2)$, then it follows from inequality \eqref{422} and dominated convergence theorem on any given $K\subset\subset X_k$, with dominant function
$$2(\sup_{\epsilon}\sup_K(|F_{k,m',\epsilon}|)+|\mathbb{I}_{\{\psi<-t_0\}}F|^2)
\sup_{\epsilon}\sup_{X_k}e^{v_\epsilon(\psi)-\varphi_{m'}}c(-v_{\epsilon}(\psi))$$
that
\begin{equation}
\begin{split}
\lim_{\epsilon\to 0}&\int_K|F_{k,m',\epsilon}-(1-v'_{\epsilon}(\psi))F|^2
e^{v_\epsilon(\psi)-\varphi_{m'}}c(-v_\epsilon(\psi))\\
=&\int_K|F_{k,m'}-(1-b(\psi))F|^2
e^{v(\psi)-\varphi_{m'}}c(-v(\psi))
\end{split}
\label{423}
\end{equation}
Combing with inequality \eqref{421} and \eqref{420}, one can obtain that
\begin{equation}
\begin{split}
&\int_K|F_{k,m'}-(1-b(\psi))F|^2
e^{v(\psi)-\varphi_{m'}}c(-v(\psi))\\
\le &(\sup_{X_k}e^{-u(-v(\psi))})
\int_{X_k}\frac{1}{B}\mathbb{I}_{\{-t_0-B<\psi<-t_0\}}|F|^2 e^{-\varphi_{m'}}
\end{split}
\label{424}
\end{equation}
which implies
\begin{equation}
\begin{split}
&\int_{X_k}|F_{k,m'}-(1-b(\psi))F|^2
e^{v(\psi)-\varphi_{m'}}c(-v(\psi))\\
\le &(\sup_{X_k}e^{-u(-v(\psi))})
\int_{X_k}\frac{1}{B}\mathbb{I}_{\{-t_0-B<\psi<-t_0\}}|F|^2 e^{-\varphi_{m'}}
\end{split}
\label{425}
\end{equation}

\emph{Step 6: Singular weight}
\\
Note that
\begin{equation}
\begin{split}
\int_{X_k }
\frac{1}{B}\mathbb{I}_{\{-t_0-B<\psi<-t_0\}}|F|^2 e^{-\varphi_{m'}}
\le
\int_{X_k }
\frac{1}{B}\mathbb{I}_{\{-t_0-B<\psi<-t_0\}}|F|^2 e^{-\varphi}<+ \infty
\end{split}
\label{426}
\end{equation}
and $\sup_{X_k} e^{-u(-v(\psi))}< +\infty$
then it follows from \eqref{425}  that
$$
 \sup_{m'}\int_{X_k}|F_{k,m'}-(1-b(\psi))F|^2
 e^{v(\psi)-\varphi_{m'}}c(-v(\psi)) < + \infty.
$$
 Combining with $\inf_{m'}\inf_{X_k}
 e^{v(\psi)-\varphi_{m'}}c(-v(\psi)) >0$, we know
 $\sup_{m'}\int_{X_k}|F_{k,m'}-(1-b(\psi))F|^2 < +\infty$.
 Note that
\begin{equation}
\begin{split}
\int_{X_k}|(1-b(\psi))F|^2\le \int_{X_k}|\mathbb{I}_{\{\psi<-t_0\}}F|^2< + \infty
\label{427}
\end{split}
\end{equation}
Then $\sup_{m'} \int_{X_k}|F_{k,m'}|^2 < +\infty$, which implies that
there exists a compactly convergence subsequence of $\{F_{k,m'}\}$  denoted by
$\{F_{k,m''}\}$ which converge to a holomorphic $(n,0)$ form on $X_k$ denoted by
$F_{k}$.\par
Note that $\sup_{m'}\sup_{X_k} e^{v(\psi)-\varphi_{m'}}c(-v(\psi))<
+\infty$, then it follows \eqref{427} and the dominated convergence theorem on any given compact subset K of $X_k$ with dominant function
\begin{equation}
2(\sup_{m''}\sup_{K}(|F_{k,m''}|^2)
+\mathbb{I}_{\{\psi<-t_0\}}|F|^2)
\sup_{X_k}e^{v(\psi)-\varphi_{m'}}c(-v(\psi))
\end{equation}
that
\begin{equation}
\begin{split}
\lim_{m'' \to +\infty}
&\int_{K}
|F_{k,m''}-(1-b(\psi))F|^2 e^{v(\psi)-\varphi_{m'}}c(-v(\psi))
\\
=&\int_{K}
|F_{k}-(1-b(\psi))F|^2 e^{v_(\psi)-\varphi_{m'}}c(-v(\psi))
\label{428}
\end{split}
\end{equation}
Note that for $m''\ge m'$, $\varphi_{m'} \le\varphi_{m''}$ holds, then it follows from \eqref{425}
and \eqref{426} that
\begin{equation}
\begin{split}
&\lim_{m'' \to +\infty}
\int_{K}
|F_{k,m''}-(1-b(\psi))F|^2 e^{v(\psi)-\varphi_{m'}}c(-v(\psi))
\\
\le&
\limsup_{m'' \to +\infty}
\int_{K}
|F_{k,m''}-(1-b(\psi))F|^2 e^{v(\psi)-\varphi_{m''}}c(-v(\psi))\\
\le&
\limsup_{m''\to +\infty}
(\sup_{X_k}e^{-u(-v(\psi))})\int_{X_k}
\frac{1}{B}\mathbb{I}_{\{-t_0-B<\psi<-t_0\}}|F|^2e^{-\varphi_{m''}}\\
\le&
(\sup_{X_k}e^{-u(-v(\psi))})\int_{X_k}
\frac{1}{B}\mathbb{I}_{\{-t_0-B<\psi<-t_0\}}|F|^2e^{-\varphi}<+ \infty
\end{split}
\label{429}
\end{equation}
Combining with equality \eqref{428}, one can obtain that
\begin{equation}\nonumber
\begin{split}
&\int_{K}
|F_{k}-(1-b(\psi))F|^2 e^{v(\psi)-\varphi_{m'}}c(-v(\psi))\\
\le&
(\sup_{X_k}e^{-u(-v(\psi))})\int_{X_k}
\frac{1}{B}\mathbb{I}_{\{-t_0-B<\psi<-t_0\}}|F|^2e^{-\varphi}<+ \infty
\end{split}
\end{equation}
for any compact subset K of $X_k$, which implies
\begin{equation}\nonumber
\begin{split}
&\int_{X_k}
|F_{k}-(1-b(\psi))F|^2 e^{v(\psi)-\varphi_{m'}}c(-v(\psi))\\
\le&
(\sup_{X_k}e^{-u(-v(\psi))})\int_{X_k}
\frac{1}{B}\mathbb{I}_{\{-t_0-B<\psi<-t_0\}}|F|^2e^{-\varphi}<+ \infty
\end{split}
\end{equation}
When $m' \to +\infty$, it follows from Levi's theorem that
\begin{equation}
\begin{split}
&\int_{X_k}
|F_{k}-(1-b(\psi))F|^2 e^{v(\psi)-\varphi}c(-v(\psi))\\
\le&
(\sup_{X_k}e^{-u(-v(\psi))})\int_{X_k}
\frac{1}{B}\mathbb{I}_{\{-t_0-B<\psi<-t_0\}}|F|^2e^{-\varphi}
\end{split}
\label{430}
\end{equation}

\emph{Step 7: ODE System}
\\
\par
We want to find $\eta$ and $\phi$ such that
$(\eta+g^{-1})=e^{-v_\epsilon(\psi_m)}e^{-\phi}\frac{1}{c(-v_{\epsilon}(\psi_m))}$.
As $\eta=s(-v_{\epsilon}(\psi_m))$ and $\phi=u(-v_{\epsilon}(\psi_m))$, we have
$(\eta+g^{-1})e^{v_\epsilon(\psi_m)}e^{\phi}=(s+\frac{s'^2}{u''s-s''})e^{-t}e^u\circ(-v_\epsilon(\psi_m))$.\\

Summarizing the above discussion about s and u,we are naturally led to a system of
ODEs:
\begin{equation}
\begin{split}
&1)(s+\frac{s'^2}{u''s-s''})e^{u-t}=\frac{1}{c(t)}\\
&2)s'-su'=1
\end{split}
\label{431}
\end{equation}
when $t\in(T,+\infty)$.\par
It is not hard to solve the ODE system \eqref{431} and get
$u(t)=-\log(\int^t_T c(t_1)e^{-t_1}dt_1)$ and
$s(t)=\frac{\int^t_T(\int^{t_2}_T c(t_1)e^{-t_1}dt_1)dt_2}{\int^t_T
c(t_1)e^{-t_1}dt_1)}$. It follows that $s\in C^{\infty}((T,+\infty))$ satisfies
$s \ge 0$, $\lim_{t \to +\infty}u(t)=-\log(\int^{+\infty}_T
c(t_1)e^{-t_1}dt_1)$ exists and $u\in C^{\infty}((T,+\infty))$ satisfies
$u''s-s''>0$.\par
As $u(t)=-\log(\int^t_T c(t_1)e^{-t_1}dt_1)$ is decreasing with respect to t, then
it follows from $-T \ge v(t) \ge \max\{t,-t_0-B_0\} \ge -t_0-B_0$,for any $t \le 0$
that
\begin{equation}\nonumber
\begin{split}
\sup_{X_k}e^{-u(-v(\psi))} \le
\sup_{X}e^{-u(-v(\psi))} \le
\sup_{t\in(T,t_0+B]}e^{-u(t)}=\int^{t_0+B}_T c(t_1)e^{-t_1}dt_1
\end{split}
\end{equation}
Hence on $X_k$, we have
\begin{equation}
\begin{split}
&\int_{X_k}
|F_{k}-(1-b(\psi))F|^2 e^{v(\psi)-\varphi}c(-v(\psi))\\
\le&
\int^{t_0+B}_T c(t_1)e^{-t_1}dt_1\int_{X_k}
\frac{1}{B}\mathbb{I}_{\{-t_0-B<\psi<-t_0\}}|F|^2e^{-\varphi}\\
\le&
C\int^{t_0+B}_T c(t_1)e^{-t_1}dt_1
\end{split}
\label{432}
\end{equation}

\emph{Step 8: When $k \to +\infty$.}
\\
\par

Note that for any given $k$, $e^{-\varphi+v(\psi)}c(-v(\psi))$ has a positive lower
bound on $\bar{X_k}$, then it follows \eqref{432} that for any given $k$ ,
$\int_{X_k}
|F_{k'}-(1-b(\psi))F|^2$ is bounded with respect to $k' \ge k$.
Combining with
\begin{equation}
\int_{X_k}|1-b(\psi)F|^2 \le
\int_{\bar {X_k}\cap\{\psi<-t_0\}}|F|^2<+ \infty
\label{433}
\end{equation}
One can obtain that $\int_{X_k}
|F_{k'}|^2$ is bounded with respect to $k' \ge k$.\par
By diagonal method, there exists a subsequence $F_{k''}$ uniformly converge on any
$\bar{X_k}$ to a holomorphic (n,0)-form on X denoted by $\widetilde F$. Then it
follow from inequality \eqref{432}, \eqref{433} and the dominated convergence theorem that
\begin{equation}
\begin{split}
&\int_{X_k}
|\widetilde F-(1-b(\psi))F|^2 e^{-\max\{\varphi-v(\psi),-M\}}c(-v(\psi))\\
\le&
C\int^{t_0+B}_T c(t_1)e^{-t_1}dt_1
\end{split}
\end{equation}
for any $M>0$, then Levi's theorem implies
\begin{equation}
\begin{split}
&\int_{X_k}
|\widetilde F-(1-b(\psi))F|^2 e^{-(\varphi-v(\psi))}c(-v(\psi))\\
\le&
C\int^{t_0+B}_T c(t_1)e^{-t_1}dt_1
\end{split}
\end{equation}
Let $k \to +\infty$, Lemma 2.1 is proved.

\subsection {Appendix: Proof of Lemma \ref{lemma for cor1.9}}
\label{4.2}
We will need the following results in our proof of Lemma \ref{lemma for cor1.9}..
\begin{Lemma}(see Chapter 3, Corollary 2.14 of \cite{Demaillybook}) Let $A$ be an analytic subset of $X$ with global irreducible components
$A_j$ of pure dimension $p$. Then any closed current $\Theta \in \mathcal{D}'_{p,p}(X)$ of order $0$ with support in $A$ is of the form $\Theta=\sum\lambda_j A_j$ where $\lambda_j \in \mathbb{C}$.
\label{lemma4.5}
\end{Lemma}

\begin{Lemma}Let $X$ be an open Riemann Surface which admits a nontrivial Green function $G_{X}(z,w)$. Given $z_1\in X$. Let $U$ be a relatively compact open subset of $X$ containing $z_1$. Denote $A=\sup\limits_{z\in \overline{U}}G_{X}(z,z_1)$. Then for any $z\in X\backslash \overline{U}$, we have $G_X(z,z_1)\ge A$.
\label{lemma4.6}
\end{Lemma}
\begin{proof}
We prove Lemma \ref{lemma4.6} by contradiction. If there exists $z_0 \in X\backslash \overline{U}$ such that $G_X(z_0,z_1) < A$.
Note that $G_X(z,z_1)$ is harmonic on $X\backslash \overline{U}$, hence smooth at $z_0\in X\backslash \overline{U}$. Then there exists a small open neighborhood $W$ of $z_0 $ such that for any $z\in W$, $G_{X}(z,z_1)<A$.
\par
Denote
$$\tilde{G}(z) = \left\{ \begin{array}{rcl}
G_X(z,z_1) & \mbox{for}
& z \in \overline{U} \\ \max\{G_X(z,z_1),A\} & \mbox{for} & z\in X\backslash\overline{U}
\end{array}\right.
$$
Note that $\tilde{G}(z)$ is a negative subharmonic function on $X$ and $\tilde{G}(z)=G_X(z,z_1)$ on $U$.
\par
Let $(V_{z_1},w)$ be a small local coordinate neighborhood of $z_1$ such that $w(z_1)=0$. By definition, $G_X(z,z_1)=\sup\limits_{v \in \Delta(z_1)} v(z)$ where $\Delta(z_1)$ is the set of negative subharmonic functions on $X$ satisfying that $v-\log|w|$ has locally finite upper bound near $z_1$.
\par
As $\tilde{G}(z)=G_X(z,z_1)$ on $U$, we know $\tilde{G}(z)\in \Delta(z_1)$, but $\tilde{G}(z)>G_X(z,z_1)$ on $W$ which is contradict to the fact that $G_X(z,z_1)=\sup\limits_{v \in \Delta(z_1)} v(z)$. Hence for any $z\in X\backslash \overline{U}$, we have $G_{X}(z,z_1)\ge A$.

\end{proof}

\begin{Lemma}Let $X$ be an open Riemann Surface which admits a nontrivial Green function $G_{X}(z,w)$. Fix $z_1\in x$, there exists open subsets $V_1,U_1$ which satisfy $z_1\in V_1 \subset\subset U_1 \subset\subset X$ and a constant $N>0$ such that $\forall (z,w)\in (X\backslash U_1)\times V_1$, we have
$$G_X(z,w)\ge NG(z,z_1).$$
\label{lemma4.7}
\end{Lemma}
Lemma \ref{lemma4.7} can be deduced from the Harnack inequality of harmonic function. For the convenience of readers, we give another proof as below.
\begin{proof}[Proof of Lemma \ref{lemma4.7}]
Let $(V_{z_1},w)$ be a small local coordinate neighborhood of $z_1$ such that $V_{z_1}\subset\subset X$, $w(z_1)=0$ and $G(z,z_1)|_{V_{z_1}}=\log|w|$.
\par
Let $V_1 \subset\subset U_1 \subset\subset V_{z_1}$ such that
$$\sup\limits_{z\in \overline{U_1}}G(z,z_1)=-t_0$$
and
$$\sup\limits_{z\in \overline{V_1}}G(z,z_1)=-t_0-1$$
for some $t_0\ge 0$.
Denote $W=\{z\in V_{z_1}|-t_0-\frac{1}{4}<G_X(z,z_1)<-t_0\}$. Then it is easy to see that $W\subset \subset V_{z_1}$ and $\overline{W}\cap\overline{V_1}=\emptyset$.
\par
Note that when $z\in \overline{W}, w\in \overline{V_1}$, $G_{X}(z,w)$ is smooth. Hence $G_{X}(z,w)$ has a lower bounded $B$ on $\overline{W}\times \overline{V_1}$.
Denote $a=\sup\limits_{z\in \overline{W}}G_X(z,z_1)$.
Let $N$ be big enough such that $B\ge N a$.
\par
Fix $w\in V_1$. We will show for any $z\in X\backslash U_1$,
$$G_X(z,w)\ge NG_X(z,z_1).$$
\par
If not, there exists $z_0\in X\backslash U_1$ such that $G_X(z_0,w)< NG_X(z_0,z_1)$. Note that both $G_X(z,w)$ and $G_X(z,z_1)$ are smooth on $X\backslash U_1$. Hence there exists a open neighborhood $H_{z_0}$ of $z_0$ such for any $z\in H_{z_0}$, we have
$$G_X(z,w)< NG_X(z,z_1).$$
Let
$$G_0(z) = \left\{ \begin{array}{rcl}
G_X(z,w) & \mbox{for}
& z \in \overline{U_1} \\ \max\{G_X(z,w),NG_X(z,z_1)\} & \mbox{for} & z\in X\backslash\overline{U_1}
\end{array}\right.
$$
\par
Then $G_0(z)$ is a nonnegative subharmonic function on $X$. Note that $w\in V_1\subset U_1$ and we have $G_0(z)=G_X(z,w) $ on $U_1$. But $G_0(z)>G_{X}(z,w)$ on $H_{z_0}$, this contradicts to the fact that $G_X(z,w)$ is the Green function of $X$ with pole at $w$.
\par
Hence $G_X(z,w)\ge NG_X(z,z_1)$ for any $z\in X\backslash U_1$ holds. As $w\in V_1$ is arbitrarily fixed, we know for any $(z,w) \in (X\backslash U_1)\times V_1$, we have
$$G_X(z,w)\ge NG_X(z,z_1).$$
\par
Lemma \ref{lemma4.7} is proved.
\end{proof}

Now we begin to prove Lemma \ref{lemma for cor1.9}.
\begin{proof}[Proof of Lemma \ref{lemma for cor1.9}]
Recall that by Siu's decomposition theorem, we have
$$\frac{i}{\pi}\partial\bar{\partial}\varphi=\sum\limits_{j\ge 1} \lambda_j[x_j]+R,\ \lambda_j>0$$
where $x_j\in X$ is a point, $\lambda_j=v(i\partial\bar{\partial}\varphi,x_j)$ is the Lelong number of $i\partial\bar{\partial}\varphi$ at $x_j$, R is a closed positive $(1,1)$ current with $v(R,x)=0$ for $x \in X$.
\par
Recall that both $E_1(i\partial\bar{\partial}\varphi)=\{x\in X|\  v(i\partial\bar{\partial}\varphi,x)\ge 1\}$ and $E=\{x\in X|\ v(i\partial\bar{\partial}\varphi,x)$ is a positive integer $\}$  are sets of isolated points.
As $(i\partial\bar{\partial}\varphi) |_{X\backslash E}\neq 0$, there are two cases:
\par
(1) There exists $ \lambda_{j_0}$ such that $\lambda_{j_0}-[ \lambda_{j_0}]>0$, where $[ \lambda_{j_0}]$ is the largest integer smaller than $\lambda_{j_0}$.
\par
(2) $R \neq 0$.
\par
For the case (1):
\par
Let $p=x_{j_0}$. Let $(U,z)$ be a relative compact coordinate neighborhood of $p$ in $X$ and by shrinking $U$, we assume that under the local homomorphism, $z(p)=o$ and $z(U)\cong B(0,2)$. We also assume that $U \cap (E_1(i\partial\bar{\partial}\varphi)\backslash \{x_{j_0}\} )=\emptyset$. Let $\theta$ be a smooth cut-off function on $X$ such that $0\le \theta \le 1-\frac{[\lambda_j]}{\lambda_j}$, $supp(\theta)\subset\subset V \subset\subset U$ and $\theta \equiv 1-\frac{[\lambda_j]}{\lambda_j}$ on $W$, where $z(W)=B(0,\frac{1}{4})$ and $z(V)=B(0,\frac{1}{2})$  under the local homomorphism. By shrinking $V$ and $U$ again, it follows from Lemma \ref{lemma4.7} that there exists $N>0$ such that for
any $(z,w) \in (X\backslash U)\times V$, we have
$$G_X(z,w)\ge NG_X(z,x_{j_0}).$$
\par
Let $T=\theta\cdot i\partial\bar{\partial}\varphi$, then $T$ is a closed positive $(1,1)$ current on $X$ with support $supp T\subset\subset V$.
\par
Let $\rho\in C^{\infty}(\mathbb{C})$ be a function with $supp \rho \subset B(0,1)$ and $\rho(z)$ depends only on $|z|$, $\rho \ge 0$ and $\int_{\mathbb{C}}\rho(z)d\lambda_z=1$. Let $\rho_{n}(z)=\frac{1}{n}\rho(\frac{z}{n})$, $\rho_n$ is a family of smoothing kernels.
\par
Let $T_n=T\ast \rho_n$ be the convolution of $T$. For any test function $h\in C^{\infty}(X)$, as $T$ has compact support and $supp T\subset\subset V \subset\subset U$, we can restrict $h$ to $U$ and denote $h|_{U}$ still by $h$ for simplicity. By the definition of convolution of currents, we have $\langle T_n(w),h(w)\rangle:=<T(w), h\ast \rho_n(w)>$. Note that $supp T\subset\subset V$, the convolution $h\ast \rho_n(w)$ is well defined for $w\in V$.
\par
We restrict $2G_X(z,w)$ to $U$ and denote $2G_X(z,w)|_{U}$ still by $2G_X(z,w)$ for simplicity. Let $u_n(z)=\langle T_n(w),2G_X(z,w)\rangle$. For fixed $z$ and fixed $n$, we will prove $\langle T_n(w),2G_X(z,w)\rangle=\langle T(w),2G_X(z,w)\ast\rho_n \rangle$.
\par
For fixed $z$, $G_X(z,w)$ is a subharmonic function on $X$. There exists a sequence of smooth subharmonic functions $G_m(w)$ decreasingly converge to $G_X(z,w)$ with respect to $m$. We still denote $G_m(w)|_{U}$ by $G_m(w)$. As $G_m(w)$ is smooth, we have
\begin{equation}\langle T_n(w),2G_m(w)\rangle=<T(w), 2G_m\ast \rho_n(w)>
\label{4.39}
\end{equation}
\par
For fixed $n$, $T_n(w)$ is a smooth positive $(1,1)-$form on $X$ with $supp T_n \subset\subset U$. As $G_m(w)$ decreasingly converge to $G_X(z,w)$ with respect to $m$, it follows from Levi's theorem that
\begin{equation}
\lim\limits_{m \to +\infty}\langle T_n(w),2G_m(w)\rangle=<T_n(w), 2G_X(z,w)>
\label{4.40}
\end{equation}
\par
For fixed n, as $G_m(w)$ decreasingly converge to $G_X(z,w)$ with respect to $m$ and $\rho_n$ has compact support, we know $(2G_m\ast \rho_n)(w)$ decreasingly converge to $(2G_X(z,w)\ast \rho_n)(w)$ with respect to $m$. Note that $T$ is a positive $(1,1)$ current on $X$ with compact support, hence $T$ is of order 0. It follows from Levi's theorem that
\begin{equation}
\lim\limits_{m \to +\infty}\langle T(w),(2G_m\ast \rho_n)(w)\rangle=<T(w), (2G_X(z,w)\ast \rho_n)(w))>
\label{4.41}
\end{equation}
\par
For fixed $z$ and fixed $n$, it follows from equality \eqref{4.39},\eqref{4.40} and \eqref{4.41} that we have
$\langle T_n(w),2G_X(z,w)\rangle=\langle T(w),(2G_X(z,w)\ast\rho_n)(w) \rangle$.
\par
As $2G_X(z,w)$ is subharmonic, then $2G_X(z,w)\ast\rho_n$ converges to $2G_X(z,w)$ decreasingly with respect to $n$. Note that $T$ is a positive $(1,1)$ current on $X$, hence $u_n(z)$ is decreasing with respect to $n$ and $u_n(z)<0$. Let $u(z)=\lim\limits_{n\to +\infty}u_n(z)$. We know $u(z)<0$.
\par
Now we show that both $\{u_n\}$ and $u$ is $L^1_{loc}$ function on $X$. Let $(K,z)$ be a relatively compact open neighborhood of some point $z'$ in $X$ such that under the local coordinate $z$, we have $z(z')=0$ and $K\cong B(0,1)$. Let $d\lambda_z$ be the lebesgue measure $B(0,1)$.

 Note that for fixed $w$, $2G_X(z,w)$ is smooth outside $z=w$ and $2G_X(z,w)=2\log|z-w|+2u(z)$ on a small neighborhood $B(w,\epsilon_0)$ of $w$, where $u(z)$ is a smooth function on $B(w,\epsilon_0)$. We also note that $\int_{z\in B(w,\epsilon_0)}|2\log|z-w||d\lambda_z<+\infty$. It follows from Fubini theorem that
\begin{equation}\nonumber
\begin{split}
\|u_n\|_{L^1(K)}&=\int_{z\in K}(\int_{w\in U} 2|G_X(z,w)|T_n(w))d\lambda_z\\
&=\int_{w\in U}(\int_{z\in K}2|G_X(z,w)d\lambda_z)T_n(w)
\end{split}
\end{equation}
\par
Let $H(w)=\int_{z\in K}2|G_X(z,w)|d\lambda_z$. If $\bar{U}\cap \bar{K}=\emptyset$, then $G_X(z,w)$ is smooth on $\bar{U}\times \bar{K}$, hence $H(w)$ is uniformly bounded on $w\in U$.
\par
When $\bar{U}\cap \bar{K}\neq\emptyset$, as $U,K$ is small, we assume that there exists an open subset $J\subset \subset X$, such that the set $K+U:=\{z+w|z\in K\  and\ w\in U\}$ is contained in $J$ and we have  $G_X(z,w)=\log|z-w|+u(z,w)$ for $(z,w)\in J\times J$. Here, when $w$ is fixed,
$u(z,w)$ is harmonic function on $z\in J$ and when $z$ is fixed, $u(z,w)$ is harmonic function on $w\in J$. We have
\begin{equation}\nonumber
\begin{split}
H(w)&=\int_{z\in K}2|G_X(z,w)|d\lambda_z\\
&=\int_{z\in K}-2G_X(z,w)d\lambda_z\\
&=\int_{z\in K}-2\log|z-w|-2u(z,w)d\lambda_z\\
&=I_1(w)+I_2(w)
\end{split}
\end{equation}
where $I_1(w)=\int_{z\in K}-2\log|z-w|d\lambda_z$ and $I_2(w)=\int_{z\in K}-2u(z,w)d\lambda_z$.
For $I_1(w)$, we have
\begin{equation}\nonumber
\begin{split}
I_1(w)&=\int_{z\in K}-2\log|z-w|d\lambda_z\\
&=\int_{z\in K+\{w\}}-2\log|z|d\lambda_z\\
&\ge \int_{z\in J}-2\log|z|d\lambda_z
\end{split}
\end{equation}
where the set $K+\{w\}:=\{z+w| z\in K\}$. Note that $\log|z|$ is integrable near $z=0$ and $J$ is relative compact in $X$, hence there exists a constant $M_1>0$ such that $I_1(w)\le M_1$ for any $w\in U$.
\par
For $I_2(w)$, by the mean value equality of harmonic function, we have
\begin{equation}\nonumber
\begin{split}
I_2(w)&=\int_{z\in B(0,1)}-2u(z,w)d\lambda_z\\
&=-2\pi u(z',w)
\end{split}
\end{equation}
As $u(z',w)$ is harmonic on $w\in U$ and $U$ is relatively compact in $X$, we know $I_2(w)$ is bounded on $\bar{U}$.
\par
The above discussion shows that the function $H(w)=\int_{z\in K}2|G_X(z,w)|d\lambda_z=I_1+I_2$ is bounded by some constant $N$ on $U$. Let $\chi $ be a $C^{\infty}_c(X)$ such that $0\le\chi\le1$ and $\chi|_{U}\equiv 1$. Then we have
\begin{equation}\nonumber
\begin{split}
\|u_n\|_{L^1(K)}&=\int_{w\in U}(\int_{z\in K}2|G_X(z,w)d\lambda_z)T_n(w)\\
&=\int_{w\in U}H(w)T_n(w)\\
&\le N \int_{w\in U}T_n(w)\\
&\le N \langle T_n(w),\chi \rangle\\
&=N \langle T(w),\chi*\rho_n \rangle\\
&\le N \|T\|<+\infty
\end{split}
\end{equation}
Hence, we know $\{u_n\} \in L^1_{loc}(X)$ and for any relative compact sunset $K \subset X$, $\|u_n\|_{L^1(K)}$ in uniformly bounded. By Fatou lemma, we have
$$\int_{z\in K}|u|d\lambda_z\le \liminf\limits_{n \to +\infty} \int_{z\in K}|u_n|d\lambda_z<+\infty.$$
This means $u \in L^1_{loc}(X)$.
\par
Now we consider $i\partial\bar{\partial}u(z)$. Let $g\in C^{\infty}_c(X)$ be a test function. We have
\begin{equation}
\begin{split}
\langle i\partial\bar{\partial}u,g\rangle
&=\langle u(z),i\partial\bar{\partial}g(z)\rangle\\
&=\lim_{n\to+\infty}\langle u_n(z),i\partial\bar{\partial}g(z)\rangle\\
&=\lim_{n\to+\infty}\langle \langle T_n(w),2G_X(z,w)\rangle,i\partial\bar{\partial}g(z)\rangle\\
&=\lim_{n\to+\infty}\langle T_n(w),\langle 2G_X(z,w),i\partial\bar{\partial}g(z)\rangle\rangle\\
&=\lim_{n\to+\infty}\langle T_n(w),g(w)\rangle\\
&=\langle T,g\rangle\\
\label{4.42}
\end{split}
\end{equation}
\par
The forth equality holds because of Fubini Theorem.
Now we explain the second equality. Given a point $q\in X$, under the local coordinate $(U_q,z_q)$, we have $i\partial\bar{\partial}g(z)=
i f(z)dz\wedge d \bar{z}$, where $f(z)$ is a smooth real function on $U_q$ with compact support. Let $(i\partial\bar{\partial}g(z))_{+}= f(z)_+ idz\wedge d \bar{z}$, $(i\partial\bar{\partial}g(z))_{-}= f(z)_{-} idz\wedge d \bar{z}$, where $f(z)_+=\max(f(z),0)$ and $f(z)_-=\max(-f(z),0)$. Then it follows from Levi's Theorem that we have
$$\lim\limits_{n \to +\infty} \int_{U_q}u_n(z)(i\partial\bar{\partial}g(z))_{+}= \int_{U_q} u(z)(i\partial\bar{\partial}g(z))_{+}$$
and
$$\lim\limits_{n \to +\infty} \int_{U_q}u_n(z)(i\partial\bar{\partial}g(z))_{-}= \int_{U_q} u(z)(i\partial\bar{\partial}g(z))_{-}$$
Since $i\partial\bar{\partial}g(z)=(i\partial\bar{\partial}g(z))_{+}-(i\partial\bar{\partial}g(z))_{-}$
, hence we have
\begin{equation}
\lim\limits_{n \to +\infty} \int_{U_q}u_n(z)i\partial\bar{\partial}g(z)= \int_{U_q} u(z)i\partial\bar{\partial}g(z)
\label{4.43}
\end{equation}
As $g(z)$ has compact support, there exists finite $\{U_{q_i}\}$ such that $suppg\subset \mathop{\cup}\limits_{i} U_{q_i}$ and on each $U_{q_i}$, equality \eqref{4.43} holds. Hence we know that on the whole $X$, we have
$$\lim\limits_{n \to +\infty} \int_X u_n(z)i\partial\bar{\partial}g(z)= \int_X u(z)i\partial\bar{\partial}g(z)$$
which implies
$$\langle u(z),i\partial\bar{\partial}g(z)\rangle
=\lim_{n\to+\infty}\langle u_n(z),i\partial\bar{\partial}g(z)\rangle$$
i.e. the second equality holds. Then it follows from \eqref{4.42} that we know $i\partial\bar{\partial}u=T=\theta i\partial\bar{\partial}\varphi$.
\par
For fixed $t>0$, as $k\ge 2$, the set $\{z:-t<kG_X(z,z_0)<0\}\subset \{z:-t<2G_X(z,z_0)<0\}$.
Let $t>0$ be small enough such that the set $\{z:-t<2G_X(z,z_0)<0\}\cap(\overline{U}\cup\{z_0\})=\emptyset$. Let $W\subset\subset X$ be an relatively compact open set of $X$ which satisfies $\overline{U}\cup\{z_0\} \subset W$ and $W\cap \{-t<2G_X(z,z_0)<0\}=\emptyset$.
\par
Then for every fixed $z\in\{-t<kG_X(z,z_0)<0\}$, $2G_X(z,w)$ is harmonic function on $W$ with respect to $w$.
\par
By the Harnack inequality of harmonic function, there exists a $M>0$ such that
$$\sup\limits_{w\in \overline{W}}(-2G_X(z,w))\le M\inf\limits_{w\in \overline{W}}(-2G_X(z,w))$$
As $0<-2G_X(z,z_0)<t$, we have
$$Mt > -2G_X(z,z_0)\ge M\inf\limits_{w\in \overline{W}}(-2G_X(z,w))\ge\sup\limits_{w\in \overline{W}}(-2G_X(z,w))\ge 0$$
This means when $t\to 0$, the function $2G_X(z,w)$ which defined on $\{z:-t<kG_X(z,z_0)<0\}\times U$ uniformly goes to $0$.
\par
Note that when $(z,w)\in \{z:-t<kG_X(z,z_0)<0\}\times U$ ($t$ big enough), $2G_X(z,w)$ is harmonic function. Then
$$u(z)=\lim\limits_{n \to +\infty}u_n(z)=\lim\limits_{n \to +\infty}\langle  T(w),2G_X(z,w)\ast \rho_n\rangle=\langle T(w),2G_X(z,w)\rangle.$$
The third equality holds because of the mean-value equality for harmonic function. Hence when $z$ satisfies $kG_X(z,z_0)\to 0$, we have $u(z)\to 0$.
\par
Now let $\tilde{\varphi}=\varphi-u$, we know $\tilde{\varphi}>\varphi$ and when $z$ satisfies $kG_X(z,z_0)\to 0$, we have $\tilde{\varphi}(z)\to \varphi(z)$. Note that $i\partial\bar{\partial}\tilde{\varphi}=i(1-\theta)\partial\bar{\partial}\varphi\ge 0$ on $X$. Hence $\tilde{\varphi} \in PSH(X)$.
\par
Note that $\theta$ is a smooth function and $0\le \theta \le 1-\frac{[\lambda_j]}{\lambda_j}$, we have $$[\lambda_{j_0}]\le v(i\partial\bar{\partial}\tilde{\varphi},x_{j_0})<\lambda_{j_0}.$$
\par
For any $x$ satisfies $0\le v(i\partial\bar{\partial}\varphi,x)<1$, we have
$$0\le v(i\partial\bar{\partial}\tilde{\varphi},x)\le v(i\partial\bar{\partial}\varphi,x)<1.$$
\par
For any $x$ satisfies $v(i\partial\bar{\partial}\varphi,x)\ge 1$, as $U \cap (E_1(i\partial\bar{\partial}\varphi)\backslash \{x_{j_0}\} )=\emptyset$ and
$supp \theta\subset \subset U$, we have
$$v(i\partial\bar{\partial}\tilde{\varphi},x)= v(i\partial\bar{\partial}\varphi,x).$$
Hence by the classification of multiplier ideal sheaves in dimensional one case, we know for any $x \in X$, we have $\mathcal{I}(\tilde{\varphi})_x=\mathcal{I}(\varphi)_x$.
\par
Next we prove $u(z)$ has lower bound $-A$ for some $A>0$ on $X\backslash U$. It follows from Lemma \ref{lemma4.7} that when $z\in X\backslash U$ and $w\in V$, we have $2G_X(z,w)\ge 2NG_X(z,x_{j_0})$. By Lemma \ref{lemma4.6}, we know $G_X(z,x_{j_0})\ge -A_0$ (where $A_0>0$ is a constant) for $z\in X\backslash \bar{U}$.
\par
Note that $G_{X}(z,w)$ is harmonic function on $(z,w)\in (X\backslash U)\times V$. For fixed $z\in X \backslash U$, we have
\begin{equation}\nonumber
\begin{split}
u(z)=&\langle T(w),2G_X(z,w)\rangle\\
=&\langle \theta(w)i\partial\bar{\partial}\varphi(w),2G_X(z,w)\rangle\\
\ge&\langle \theta(w)i\partial\bar{\partial}\varphi(w),-2NA_0\rangle,
\end{split}
\end{equation}
here $-2NA_0$ is actually a constant function $f(w)\equiv -2NA_0$ defined on $V$. Note that $supp \theta \subset \subset V$, hence the inequality $``\ge"$ holds. Let $A=\langle \theta(w)i\partial\bar{\partial}\varphi(w),2NA_0\rangle$, we know on $X\backslash U$, $u(z)>-A$.
\par
As $\varphi-\tilde{\varphi}=u(z)$, we know that there exists a relatively compact open subset $U\subset \subset X$ such that $\varphi-\tilde{\varphi}$ has lower bound $-A$ ($A>0$ is a constant) for any $z \in X\backslash U$.
\par
Then in case (1), we have a function $\tilde{\varphi}$ satisfies the conditions in the Lemma \ref{lemma for cor1.9}.
\par
For the case (2):
\par
As $R\neq 0$, there must be a point $p\in suppR\backslash E_1(i\partial\bar{\partial}\varphi)$. If not, we must have $supp R \subset E_1(i\partial\bar{\partial}\varphi)$. As $R$ is a closed positive $(1,1)$ current, $R$ is of order 0.
 Note that $E_1(i\partial\bar{\partial}\varphi)$ is an analytic subset of $X$ with irreducible components $\{x_j\}_{j=1,2,\ldots}$, then it follows from Lemma \ref{lemma4.5} that $R=\sum_{j\ge 1} a_j [x_j]$, where $a_j=v(R,x_j)$ is the Lelong number of $R$ at $x_j$. However by Siu's decomposition theorem, we know $v(R,x)=0$, for any $x \in X$, which implies that all $a_j=0$ and then $R=0$. This contradicts to the fact that $R\neq 0$.
\par
Let $p \in Supp R \backslash E_1(i\partial\bar{\partial}\varphi)$. Let $(U_2,z)$ be a relative compact coordinate neighborhood of $p$ in $X$ and by shrinking $U$, we assume that under the local homomorphism, $z(p)=o$ and $z(U_2)\cong B(0,2)$. We also assume that $U_2 \cap (E_1(i\partial\bar{\partial}\varphi) )=\emptyset$. Let $\theta_2$ be a smooth cut-off function on $X$ such that $0\le \theta \le 1$, $supp(\theta_2)\subset\subset V_2 \subset\subset U_2$ and $\theta \equiv 1$ on $W_2$, where $z(W_2)=B(0,\frac{1}{4})$ and $z(V_2)=B(0,\frac{1}{2})$  under the local homomorphism. By shrinking $V_2$ and $U_2$ again, it follows from Lemma \ref{lemma4.7} that there exists $N_2>0$ such that for
any $(z,w) \in (X\backslash U_2)\times V_2$, we have
$$G_X(z,w)\ge N_2G_X(z,p).$$
\par
Let $T_2=\theta_2\cdot i\partial\bar{\partial}\varphi$. We can do the same thing as we did in the case (1) and get a function $u_2(z)$ such that $u_2(z)<0$, $i\partial\bar{\partial}u_2(z)=\theta_2\cdot i\partial\bar{\partial}\varphi$ and when $z$ satisfies $kG_X(z,z_0)\to 0$, we have $u_2(z)\to 0$. Especially, $u_2(z)$ has lower bound $-A_2$ for some $A_2>0$ on $X\backslash U_2$.

\par
Let $\tilde{\varphi}_2(z)=\varphi-u_2(z)$. We know $\tilde{\varphi}_2>\varphi$. When $z$ satisfies $kG_X(z,z_0)\to 0$, we have $\tilde{\varphi}_2(z)\to \varphi(z)$.
As $\varphi_2(z)-\tilde{\varphi}_2(z)=u_2(z)$, we know that there exists a relatively compact open subset $U_2\subset \subset X$ such that $\varphi_2-\tilde{\varphi}_2$ has lower bound $-A_2$ ($A_2>0$ is a constant) for any $z \in X\backslash U_2$.
Note that $i\partial\bar{\partial}(\tilde{\varphi}_2)=i(1-\theta_2)\partial\bar{\partial}\varphi\ge 0$ on $X$. Hence $\tilde{\varphi}_2 \in PSH(X)$.

\par
Note that $\theta_2$ is a compact smooth function with $supp(\theta_2)\subset\subset U_2$ and $0\le \theta_2 \le 1$. It is easy to see that for any $x$ satisfies $0\le v(i\partial\bar{\partial}\varphi,x)<1$, we have
$$0\le v(i\partial\bar{\partial}\tilde{\varphi},x)\le v(i\partial\bar{\partial}\varphi,x)<1.$$
\par
For any $x$ satisfies $v(i\partial\bar{\partial}\varphi,x)\ge 1$, as $U_2 \cap (E_1(i\partial\bar{\partial}\varphi))=\emptyset$ and $supp(\theta_2)\subset \subset U_2$, we have
$$v(i\partial\bar{\partial}\tilde{\varphi},x)= v(i\partial\bar{\partial}\varphi,x).$$
\par
Hence by the classification of multiplier ideal sheaves in dimensional one case, we know for any $x \in X$, we have $\mathcal{I}(\tilde{\varphi}_2)_x=\mathcal{I}(\varphi)_x$.
\par
Then in case (2), we have a function $\tilde{\varphi}_2$ satisfies the conditions in the Lemma \ref{lemma for cor1.9}.

\par
Lemma \ref{lemma for cor1.9} is proved.
\end{proof}

\subsection {Appendix: a property of multiplicative function on Open Riemann surface.}
\label{Appendix4.3}
Let $X$ be a open Riemann surface. We recall the following construction in Section 1.
\par
Let $p:\Delta\to X$ be the universal covering from unit disc $\Delta$ to $X$.
We call the holomorphic function $f$ (resp. holomorphic $(1,0)$ form $F$) on $\Delta$ is a multiplicative function (resp. multiplicative differential (Prym differential)) if there is a character $\chi$, where $\chi\in Hom(\pi_1(X),C^*)$ and $|\chi|=1$, such that $g^*f=\chi(g)f$ (resp. $g^*F=\chi(g)F$) for every $g\in \pi_1(X)$ which naturally acts on the universal covering of $X$. Denote the set of such kinds of $f$ (resp. F) by $\mathcal{O}^{\chi}(X)$ (resp. $\Gamma^{\chi}(X)$).
\par
As $p$ is a universal covering, then for any harmonic function $h$ on $X$, there exists a $\chi_h$ and a multiplicative function $f_h \in \mathcal{O}^{\chi_h}(X)$, such that $|f_h|=p^*e^{h}$.
Let $s$ be a holomorphic function on $X$ and $s$ has no zero points on $X$. We know $\log|s|$ is a harmonic function on $X$.
\par
In this appendix, we recall the following well-known property.
\begin{Lemma}\label{Aproposition4.1}
$\chi_h=\chi_{h+\log|s|}$.
\end{Lemma}
\begin{proof} We firstly recall the construction of $f_h$ and $\chi_h$.
\par
As $h$ is harmonic on $X$, then $p^{*}h$ is harmonic on $\Delta$. Since $\Delta$ is simple connected, there exists $f\in \mathcal{O}_{\Delta}$ such that $f=p^*h+i v$. Then $e^f$ is holomorphic on $\Delta$ and $|e^f|=|e^{p^* h +iv}|=p^*e^h$. We denote $f_h=e^f$.
\par
Let $\Gamma$ be a subgroup of $Aut(\Delta)$ such that $X=\Delta/\Gamma$. Then by the theorem of covering spaces, we know $\pi_1(X)\cong \Gamma$. Hence for any $g\in \pi_1(X)$, $g$ naturally acts on $\Delta$ and for any $z\in X$, $p^{-1}(z)$ is invariant under the act of $g$.
Fix $z_1\in X$, we denote $p^{-1}(z_1)=\{x_0,x_1,\cdots\}$. By the theorem of covering spaces, we know there is a bijection between $p^{-1}(z_1)$ and $\pi_1(X,z_1)$.
\par
For any $g_i\in\pi_1(X)$, we assume that $g(x_0)=x_i$. Then we define
\begin{equation}\chi_h(g_i)=\frac{e^f(x_i)}{e^f(x_0)}
=\frac{e^{p^*u(x_i)+iv(x_i)}}{e^{p^*u(x_0)+iv(x_0)}}
=\frac{e^{u(z_1)+iv(x_i)}}{e^{u(z_1)+iv(x_0)}}
=\frac{e^{iv(x_i)}}{e^{iv(x_0)}}.
 \label{A4.44}
\end{equation}
Hence $|\chi_h(g_i)|=1$. Now we prove that $e^f\in\mathcal{O}^{\chi_{h}}(X)$, i.e., for any $g_i\in \pi_1(X)$, we have $g_i^* e^f=\chi(g_i)e^f$.
\par
Given $y_0\in \Delta$, denote $w_1=p(y_0)\in X$. Denote $p^{-1}(w_1)=\{y_0,y_1,y_2,\cdots\}$. We know there exists a bijection between $\pi_1(X,w_1)$ and $p^{-1}(w_1)$.
\par
As the fundamental group $\pi_1(X)$ is base point free, we have $\pi_1(X,w_1)\cong\pi_1(X,z_1)\cong \Gamma$. Hence we have a bijection
$$\Phi:p^{-1}(z_1)\to p^{-1}(w_1)$$
$$ \ \ \ \ x_i\to y_i$$
which satisfies $g\circ \Phi=\Phi\circ g$.
\par
For any $g\in \pi_1(X)$, we assume that $g(y_0)=y_i$, then $g(x_0)=x_i$.
Then we have
\begin{equation}
  g^*e^{f(y_0)}=e^{f\circ g(y_0)}=e^{f(y_i)}.
  \label{A4.45}
\end{equation}
and by the definition of $\chi_h$ (see formula \eqref{A4.44}),
\begin{equation}
  \chi_h(g)e^{f(y_0)}=\frac{e^{f(x_i)}}{e^{f(x_0)}}e^{f(y_i)}.
  \label{A4.46}
\end{equation}
To prove \eqref{A4.45} equals to \eqref{A4.46}, as $f=p^u +iv$, it sufficient to prove that
\begin{equation}
  v(x_i)-v(x_0)=v(y_i)-v(y_0).
  \label{A4.47}
\end{equation}
As $p^* u$ is harmonic on $\Delta$, let $w=\frac{\partial p^* u}{\partial x}dy-
\frac{\partial p^* u}{\partial y}dx$ on $\Delta$, then $dw=0$. Let $\tilde{L}_{qp}$ be any path from $q$ to $p$. We can define $v(x)$ as below.
$$v(p)=\int_{\tilde{L}_{0p}}w.$$
where $0$ is the origin in $\Delta$.
\par
Then we have $v(x_i)-v(x_0)=\int_{\tilde{L}_{x_0x_i}}w$ and $v(y_i)-v(y_0)=\int_{\tilde{L}_{y_0y_i}}w$. Denote $L_{z_1}=p_* \tilde{L}_{x_0x_i}$ and $L_{w_1}=
p_* \tilde{L}_{y_0y_i}$. Then by the isomorphic between $\pi_1(X,z_1)\cong\pi_1(X,w_1)\cong\Gamma$ and $g\circ \Phi=\Phi\circ g$, we know the path $L_{z_1}$ and $L_{w_1}$ are homotopic. Hence
\begin{equation}
\begin{split}
 v(x_i)-v(x_0)&=\int_{\tilde{L}_{x_0x_i}}w\\
 &=\int_{p_*\tilde{L}_{x_0x_i}}p_*w\\
 &=\int_{L_{Z_1}}p_*w\\
 &=\int_{L_{W_1}}p_*w\\
 &=\int_{p_*\tilde{L}_{y_0y_i}}p_*w\\
 &=\int_{\tilde{L}_{y_0y_i}}w\\
 &=v(y_i)-v(y_0)
\end{split}
\end{equation}
Hence we know given any $y_0\in\Delta$, for any $g\in \pi_1(X)$, we have
$$\chi_h(y_0)e^f=g^*(e^f).$$
We also need to show that $\chi$ is a homomorphism from $\pi_1(X)$ to $ \mathbb{C}^*$.
\par
For any $g_1,g_2 \in \pi_1(X)$. We assume that $g_1(x_0)=x_1$ and $g_2(x_1)=x_2$, then we have
$$\chi(g_1)\chi(g_2)
=\frac{e^{iv(x_1)}}{e^{iv(x_0)}}
\frac{e^{iv(x_2)}}{e^{iv(x_1)}}
=\frac{e^{iv(x_2)}}{e^{iv(x_0)}}.$$
and note that $g_1\circ g_2(x_0)=x_2$,
$$\chi(g_1g_2)=\frac{e^{iv(x_2)}}{e^{iv(x_0)}}.$$
Hence $\chi(g_1)\chi(g_2)=\chi(g_1g_2)$, $\chi$ is a homomorphism from $\pi_1(X) $ to $\mathbb{C}^*$.
\par
Now we can prove $\chi_h=\chi_{h+\log|s|}$. Denote $h_2=h+\log|s|$.
\par
Note that we have already found a holomorphic function $f_h$ on $\Delta$ which satisfies $|f_h|=p^*e^h$. Then for any $g\in \pi_1(X,z_1)$, we define $\chi_h(g)=\frac{f_h(x_i)}{f_h(x_0)}$, where $x_i,x_0\in p^{-1}(z_1)$ and $g(x_0)=x_i$.
\par
It is easy to see that $f_hp^*s$ is a holomorphic function on $\Delta$ which satisfies $|f_h\cdot p^*s|=p^*e^h|p^*s|=p^*e^{h+\log |s|}$. Then similarly as above, for any $g\in \pi_1(X,z_1)$, we define $$\chi_{h_2}(g)=\frac{f_h(x_i)\cdot p^*s(x_i)}{f_h(x_i)\cdot p^*s(x_0)}.$$ Note that $p^*s$ is fiber-constant, we know $p^*s(x_1)=p^*s(x_0)$ for any $x_i,x_0\in p^{-1}(z_1)$. Hence
$$\chi_{h_2}(g)=\frac{f_h(x_i)\cdot p^*s(x_i)}{f_h(x_i)\cdot p^*s(x_0)}=\frac{f_h(x_i)}{f_h(x_0)}=\chi_h(g).$$
As $g\in \pi_1(X,z_1)$ is arbitrary chosen, we know $\chi_h=\chi_{h+\log|s|}$.
\par
Lemma \ref{Aproposition4.1} is proved.
\end{proof}


\vspace{.1in} {\em Acknowledgements}.
The authors would like to thank Professor Wlodzimierz Zwonek for his comment on our paper.
\par
The first author was supported by NSFC-11825101, NSFC-11522101 and NSFC-11431013.

\bibliographystyle{references}
\bibliography{xbib}

\end{document}